\documentclass{amsart}

\usepackage[english]{babel}
\usepackage[utf8]{inputenc}
\usepackage[margin=3cm]{geometry}
\usepackage{hyperref}
\usepackage{xcolor}
\hypersetup{
	colorlinks=true,
	linkcolor=blue,  
	urlcolor=cyan,
	pdftitle={},
	pdfauthor={},
}
\usepackage{graphicx}
\usepackage{amssymb}
\usepackage{amsthm, amsfonts,stmaryrd}
\usepackage{mathtools}
\usepackage{thmtools}
\usepackage[shortlabels]{enumitem}
\usepackage{cleveref}
\usepackage{upgreek}
\usepackage{subfig}

\theoremstyle{plain}
\newtheorem{theorem}{Theorem}[section]
\newtheorem*{theorem*}{Theorem}
\newtheorem{lemma}[theorem]{Lemma}
\newtheorem{assumption}[theorem]{Assumption}
\newtheorem{corollary}[theorem]{Corollary}
\newtheorem{proposition}[theorem]{Proposition}

\theoremstyle{definition}
\newtheorem{definition}[theorem]{Definition}
\newtheorem*{definition*}{Definition}
\newtheorem{example}[theorem]{Example}
\newtheorem{remark}[theorem]{Remark}

\numberwithin{equation}{section}

\usepackage{array,multirow,makecell}
\newcommand{\PreserveBackslash}[1]{\let\temp=\\#1\let\\=\temp}
\newcolumntype{R}[1]{>{\raggedleft\arraybackslash }b{#1}}
\newcolumntype{L}[1]{>{\raggedright\arraybackslash }b{#1}}
\newcolumntype{C}[1]{>{\centering\arraybackslash }m{#1}}

\newcommand{\eps}{\varepsilon}

\def\Ac{\mathcal{A}}
\def\Pc{\mathcal{P}}
\def\Fc{\mathcal{F}}
\def\Tc{\mathcal{T}}

\def\Cc{\mathcal{C}}
\def\Uc{\mathcal{U}}
\def\Xc{\mathcal{X}}
\def\Yc{\mathcal{Y}}

\def\R{\mathbb{R}}
\def\Rd{\R^n}
\def\Rn{\R^n}
\def\N{\mathbb{N}}
\def\rmx{\mathrm{x}}
\def\rmy{\mathrm{y}}
\newcommand{\rmz}{{\mathrm z}}
\DeclareMathOperator{\Id}{Id}

\def\BW{\operatorname{BW}}
\def\supp{\text{\normalfont supp}}
\def\dom{\text{\normalfont dom}}
\def\p{\text{\normalfont p}}

\newcommand{\calP}{\mathcal{P}}
\newcommand{\calR}{\mathcal{R}}

\newcommand{\Sc}{\mathfrak{S}_c}

\newcommand{\ov}[1]{\overline{#1}}

\newcommand{\xb}{{\bar x}}
\newcommand{\yb}{{\bar y}}
\newcommand{\zb}{{\bar z}}
\newcommand{\wb}{{\bar w}}

\newcommand{\nub}{{\bar\nu}}
\newcommand{\cu}{{\underline c}}
\newcommand{\xu}{{\underline x}}
\newcommand{\yu}{{\underline y}}
\newcommand{\Xu}{{\underline X}}
\newcommand{\Yu}{{\underline Y}}
\newcommand{\abs}[1]{\lvert #1 \rvert}

\usepackage[textwidth=2cm, textsize=footnotesize,colorinlistoftodos]{todonotes}
\makeatletter
\providecommand\@dotsep{5}
\renewcommand{\listoftodos}[1][\@todonotes@todolistname]{%
  \@starttoc{tdo}{#1}}
\makeatother

\def\smtw{(LMP)}
\def\snccconv{(NNCC-conv)}

\DeclareMathOperator{\KL}{KL}
\newcommand{\ib}{{\bar\imath}}
\newcommand{\jb}{{\bar\jmath}}

\newcommand{\lb}{{\bar\ell}}
\newcommand{\mb}{{\bar m}}

\DeclareMathOperator{\GW}{GW}
\DeclareMathOperator{\GH}{GH}

\newcommand{\Gauged}{\mathbb{G}}
\newcommand{\bX}{{\mathbf X}}
\newcommand{\bY}{{\mathbf Y}}
\newcommand{\bZ}{{\mathbf Z}}
\DeclareMathOperator{\hypo}{hypo}
\DeclareMathOperator{\contact}{contact}
\newcommand{\bracket}[1]{\langle #1 \rangle}
\newcommand{\norm}[1]{\lVert #1 \rVert}

\title[Nonnegative cross-curvature in infinite dimensions]{
Nonnegative cross-curvature in infinite dimensions: \\
Synthetic definition and spaces of measures}

\author[F. Léger]{Flavien Léger}
\address{Flavien Léger, Inria and CEREMADE, Université Paris Dauphine, Paris, France} 
\email{flavien.leger@inria.fr}

\author[G. Todeschi]{Gabriele Todeschi}
\address{Gabriele Todeschi, LIGM, Univ. Gustave Eiffel, CNRS, F-77454 Marne-la-Vallée, France
} 
\email{gabriele.todeschi@univ-eiffel.fr}

\author[F.-X. Vialard]{François-Xavier Vialard}
\address{François-Xavier Vialard, LIGM, Univ. Gustave Eiffel, CNRS, F-77454 Marne-la-Vallée, France} 
\email{francois-xavier.vialard@univ-eiffel.fr}

\date{\today}

\begin{document}

\begin{abstract}
    Nonnegative cross-curvature (NNCC) is a geometric property of a cost function defined on a product space originating in optimal transportation and the Ma--Trudinger--Wang theory. 
    Motivated by applications in optimization, gradient flows and mechanism design,
    we propose a variational formulation of nonnegative cross-curvature on c-convex domains
    applicable to infinite dimensions and nonsmooth settings.
    The resulting class of NNCC spaces is closed under Gromov--Hausdorff convergence and for this class, we extend many properties of classical nonnegative cross-curvature: 
    stability under generalized Riemannian submersions, 
    characterization in terms of the convexity of certain sets of c-concave functions, 
    and in the metric case, it is a subclass of positively curved spaces in the sense of Alexandrov.     
    One of our main results is that Wasserstein spaces of probability measures inherit the NNCC property from their base space. 
    Additional examples of NNCC costs include the Bures--Wasserstein and Fisher--Rao squared distances, the Hellinger--Kantorovich squared distance (in some cases), the relative entropy on probability measures, and the $2$-Gromov--Wasserstein squared distance on metric measure spaces. 
\end{abstract}

\maketitle

\setcounter{tocdepth}{2}
\tableofcontents

\section{Introduction}
The MTW condition was introduced by Ma, Trudinger, and Wang in their study of the regularity of the optimal transport problem with a general cost function \cite{ma2005regularity}. 
They formulated this condition as the positivity of a fourth-order tensor, now called the MTW tensor, on orthogonal directions. Kim and McCann subsequently studied a strengthening of the MTW condition \cite{kim2010continuity,kim2012towards} requiring nonnegativity of the MTW tensor in every direction (as opposed to orthogonal directions only). They called this condition \emph{nonnegative cross-curvature}, cross-curvature being the name they gave to the MTW tensor.
The main goal of this paper is to give a notion of nonnegative cross-curvature for a product of two arbitrary sets endowed with an arbitrary cost function. 

Interestingly, the MTW tensor and nonnegative cross-curvature appeared in other contexts. In \cite{Figalli_Kim_McCann_screening2011} Figalli, Kim and McCann 
studied certain calculus of variations
problems coming from economics where the optimization is constrained to functions satisfying a generalized convexity condition known as c-concavity. In their setting, they proved that nonnegative cross-curvature guarantees the convexity of the set of all c-concave functions and the convexity of their objective function, with both theoretical and computational repercussions. 
In \cite{leger-vialard2023} the first and third authors obtained geometric formulas for the asymptotics of Laplace-type integrals that concentrate on the graph of a map. In particular, their formulas involve scalar contractions of the MTW tensor. 
In \cite{matthes2019variational} Matthes and Plazotta introduced a second-order time discretization of gradient flows in metric spaces which is well-posed under a semi-convexity assumption satisfied by Riemannian manifolds whose squared distance has nonnegative cross-curvature. 
Finally in a recent preprint \cite{leger2023gradient} the first author and Aubin-Frankowski established a framework for doing explicit and implicit first-order optimization schemes using an arbitrary minimizing movement cost function. They proved convergence rates which were shown to be tractable when the cost function has nonnegative cross-curvature.

The applications of nonnegative cross-curvature mentioned above are severely constrained by the available theory, which requires a certain amount of regularity on the objects at play. For example, the cost function $c(x,y)$ typically needs to be four times differentiable and defined on domains $X$ and $Y$ which are manifolds with the same finite dimension. 
In the principal--agent problems studied by Figalli, Kim and McCann, the cost $c(x,y)$ may represent the disutility of an agent $x\in X$ to purchase a good $y\in Y$. Since the sets of agents $X$ and products $Y$ may have nothing in common and may either be discrete or continuous, it is desirable to remove any differentiable restriction on the cost function and have a theory that does not force these two spaces to be modeled by manifolds.
The optimization scheme proposed by Matthes and Plazotta is defined on metric spaces $(X,d)$ and uses a squared distance cost $c(x,y)=d^2(x,y)$ with $X=Y$. Their main interest is infinite-dimensional but when using nonnegative cross-curvature their analysis confines them to finite-dimensional Riemannian manifolds. A similar situation arises in the framework of Léger and Aubin-Frankowski where the use of nonnegative cross-curvature is limited to finite-dimensional smooth manifolds. This highlights the value of a condition that extends to the infinite-dimensional setting.

\subsection{The MTW condition}

There has been more focus in the literature on the closely related MTW condition than on the nonnegative cross-curvature condition, due to its direct connection to optimal transport. 
The solution of an optimal transport problem is a measure that concentrates on the graph of a map under specific assumptions.
Early results on the regularity of this map date back to the pioneering work of Caffarelli \cite{caffarelli1992regularity,caffarelli1996boundary}, Delanoë \cite{Delanoe1991} and Urbas \cite{Urbas1997}. 
Later on, the MTW condition was identified by Ma, Trudinger and Wang \cite{ma2005regularity} as a key ingredient to ensure regularity of the transport map for the problem with a general cost function $c(x,y)$. This followed contributions for particular cost functions \cite{GlimmOliker,wang2004design,GangboMcCann_shape2000}. Loeper \cite{loeper2009regularity} then showed that the MTW condition was necessary for regularity.
Villani and Loeper \cite{loeper2010nonfocal,villani2009optimal} expanded this approach and introduced more general $\operatorname{MTW}(K,C)$ conditions, to distinguish between the different components of the MTW tensor.
Following these advancements, several works focused on finding domains and cost functions satisfying these regularity requirements \cite{figalli2008approximation,figalli2009continuity,loeper2011sphere,FigalliRiffordVillani_nearly_round_2012,Kim_counterexamples_2008,Lee_Li_new_examples_2012,Lee_McCann_mechanical_action2011,loeper2010nonfocal}. Let us note there also exist more applied situations where the regularity of the transport map is of interest. These include the reflector problem \cite{wang2004design}, stability and statistical estimation of optimal transport \cite{gallouet2022strong,vacher2021dimension,muzellec2021near}, as well as numerical methods for solving transport problems \cite{JacobsLeger2020,JacobsLeeLeger2021}.

While the original setting to formulate the MTW condition requires the cost function $c(x,y)$ to be four times differentiable, there have been several contributions that reduce this requirement, depending on the desired application. Loeper \cite{loeper2009regularity} obtained a synthetic formulation of the MTW condition based on a certain maximum principle, see also \cite[Section 7.5]{ma2005regularity}, \cite[Section 2.5]{TrudingerWang_rational_mech_2009}. When $c(x,y)$ is the square of a Riemannian distance, Villani \cite{Villani_4th_order} expressed the MTW condition in terms of distances and angles and showed its stability under Gromov--Hausdorff convergence. 
Guillen and Kitagawa \cite{GuillenKitagawa2014} formulated a quantitative version of Loeper's maximum principle using a $C^1$ cost $c$. They then obtained the regularity of the optimal transport map assuming only that $c\in C^3$. 
More recently using a cost function $c\in C^2,$ Loeper and Trudinger \cite{LoeperTrudinger2021} formulated a local weakening of Loeper's maximum principle that is equivalent to the MTW condition when $c\in C^4$. Their interest was showing that under their condition a locally c-concave function is in fact c-concave globally. 
Finally, Rankin \cite{Rankin_A3w_remark_2023} obtained an equivalent form of the MTW condition when $c\in C^2$. His approach can be seen as performing a Taylor expansion of the cross-difference \cite{McCann_glimpse2014} in orthogonal directions to recover the MTW condition.

\subsection{The Kim--McCann geometry}\label{ssec:KM}

A contribution of Kim and McCann fully revealed the geometric nature of the MTW tensor. 
In \cite{kim2010continuity} they introduced a pseudo-Riemannian structure on a product manifold $X\times Y$ equipped with a nondegenerate cost function $c\in C^4(X\times Y)$ (see \Cref{sec:sncc-def}). Their metric can be written as $g_{\operatorname{KM}}=-\nabla^2_{xy}c$.
In this geometry the curves known in optimal transport as c-segments (\Cref{def:c-segment}) are geodesics for which the second variable is kept fixed. Their main result is that the MTW tensor $\mathfrak{S}_c$ (\Cref{def:MTW-tensor}) can be expressed through the Riemann curvature tensor of $g_{\operatorname{KM}}$, as
\[
\mathfrak{S}_c(x,y)(\xi,\eta) = 2\,R_{\operatorname{KM}}(U,V,U,V)\,.
\]
Here $\xi$ and $\eta$ are tangent vectors at $x$ and $y$ respectively and $U,V$ are tangent vectors at $(x,y)$ in the tangent bundle $T(X\times Y)$ defined by $U=(\xi,0)$ and $V=(0,\eta)$. 
The nonnegative cross-curvature condition: $\mathfrak{S}_c(x,y)(\xi,\eta)\geq 0$ for all $\xi,\eta,$ and the MTW condition: $\mathfrak{S}_c(x,y)(\xi,\eta)\geq 0$ when $\nabla^2_{xy}c(x,y)(\xi,\eta)=0$ can therefore be connected to the nonnegativity of $R_{\operatorname{KM}}$ along certain directions.

\subsection{Stability by products}
While nonnegative cross-curvature is a more stringent condition than the MTW condition, it comes with an important benefit: stability by taking products, as shown by Kim and McCann \cite{kim2012towards}. 
Namely, if two product manifolds endowed with cost functions are nonnegatively cross-curved, then so is their product endowed with the natural sum of costs. 
In contrast, this is not the case for the MTW condition.
This stability by products can be observed directly since nonnegative cross-curvature corresponds to the nonnegativity of the sectional curvature of the Kim--McCann pseudo-Riemannian metric in certain directions (see Section \ref{ssec:KM}). 
Since the sectional curvature of a product is the sum of the sectional curvatures, the property follows since the expressions $U = (\xi,0)$ and $V = (0,\eta)$ as in the previous section are also stable by taking products. Similar arguments do not hold for the MTW condition.\footnote{What prevents the MTW condition from passing to products is precisely the fact that there are more vectors $(U_1, U_2, V_1, V_2)$ such that $g^1_{\operatorname{KM}}(x_1,y_1)(U_1, V_1) + g^2_{\operatorname{KM}}(x_2,y_2)(U_2, V_2) = 0$ than the product of sets of vectors satisfying $g^1_{\operatorname{KM}}(x_1,y_1)(U_1, V_1) = 0$ and $g^2_{\operatorname{KM}}(x_2,y_2)(U_2, V_2) = 0$. Here, we used two couples of spaces $X_1 \times Y_1$ and $X_2 \times Y_2$ and their corresponding metrics.}

This stability by taking products is key for our extension of nonnegative cross-curvature to infinite-dimensional cases. Let us informally discuss the Wasserstein cost on the space of measures.
Nonnegative cross-curvature is also stable by (smooth and finite-dimensional) Riemannian submersions, as shown by Kim and McCann \cite{kim2010continuity,kim2012towards}. 
Since the Wasserstein space can be considered as the quotient of an infinite product, it motivates one of our main results: if the underlying cost is nonnegatively cross-curved, so is the Wasserstein cost on the space of measures.

\subsection{Description of our main results}

In a smooth setting, a formula of Kim and McCann reformulates nonnegative cross-curvature as the convexity of a difference of cost functions along specific curves known as c-segments (\Cref{thm:KM_characterization}). We observe that their condition characterizes not only nonnegative cross-curvature but also c-segments (\Cref{lemma:auto-csegment}), which frees us from any differentiable prerequisite needed to define c-segments. Our notion of a space with nonnegative cross-curvature (NNCC space) follows naturally.

\begin{definition}
    Let $X$ and $Y$ be two arbitrary sets and $c\colon X\times Y\to [-\infty,+\infty]$ an arbitrary function. 
    We say that $(X\times Y,c)$ is an NNCC space if for every $(x_0,x_1,\yb)\in X\times X\times Y$ such that $c(x_0,\yb)$ and $c(x_1,\yb)$ are finite, there exists a path $\rmx\colon [0,1]\to X$ such that $\rmx(0)=x_0$, $\rmx(1)=x_1$, $c(\rmx(s),\yb)$ is finite and for all $s\in (0,1)$,
    \begin{equation}\label{eq:intro-nncc-inequality}
        \forall y\in Y,\quad c(\rmx(s),\yb)-c(\rmx(s),y)\leq (1-s)[c(x_0,\yb)-c(x_0,y)]+s[c(x_1,\yb)-c(x_1,y)]\,,
    \end{equation}
    whenever the right-hand side is well-defined.
\end{definition}

Note that the source of undefiniteness in the right-hand side of \eqref{eq:intro-nncc-inequality} comes from the terms $c(x_i,y)$, since $c(x_i,\yb)$ is assumed to be finite for $i=0,1$. We refer to the paths $s\mapsto (\rmx(s),\yb)$ as \emph{variational c-segments}. 
They are important in their own right since they play a role similar to straight lines in affine geometry and geodesics in positively or nonpositively curved  metric spaces (see Section \ref{sec:NNCC-PC}), and can be seen as providing a nonsmooth geometry to the product $X\times Y$ based on the cost $c$.

In \Cref{sec:products-submersions} we extend two results of Kim and McCann that say that nonnegative cross-curvature is preserved by products of domains (\Cref{prop:products}) and Riemannian submersions (\Cref{prop:sncc-surj}). 
For the latter we identify on general sets and costs a structure 
that generalizes Riemannian submersions as well as submetries in metric spaces, and which we call cost submersion (\Cref{DefCostSubmersion}). Given two projections $P_1\colon X\to\Xu$, $P_2\colon Y\to\Yu$, if the product $X\times Y$ can be turned into a fibered space whose fibers $(P_1^{-1}(\{\xu\}), P_2^{-1}(\{\yu\}))$ are ``equidistant'' for the cost $c$, then there is a natural cost on the fibers defined by $\cu(\xu,\yu)=\inf\{c(x,y) : x\in P_1^{-1}(\{\xu\}), y\in P_2^{-1}(\{\yu\})\}$. In that case we can project certain variational c-segments and obtain 

\begin{proposition}
    If $(X\times Y,c)$ is an NNCC space then so is $(\Xu\times \Yu, \cu)$.
\end{proposition}
\noindent
This is a similar result to the two facts that in Riemannian geometry, nonnegative sectional curvature is preserved under Riemannian submersions, and in metric geometry nonnegative curvature in the sense of Alexandrov is preserved under submetries.

In \Cref{sec:NNCC-PC} we consider the case where the cost is given by a squared distance, looking in particular at the connections between NNCC spaces and positively curved metric spaces in the sense of Alexandrov (PC spaces). In \Cref{lemma:distance-1-convex-c-segments} we show that given an arbitrary metric space $(X,d)$, a variational c-segment $(\rmx(s),\yb)$ on $(X\times X,d^2)$ satisfies the inequality 
\begin{equation}\label{eq:intro-vcs-ineq}
    d^2(\rmx(s),\yb)\leq (1-s)d^2(\rmx(0),\yb)+s\,d^2(\rmx(1),\yb)-s(1-s)d^2(\rmx(0),\rmx(1))\,.
\end{equation}
In other words, the squared distance to $\yb$ is $1$-convex along $\rmx$. If $(X,d)$ is a PC space, this is exactly the opposite inequality that a geodesic satisfies. The need for reverse inequalities of the type \eqref{eq:intro-vcs-ineq} on spaces of measures is precisely what led Ambrosio, Gigli and Savaré to introduce their \emph{generalized geodesics} \cite{ambrosio2005gradient}, see \Cref{sec:perspectives}. 
As a consequence of \eqref{eq:intro-vcs-ineq} we generalize a result by Loeper who showed that on Riemannian manifolds, nonnegative sectional curvature is always implied by the MTW condition (thus by nonnegative cross-curvature); our analogue in metric geometry reads

\begin{proposition}
    Let $(X,d)$ be a geodesic space. If $(X\times X,d^2)$ is an NNCC space then $(X,d)$ is a PC space. 
\end{proposition}

In \Cref{SecGH} we consider the Gromov--Hausdorff notion of convergence in the class of compact metric spaces. We show that when the cost is a squared distance $d^2(x,y)$, or more generally of the form $f(d(x,y))$ where $f$ is locally Lipschitz, the class of NNCC spaces is closed under Gromov--Hausdorff convergence:

\begin{theorem}
    If $(X_n\times X_n,f(d_n))$ is a sequence of NNCC spaces and $(X_n,d_n)$ converges to $(Z,d_Z)$ in the Gromov--Hausdorff topology, then $(Z\times Z,f(d_Z))$ is an NNCC space.
\end{theorem}

In \Cref{sec:sncc-discussion} we compare the NNCC condition to Loeper's maximum principle, a synthetic version of the MTW condition. We point out similarities in the two formulations, ``below the chord'' and ``below the max''.

In \Cref{sec:Figalli-Kim-McCann} we derive a characterization of NNCC spaces that extends one of the main results of Figalli, Kim and McCann in their aforementioned work. In particular, we prove that NNCC spaces with finite costs are characterized by the convexity of the set of c-concave functions with nonempty c-subdifferentials,
\begin{equation}\label{eq:intro-Phic0}
\Phi^c_0 = \{\phi\colon Y\to\R : \partial^c\phi(y)\neq\emptyset \text{ for every $y\in Y$}\}.
\end{equation}

Section \ref{sec:sncc_proba} is concerned with one of our main results:  nonnegative cross-curvature can be lifted from a ground space to the space of measures. In fact, combining \Cref{ThMainTheorem} and \Cref{ThConverseImplication} we can state the following equivalence:

\begin{theorem}\label{thm:wasserstein-intro}
    Let $X$ and $Y$ be Polish spaces and let $c\colon X\times Y\to\R\cup\{+\infty\}$ be a lower semi-continuous cost function bounded from below such that $\sup_{x\in X}\inf_{y\in Y}c(x,y)<+\infty$. Then $(\Pc(X)\times\Pc(Y),\Tc_c)$ is an NNCC space if and only if $(X\times Y,c)$ is an NNCC space.
\end{theorem}
\noindent
As an immediate corollary we have:
\begin{corollary}
    The $2$-Wasserstein space $(\Pc_2(\R^n)\times\Pc_2(\R^n),W_2^2)$ is NNCC. 
\end{corollary}
\noindent
In \Cref{SecCounterExampleMTW} we show that the MTW condition, even in its weak formulation known as Loeper's maximum principle, does not lift to the Wasserstein space, motivating the focus on nonnegative cross-curvature rather than the MTW condition in this paper.

Finally \Cref{SecExamples} contains several new examples of finite- and infinite-dimensional NNCC spaces. We compile them in the following list which contains also examples located in other sections:
\begin{enumerate}[1.]
    \item (\Cref{ssec:examples_riemannian}) Compact Riemannian manifolds with convex injectivity domains which are nonnegatively cross-curved in the smooth sense, $\mathfrak{S}_c\geq 0$, with $c$ being the squared Riemannian distance. This includes spheres and products of spheres.
    
    \item (\Cref{ssec:BW}) The squared Bures--Wasserstein metric on symmetric positive semi-definite matrices
    \[\operatorname{BW}^2(S_1,S_2) = \operatorname{tr}(S_1) + \operatorname{tr}(S_2) - 2 \operatorname{tr}((S_1^{1/2}S_2S_1^{1/2})^{1/2})\,.
    \]
    This example has the structure of a manifold with boundary, a case encompassed by our definition.
    
    \item (Examples \ref{ex:Bregman} and \ref{ex:Bregman-II}) Bregman divergences on infinite-dimensional spaces.
    
    \item (\Cref{sec:KL}) The relative entropy or Kullback--Leibler divergence $\operatorname{KL}(\mu,\nu)$ between probability measures.
    
    \item (\Cref{sec:hellinger}) The squared Hellinger distance 
    \[\mathcal{H}^2(\mu,\nu) = \int  \Big(\sqrt{\frac{d\mu}{d\lambda}} - \sqrt{\frac{d\nu}{d\lambda}}\Big)^2 d\lambda\,,
    \] 
    and the square of the Fisher--Rao distance $\operatorname{FR}(\mu,\nu) = \arccos\left(1 - \frac{1}{2} \mathcal{H}(\mu,\nu)\right)$ on probability measures.
    
    \item (\Cref{sec:GW}) The squared Gromov--Wasserstein distance
    \[
        \GW^2(\bX,\bY)=\inf_{\pi\in\Pi(\mu,\nu)}\int\abs{f(x,x')-g(y,y')}^2\,d\pi(x,y)\,d\pi(x',y')\,,
    \]
    defined on metric measure spaces or more generally gauge spaces.
    
    \item (\Cref{SecUOT}) The unbalanced optimal transport cost when the underlying cone is an NNCC space.
\end{enumerate}

\subsection{Perspectives} \label{sec:perspectives}

\subsubsection*{EVI gradient flows}
In their seminal book \cite{ambrosio2005gradient}, Ambrosio, Gigli and Savaré develop a theory to construct the gradient flow of a function $F$ on a general metric space $(X,d)$. 
A key condition underpinning much of the uniqueness, stability and rates of convergence of their \emph{EVI gradient flows} is \textbf{Assumption 4.0.1}, which requires for every $x_0,x_1,\yb\in X$ the existence of a curve $\gamma$ connecting $x_0$ to $x_1$ along which
\begin{equation}\label{eq:assumption-4.0.1.}
    F(\cdot) + \frac{1}{2\tau}d^2(\cdot,\yb) \quad\text{is $(\frac{1}{\tau}+\lambda)$-convex}.
\end{equation}
Here $\tau>0$ is a step size and $\lambda$ a strong convexity parameter. It is natural to split \eqref{eq:assumption-4.0.1.} into a structural condition on $(X,d)$ and a convexity condition on $F$, namely $1$-convexity of $\frac12 d^2(\cdot,\yb)$ and $\lambda$-convexity of $F$ along $\gamma$. 
However \cite{ambrosio2005gradient} provides little information on how to find such curves in general, except in the case where $X$ is the space of probability measures over a Euclidean space and $d$ is the $2$-Wasserstein distance---the central focus of the book---in the form of \emph{generalized geodesics}.

Our variational c-segments $(\rmx(\cdot),\yb)$ offer a solution to the question above, since they provide in a metric setting exactly the needed $1$-convexity, as a result of \eqref{eq:intro-vcs-ineq}. 
As an immediate consequence, all our examples of metric NNCC spaces provide \textbf{new playing fields to construct EVI gradient flows}: 
the Wasserstein space on the sphere and on the Bures--Wasserstein space, the Gromov--Wasserstein space, and the squared Hellinger and Fisher--Rao distances on probability measures.

Perhaps more importantly, variational c-segments and NNCC spaces pave the way for extending the Ambrosio--Gigli--Savaré theory of gradient flows beyond metric spaces, using minimizing movement schemes based on cost functions $c(x,y)$ which are more general than squared distances, by following the arguments put forth by Léger--Aubin-Frankowski in finite dimensions \cite{leger2023gradient}.

\subsubsection*{Functional inequalities}
The Prékopa--Leindler inequality \cite{Prekopa1,Prekopa2,Leindler} is a functional form of the celebrated Brunn--Minkowski inequality \cite{Gardner_BrunnMinkowski2002}: it says that given $0<s<1$ and three nonnegative functions $f,g,h$ on $\Rn$, if 
\begin{equation}\label{eq:intro-PLI}
    \forall x_0,x_1\in\Rn, \quad h((1-s)x_0+sx_1)\geq f(x_0)^{1-s}g(x_1)^s,
\end{equation}
then $\int h\geq (\int f)^{1-s}(\int g)^s$. From the Prékopa--Leindler inequality can be deduced concentration of measure inequalities \cite{Maurey91,BobkovLedoux2000}, marginal preservation of log-concavity \cite{Prekopa2}, and log-Sobolev, Brascamp--Lieb, Talagrand, and Poincaré inequalities \cite{Maurey91,BobkovLedoux2000,CEMS06}. 
In \cite{CEMS01} the Prékopa--Leindler inequality was generalized to Riemannian manifolds, and later understood to in fact characterize Ricci curvature lower bounds. 
This Riemannian extension relies on an optimal transport interpretation due to McCann \cite{McCannPhdThesis}: when $f,g$ are probability densities, a function $h$ satisfying \eqref{eq:intro-PLI} must majorize the McCann displacement interpolant $\rho_s$ between $f$ and $g$, in which case the integral of $h$ exceeds $1$. 

When $X$ and $Y$ are $n$-dimensional manifolds and $c\in C^2(X\times Y)$, our construction of \textbf{lifted c-segments} may be used to further extend McCann's perspective beyond the Riemannian case, to a geometry where fluid particles travel along variational c-segments. Indeed if $(\mu_s,\nu)$ is a lifted c-segment induced by $(\Gamma,\gamma_\nu,\Lambda_s)$ (\Cref{def:lifted_csegment}), then adapting the arguments of \cite{McCannPhdThesis,CEMS01} shows the bound
\begin{equation}
    \rho_s(\Lambda_s,y)\leq \rho_0(x_0,y)^{1-s} \rho_1(x_1,y)^{s} 
\end{equation}
to be satisfied, under regularity assumptions, for every $(x_0,x_1,y)\in\Gamma$, where now $\rho_s(x,y)$ stands for the Radon--Nikodym derivative of $\mu_s\otimes\nu$ with respect to the volume form of the Kim--McCann metric. This point of view opens the door to Prékopa--Leindler inequalities based on cost functions $c(x,y)$, from which may be deduced new log-Sobolev, Brascamp--Lieb, Poincaré and Talagrand inequalities.

\subsubsection*{Mechanism design}

Mechanism design \cite{borgers2015introduction,VohraBook2011} is a branch of game theory with applications to contract theory \cite{QuinziiRochet85}, voting theory \cite{Gibbard73,Satterthwaite75,Barbera2001}, optimal taxation \cite{mirrlees1971exploration,LaffontMartimort2002}, market design \cite{bichler2017market}, and more. A typical setup involves an agent (or a population of agents) and a mechanism designer, or \emph{principal}, who is able to offer the agent one of several alternatives $y \in Y$ in exchange of a monetary transfer $t\in\R$ (which may be positive or negative). The agent's characteristics are encoded by a \emph{type} $x \in X$, which is unknown to the principal. The utility of an agent of type $x$ receiving alternative $y$ and paying transfer $t$ is often assumed to be quasi-linear, i.e.\ of the form $-c(x,y)-t$.

Whatever the principal's objective may be (for instance, maximize her profits), her task is to devise a \emph{direct mechanism} $(T,t)$, where the ``decision rule'' $T\colon X\to Y$ and the ``transfer rule'' $t\colon X\to\R$ describe the alternative $T(x)$ and transfer $t(x)$ offered to any agent claiming to be of type $x$. A direct mechanism should then ensure that agents report their type truthfully: this takes the form of the so-called \emph{incentive compatibility condition}
\begin{equation}
    \forall x,x'\in X,\quad c(x,T(x))+t(x)\leq c(x,T(x'))+t(x')\,.
\end{equation}
In \cite{CarlierJME01}, Carlier characterized incentive compatible mechanisms $(T,t)$ in terms of c-concave functions $\psi\colon X\to \R$ with nonempty c-subdifferentials, in the form $T(x)\in\partial^c\psi(x)$, $t(x)=\psi(x)-c(x,T(x))$. This is precisely the set of functions \eqref{eq:intro-Phic0} whose convexity is shown to be equivalent to the NNCC condition in \Cref{sec:Figalli-Kim-McCann}. Therefore, like the Figalli--Kim--McCann result, a better understanding of convexity in the choice of direct mechanisms can be expected to have important theoretical and numerical consequences. 

\subsubsection*{Optimal transport in infinite dimensions}
Optimal transport in infinite dimensions has been studied for different infinite-dimensional underlying spaces: the Wiener space \cite{FeyelUstunel}, the configuration space \cite{decreusefond2006wassersteindistanceconfigurationspace}, RCD spaces \cite{CAVALLETTI2014136}, the Wasserstein space \cite{DELLOSCHIAVO2020108397, emami2024mongekantorovichproblemwassersteinspace}. In these articles, the existence of an optimal map was often established. For instance,  in \cite{emami2024mongekantorovichproblemwassersteinspace}, under appropriate hypotheses, an optimal map exists and is given by the exponential map of the gradient of a c-concave function for the cost $W_2^2$. The knowledge of the NNCC property on the Wasserstein space gives practical information on the structure of these optimal potentials. Lastly, the NNCC property would also appear as a key property for studying the regularity of optimal maps in infinite dimensions, which, to the best of our knowledge, has not yet been addressed in the literature.

\section{A synthetic formulation of nonnegative cross-curvature}\label{sec:sncc}

\subsection{Variational c-segments and NNCC spaces}\label{sec:sncc-def}

To understand the rationale behind our nonsmooth definition of nonnegative cross-curvature (\Cref{def:nncc-space}) let us first consider a smooth finite-dimensional setting. Take two $n$-dimensional smooth manifolds $X$ and $Y$ together with a function $c\colon X\times Y\to\R$, called the \emph{cost function}. The MTW tensor, which will be defined in a moment, is intimately tied to particular curves in the product space $X\times Y$ called \emph{c-segments}.

\begin{definition}\label{def:c-segment}
    Let $c\in C^1(X\times Y)$. A c-segment is a path $(s\in[0,1])\mapsto (\rmx(s),\yb)\in X\times Y$ that satisfies
    \begin{equation}\label{eq:def-c-segment}
        \nabla_yc(\rmx(s),\yb)=(1-s)\nabla_yc(\rmx(0),\yb)+s\nabla_yc(\rmx(1),\yb)\,.
    \end{equation}
\end{definition}

A standard set of assumptions to work with the MTW tensor is the following \cite{ma2005regularity,kim2010continuity,loeper2009regularity,TrudingerWang_2nd_boundary2009}:
\begin{align}
    &\text{$c\in C^4(X\times Y)$;} \label{eq:ass-C4}\\
    &\text{for all $(x,y)\in X\times Y$, the linear maps $\nabla^2_{xy}c(x,y)\colon T_xX \to T^*_yY$ are nonsingular.}
    \label{eq:ass-non-deg}
\end{align} 
Here $T_xX$ and $T^*_yY$ denote a tangent and a cotangent space respectively. Assumption \eqref{eq:ass-non-deg} is often called \emph{non-degeneracy} of the cost and ensures that c-segments \eqref{eq:def-c-segment} are well-defined locally in time, by the implicit function theorem. 
When needing globally well-defined c-segments, an additional pair of natural assumptions are the following: for all $(x,y)\in X\times Y$,
\begin{align}
    &\text{the maps $\nabla_xc(x,\cdot)\colon Y\to T_x^*X$, \quad $\nabla_yc(\cdot, y)\colon X\to T^*_yY$ are injective;}\label{eq:ass-bitwist}\\
    &\text{$\nabla_xc(x,Y)$ and $\nabla_yc(X,y)$ are convex subsets of the cotangent spaces $T^*_x X$ and $T^*_y Y$.}\label{eq:ass-biconvex}
\end{align}
Assumption \eqref{eq:ass-bitwist} is often called the \emph{bi-twist condition} and \eqref{eq:ass-biconvex} is often referred to as c-convexity of the domains $X$ and $Y$. Together \eqref{eq:ass-bitwist} and \eqref{eq:ass-biconvex} imply for any $x_0,x_1\in X$ and any $\yb\in Y$ the existence of a unique c-segment $s\mapsto (\rmx(s),\yb)$ such that $\rmx(0)=x_0$ and $\rmx(1)=x_1$. Note also that \eqref{eq:ass-non-deg} is implied by \eqref{eq:ass-bitwist} under \eqref{eq:ass-C4}.

\begin{definition}[MTW tensor]\label{def:MTW-tensor}
    Under the assumptions \eqref{eq:ass-C4}--\eqref{eq:ass-non-deg} the MTW tensor can be defined by 
\begin{equation} \label{eq:Sc-coord}
    \mathfrak{S}_c(x,y)(\xi,\eta) = (c_{ik\mb}c^{\mb r}c_{r\jb\lb}-c_{i\jb k\lb})|_{(x,y)}\,\xi^i\eta^\jb\xi^k\eta^\lb.
\end{equation}
Here $x\in X$, $y\in Y$, $\xi$ is a tangent vector at $x$ and $\eta$ is a tangent vector at $y$. Given local coordinates $x^i$ on $X$ and $y^\ib$ on $Y$, we denote $c_i=\partial_ic$, $c_\ib=\partial_\ib c$, $c_{ij}=\partial_{ij}c$, etc, so that unbarred indices refer to $x$-derivatives while barred indices refer to $y$-derivatives. We also denote by $c^{\jb i}$ the inverse of the matrix $c_{i\jb}$ and adopt the convention that summation over repeated indices is not explicitly written.
\end{definition}

\begin{definition}[Nonnegative cross-curvature]\label{DefNNCCStrong}
    We say that $c$ has nonnegative cross-curvature, or that $c$ is nonnegatively cross-curved, if for all $(x,y)\in X\times Y,$ and all tangent vectors $\xi\in T_xX$ and $\eta\in T_yY$,
    \begin{equation}\label{eq:condition-nncc}
    \mathfrak{S}_c(x,y)(\xi,\eta)\geq 0\,.
    \end{equation}
\end{definition}

The first step towards a nonsmooth version of condition \eqref{eq:condition-nncc} is a characterization of nonnegative cross-curvature due to Kim and McCann, which we state here in a slightly stronger form than their original statement. In particular Assumption \eqref{eq:ass-bitwist} can be relaxed, the important point being that c-segments connecting different points always exist. 

\begin{theorem}[{\cite[Theorem 2.10]{kim2012towards}}, {\cite[Lemma 4.3]{Figalli_Kim_McCann_screening2011}}]\label{thm:KM_characterization}
    Let $X$, $Y$ and $c$ satisfy Assumptions \eqref{eq:ass-C4}--\eqref{eq:ass-biconvex}. Then $c$ is nonnegatively cross-curved if and only if for any three points $x_0,x_1\in X$ and $\yb\in Y$, and for any $y\in Y$, the function 
    \begin{equation}\label{eq:thm:KM_characterization}
    (s\in[0,1])\mapsto c(\rmx(s),\yb)-c(\rmx(s),y)
    \end{equation}
    is convex, where $s\mapsto(\rmx(s),\yb)$ is the c-segment such that $\rmx(0)=x_0$ and $\rmx(1)=x_1$.  
\end{theorem}

Looking at the condition \eqref{eq:thm:KM_characterization} given by \Cref{thm:KM_characterization}, it seems that the cost $c$ needs to be at least differentiable with respect to the $y$ variable in order to be able to define c-segments.
Our key observation is that defining c-segments is in fact not necessary, as the next result shows.

\begin{lemma}\label{lemma:auto-csegment}

Consider two smooth manifolds $X$, $Y$ and $c\in C^1(X\times Y)$. Let $(s\in [0,1]) \mapsto \rmx(s)$ be a smooth curve in $X$ and let $\yb\in Y$ be such that for all $s\in[0,1]$,
\begin{equation}\label{eq:lemma:auto-csegment}
    \forall y\in Y,\quad c(\rmx(s),\yb)-c(\rmx(s),y)\leq (1-s)[c(\rmx(0),\yb)-c(\rmx(0),y)] + s[c(\rmx(1),\yb)-c(\rmx(1),y)]\,.
\end{equation}
Then $\nabla_yc(\rmx(s),\yb) = (1-s)\nabla_y c(\rmx(0),\yb) + s \nabla_yc(\rmx(1),\yb)$.
\end{lemma}

The c-segment equation is therefore implied by \eqref{eq:lemma:auto-csegment}. Note that condition \eqref{eq:lemma:auto-csegment} is weaker than convexity of $s\mapsto c(\rmx(s),\yb)-c(\rmx(s),y)$. It simply assumes that $c(\rmx(s),\yb)-c(\rmx(s),y)$ lies \emph{below the chord}, i.e.\ below the linear interpolation of its values at $s=0$ and $s=1$. 

\begin{proof}[Proof of \Cref{lemma:auto-csegment}]
    Fix $(\rmx,\yb)$ satisfying \eqref{eq:lemma:auto-csegment} and consider an arbitrary smooth curve $(t\in (-\varepsilon,\varepsilon))\mapsto \rmy(t)$ in $Y$ such that $\rmy(0)=\yb$. Then we have a family $\{g(\cdot, t)\}_{t\in (-\varepsilon,\varepsilon)\setminus\{0\}}$ of functions
    \[
    g(s,t)\coloneqq \frac{c(\rmx(s),\rmy(0))-c(\rmx(s),\rmy(t))}{t}
    \]
    which satisfies 
    \begin{equation}\label{eq:proof-lemma:auto-csegment-1}
        g(s,t)\leq (1-s)\,g(0,t)+s\,g(1,t)\,.
    \end{equation}
    As $t\to 0$ this family converges pointwise to $h(s)\coloneqq -\langle q(s),\dot\rmy(0)\rangle$ with $q(s)\coloneqq \nabla_y c(\rmx(s),\yb)$. Here $\langle \cdot,\cdot\rangle$ denotes the duality pairing between a cotangent vector and a tangent vector. Inequality \eqref{eq:proof-lemma:auto-csegment-1} passes to the limit, $h(s)\leq (1-s)h(0)+s\,h(1)$, i.e.\ $\langle q(s)-(1-s)\,q(0)-s\,q(1),\dot\rmy(0)\rangle \geq 0$.
    Since $\dot\rmy(0)$ is an arbitrary tangent vector we deduce that $q(s)=(1-s)q(0)+s\,q(1)$. 
\end{proof}

All the considerations above motivate our extension of nonnegative cross-curvature to nonsmooth settings. When $X$ and $Y$ are two arbitrary sets and $c\colon X\times Y\to [-\infty,+\infty]$ is an arbitrary function potentially taking $\pm\infty$ values, we refer to $(X\times Y,c)$ as a \emph{cost space}, by analogy with the terminology for a metric space.
We say that a quantity $\alpha\in [-\infty,+\infty]$ is finite if it is not equal to $\pm\infty$. 
Throughout this paper, the arithmetic rules we adopt for infinite values are the usual ones in the totally ordered set $[-\infty,+\infty]$,
\begin{align*}
    \alpha+(+\infty) = +\infty + \alpha = \alpha-(-\infty) = +\infty \quad &\text{for } \alpha\in (-\infty,+\infty]\\
    \alpha+(-\infty) = -\infty + \alpha = \alpha-(+\infty) = -\infty \quad &\text{for } \alpha\in [-\infty,+\infty)\,.
\end{align*}
The expressions $(+\infty)+(-\infty)$, $(+\infty)-(+\infty)$, $(-\infty)+(+\infty)$ and $(-\infty)-(-\infty)$ will be called \emph{undefined combinations}. They are a priori undefined but may take a specific $\pm\infty$ value on a case-by-case basis. We are now ready to introduce the curves that will play the role of c-segments on cost spaces.

\begin{definition}[Variational c-segments]\label{def:sncc}
    Let $(X\times Y,c)$ be a cost space. 
    Given a path $\rmx\colon[0,1]\rightarrow X$ and $\yb\in Y$, we say that $(\rmx,\yb)$ is a \emph{variational c-segment on $(X\times Y,c)$} if for all $s\in[0,1]$, $c(\rmx(s),\yb)$ is finite and for all $s\in (0,1)$, we have 
    \begin{equation}\label{eq:NNCC-inequality}
        \forall y\in Y,\quad c(\rmx(s),\yb)-c(\rmx(s),y)\leq (1-s)[c(\rmx(0),\yb)-c(\rmx(0),y)]+s[c(\rmx(1),\yb)-c(\rmx(1),y)],
    \end{equation}
    with the rule $(+\infty)+(-\infty)=(-\infty)+(+\infty)=+\infty$ in the right-hand side.     
\end{definition}

Here by path we mean an arbitrary function $\rmx\colon[0,1]\rightarrow X$. In particular, we do not impose any continuity or regularity on these paths. 
If the space $(X\times Y,c)$ is clear from context we may omit it and simply say variational c-segment. 
We call condition \eqref{eq:NNCC-inequality} the \emph{NNCC inequality}. 
The NNCC inequality is nonlocal since it has to hold for every $y\in Y$. This is a stringent condition, and for an arbitrary cost space $(X\times Y,c)$ there may not exist many paths that satisfy it. When they always exist we define:

\begin{definition}[NNCC space]\label{def:nncc-space}
    We say that $(X\times Y,c)$ is a cost space with nonnegative cross-curvature (NNCC space) if for every $(x_0,x_1,\yb)\in X\times X\times Y$
    such that $c(x_0,\yb)$ and $c(x_1,\yb)$ are finite, there exists a variational c-segment $(\rmx,\yb)$ on $X\times Y$ such that $\rmx(0)=x_0$ and $\rmx(1)=x_1$. 
\end{definition}

When $X=Y$ we will sometimes say that $(X,c)$ is an NNCC space rather than write the product $(X\times X,c)$. NNCC spaces bundle together the notions of nonnegative cross-curvature and c-convexity of the domain \eqref{eq:ass-biconvex}, which is a key assumption for \Cref{thm:KM_characterization}. This is similar in spirit to defining standard convexity of a function $f\colon \Omega\subset\R^n\to\R$ via $f((1-s)x_0+sx_1)\leq (1-s)f(x_0)+sf(x_1)$, where one generally assumes the domain $\Omega$ to be convex so that $(1-s)x_0+sx_1$ is in $\Omega$. On the other hand, if $f$ is twice differentiable the condition $\nabla^2f\geq 0$ can be considered independently from the convexity of $\Omega$. In the same way the differential and local MTW tensor $\Sc$ can be defined independently of the c-convexity of the domains.

NNCC spaces $(X\times Y,c)$ break the symmetry between the sets $X$ and $Y$ that existed when working with $\mathfrak{S}_c$. The two sets need not have the same dimension (in a finite-dimensional manifold setting, say) and they play different roles since $X$ is home to a curve $\rmx$, thus presumably has some continuous structure (although we do not impose this at the outset), while $Y$ may truly be an arbitrary set.

The NNCC inequality does not ask $s\mapsto c(\rmx(s),\yb)-c(\rmx(s),y)$ to be convex but simply \emph{below the chord}, i.e.\ below the line joining the values at $s=0$ to $s=1$. This choice is sufficient for the applications we have in mind and is more closely related to the Figalli--Kim--McCann characterization of nonnegative cross-curvature, see our characterization in terms of c-subdifferentials in \Cref{prop:NNCC-contact-sets}. That being said, it may sometimes be of interest to consider the ``convex'' version of the NNCC inequality, so we introduce the following condition:
\begin{enumerate}[\normalfont(NNCC-conv)]
    \item For any $x_0,x_1\in X$ and $\yb\in Y$ such that $c(x_0,\yb)$ and $c(x_1,\yb)$ are finite, there exists a path $\rmx\colon[0,1]\rightarrow X$ with $\rmx(0)=x_0$, $\rmx(1)=x_1$ such that $c(\rmx(s),\yb)$ is finite and such that for every $y\in Y$, the function $s\mapsto c(\rmx(s),\yb)-c(\rmx(s),y)$ is convex.\label{condition-NNCC-conv}
\end{enumerate}
We shall refer to the paths $(\rmx(s),\yb)$ as \snccconv{}-variational c-segments. The function $h(s)\coloneqq c(\rmx(s),\yb)-c(\rmx(s),y)$ takes values in $[-\infty,+\infty]$ and its convexity should be understood as the convexity of its epigraph \cite{RockafellarBook}
\begin{equation*}
    \operatorname{epi} h \coloneqq \{(s,\alpha)\in [0,1]\times\R : h(s) \leq \alpha\}.
\end{equation*}
When $X$ is a topological space, let us also define another condition which will be related to the MTW condition and the Loeper maximum principle in \Cref{sec:sncc-discussion}. 
\begin{enumerate}[\normalfont(LMP)]
    \item For any $x_0,x_1\in X$ and $\yb\in Y$ such that $c(x_0,\yb)$ and $c(x_1,\yb)$ are finite, there exists a continuous path $\rmx\colon[0,1]\rightarrow X$ with $\rmx(0)=x_0$, $\rmx(1)=x_1$ such that $c(\rmx(s),\yb)$ is finite for every $s\in [0,1]$ and such that 
    \begin{equation}\label{eq:condition-LMP}
    \forall y\in Y,\quad c(\rmx(s),\yb)-c(\rmx(s),y)\leq \max\{c(x_0,\yb)-c(x_0,y),c(x_1,\yb)-c(x_1,y)\}.
    \end{equation}\label{condition-LMP}
\end{enumerate}
We similarly refer to the paths $(\rmx(s),\yb)$ in \eqref{eq:condition-LMP} as \smtw{}-variational c-segments. 
Condition \smtw{} is independent of the continuous reparametrization. That is, if $\rmx(\cdot)$ is a curve satisfying \eqref{eq:condition-LMP}, then for any continuous function $f:[0,1]\rightarrow[0,1]$, with $f(0)=0$ and $f(1)=1$, $\rmx(f(\cdot))$ satisfies it as well. 
In contrast, these generalized reparametrizations do not preserve the inequality defining NNCC, nor the inequality \snccconv{}.

Let us now look at some elementary examples.

\begin{example}[Hilbert squared norms]\label{ex:Hilbert-norm}
    Let $(H,\langle\cdot,\cdot\rangle)$ be a Hilbert space and denote $\lVert u\rVert^2 = \langle u,u \rangle$. First, consider the cost $c(x,y)=\lVert x-y\rVert^2$. Then $(H\times H, c)$ is an NNCC space. Indeed, for any $x_0,x_1,\yb \in H$ define $\rmx(s)=(1-s) x_0+s x_1$. For any $y\in H$ we have 
    \begin{equation}\label{eq:ex-Hilbert}
    c(\rmx(s),\yb)-c(\rmx(s),y) = 2\langle \rmx(s),y-\yb\rangle + \lVert \yb\rVert^2-\lVert y\rVert^2,
    \end{equation}
    which is an affine function of $s$ in view of the linear interpolation $\rmx(s)$. The NNCC inequality is here an equality.
    Furthermore variational c-segments can be shown to be unique, using \eqref{eq:ex-Hilbert} and adapting the argument of Lemma \ref{lemma:auto-csegment}.
    
    More generally consider two arbitrary sets $X$ and $Y$, a function $F\colon X\to H$ whose image $F(X)$ is a convex subset of $H$ and an arbitrary function $G\colon Y\to H$. Let 
    \[
        c(x,y)=\lVert F(x)-G(y)\rVert^2.
    \] 
    Then $(X\times Y,c)$ is an NNCC space. To see why, fix $x_0,x_1\in X$ and $\yb \in Y$. Since $F(X)$ is a convex subset of $H$, for each $s\in (0,1)$ there exists a point $\rmx(s)\in X$ such that $F(\rmx(s))=(1-s) F(x_0) + s F(x_1)$. Define also $\rmx(0)=x_0$ and $\rmx(1)=x_1$. Then for any $y\in Y$,
    \begin{align*}
        c(\rmx(s),\yb)-c(\rmx(s),y) &= \langle 2F(\rmx(s))-G(\yb)-G(y), G(y)-G(\yb)\rangle\\
        &= 2\langle (1-s) F(x_0) + s F(x_1), G(y)-G(\yb)\rangle + \lVert G(\yb)\rVert^2 - \lVert G(y)\rVert^2,
    \end{align*}
    again an affine function of $s$.     
\end{example}

\begin{example}[Bregman divergences] \label{ex:Bregman}
    The previous example can be extended to Bregman divergences in Banach spaces. Let $E$ be a Banach space, $u\colon E\to\R$ a Gateaux-differentiable function and consider the \emph{Bregman divergence}
    \[
    c(x,y) = u(x)-u(y)-Du(y)(x-y),
    \]
    where $Du$ denotes the Gateaux derivative of $u$. Take a convex subset $C\subset E$ and an arbitrary subset $D\subset E$. Then $(C\times D, c)$ is an NNCC space. Indeed one can argue as in the previous example (which is a particular case with $u(x)=\lVert x\rVert^2$), and for $x_0,x_1\in C$, $\yb\in D$ define $\rmx(s)=(1-s)x_0+sx_1$ and find that 
    \[
    c(\rmx(s),\yb)-c(\rmx(s),y) = (Du(y)-Du(\yb))(\rmx(s)) + u(y)-u(\yb)+Du(\yb)(\yb)-Du(y)(y)\,.
    \]
\end{example}

\begin{example}[Bregman divergences II] \label{ex:Bregman-II}
    Let us now consider the case of a reversed Bregman divergence. Let $E$ be a Banach space, $u\colon E\to\R$ a Gateaux-differentiable function and consider
    \[
    c(x,y) = u(y)-u(x)-Du(x)(y-x)\,.
    \]
    Compared to the previous example, we have switched the roles of $x$ and $y$. 
    Let $X \subset E$ be such that $Du(X) \subset E^*$ is a convex subset of the dual space of $E$ and let $Y$ be an arbitrary subset of $E$. Then we can find $\rmx(s)\in X$ such that $Du(\rmx(s)) = (1-s)Du(x_0) + s Du(x_1)$, for any $x_0,x_1\in X$ and $0<s<1$. The corresponding difference of costs in the NNCC inequality is affine in $s$ which shows $(X\times Y,c)$ to be an NNCC space.
\end{example}

\begin{example}[The semi-geostrophic cost]
	Another possible extension of the Hilbert case is to consider the cost
	\begin{equation}\label{eq:semigeostrophic}
	c((x,a),(y,b)) = \frac{1}{2b}\| x- y\|^2 + g \frac a b
	\end{equation}
	on $X=Y=H\times (0,+\infty)$, where $g\in\R\setminus \{0\}$ is a given constant. Following Lemma \ref{lemma:auto-csegment}, a simple computation shows that a variational c-segment from $\big((x_0,a_0),(\yb,\bar b)\big)$ to $\big((x_1,a_1),(\yb,\bar b)\big)$ is necessarily of the form
	\[
	(\rmx(s),a(s))=\Big( (1-s)x_0+sx_1,  (1-s)a_0+sa_1+\frac{s(1-s)}{2g}\|x_0-x_1\|^2\Big)\,,
	\]
	independently of the base point $(\yb,\bar b)$, and the difference of costs is again affine so that $(X\times Y,c)$ is NNCC.
	If $H=\R^2$ and $g$ is the standard acceleration of gravity, \eqref{eq:semigeostrophic} is referred to as the semi-geostrophic cost and it is used within the framework of optimal transport for applications to the  semi-geostrophic equations, a model for atmospheric flows and frontogenesis \cite{CullenMaroofi2003}. 
\end{example}

Note that in the previous examples c-segments $(\rmx,\yb)$ are independent of the base point $\yb$. In the next example, we present the case where the cost is a distance, which leads to discontinuous variational c-segments.

\begin{example}[The Monge cost]\label{ExMongeCost}
    Let $(X,d)$ be any metric space and consider the cost $c(x,y)=d(x,y)$. Then $(X\times X, c)$ is an NNCC space. To show this, fix $x_0,x_1,\yb\in X$ and define the path
    \[
    \rmx(s)=\begin{cases}
        x_0&\text{for }s=0\,;\\
        \yb&\text{for }0<s<1\,;\\
        x_1&\text{for }s=1\,.
    \end{cases}
    \]
    Let $y\in X$ and define $f(s)=c(\rmx(s),\yb)-c(\rmx(s),y)$. We see that 
    \[
    f(s)=\begin{cases}
        d(x_0,\yb)-d(x_0,y)&\text{when }s=0\,;\\
        -d(\yb,y)&\text{when }0<s<1\,;\\
        d(x_1,\yb)-d(x_1,y)&\text{when }s=1\,.
    \end{cases}
    \]
    By the triangle inequality, the constant value $f(s)$ for $s\in (0,1)$ is smaller than both $f(0)$ and $f(1)$, and therefore $f$ is convex on $[0,1]$, showing that $(\rmx,\yb)$ is a variational c-segment.
\end{example}

\subsection{Products and submersions} \label{sec:products-submersions}

In \cite{kim2012towards} Kim and McCann exhibited two operations preserving classical nonnegative cross-curvature:
direct products, and Riemannian submersions for costs given by a squared Riemannian distance. In this section, we present nonsmooth extensions of these two results. We start with finite products, and we will also consider an instance of infinite products in \Cref{prop:RE} in the next section.

\begin{proposition}[Products preserve nonnegative cross-curvature]\label{prop:products}
    Let $A$ be a finite set and let $(X^a\times Y^a,c^a)_{a\in A}$ be a family of NNCC spaces, with costs $c^a\colon X^a\times Y^a\to  (-\infty,+\infty]$. Let $\boldsymbol{X}=\prod_{a\in A}X^a$, $\boldsymbol{Y}=\prod_{a\in A}Y^a$ and $\boldsymbol{c}(\boldsymbol x,\boldsymbol y)=\sum_{a\in A}c^a(x^a,y^a)$ defined on $\boldsymbol X\times \boldsymbol Y$. Then $(\boldsymbol X\times \boldsymbol Y, \boldsymbol c)$ is an NNCC space.
\end{proposition}
\begin{proof}
    Let $\boldsymbol{x}_0=(x^a_0)_a\in\boldsymbol{X}$, $\boldsymbol{x}_1=(x^a_1)_a\in\boldsymbol{X}$ and $\boldsymbol{\yb}=(\yb^a)_a\in\boldsymbol{Y}$
    be such that $\boldsymbol c(\boldsymbol x_0, \boldsymbol \yb)$ and $\boldsymbol c(\boldsymbol x_1, \boldsymbol \yb)$ are finite. Then each $c^a(x^a_0,\yb^a)$ and $c^a(x^a_1,\yb^a)$, $a\in A$, is finite. 
    For each $a\in A$, since $(X^a\times Y^a,c^a)$ is an NNCC space there exists a variational c-segment $(\rmx^a(\cdot),\yb^a)$ from $(x^a_0,\yb^a)$ to $(x^a_1,\yb^a)$; then for all $y^a\in Y^a$, 
    \[
    c^a(\rmx^a(s),\yb^a)-c(\rmx^a(s),y^a)\leq (1-s)[c(x^a_0,\yb^a)-c(x^a_0,y^a)]+s[c(x^a_1,\yb^a)-c(x^a_1,y^a)]\,.
    \]
    Summing this inequality over $a$ gives the desired result. 
    Note that here there is no undefined combination since costs $c^a$ do not take the value $-\infty$.
\end{proof}

Let us now turn our attention to Riemannian submersions. We first present an extension of Kim and McCann's results to principal bundles endowed with an invariant cost in a smooth setting (\Cref{ThPrincipalFiberResult}). We then introduce \emph{cost submersions}, a nonsmooth notion of projections that preserve NNCC spaces (\Cref{prop:sncc-surj}). 

In Riemannian geometry, Riemannian submersions often appear in the following context. Consider a principal fiber bundle $(G,L)$ endowed with a Riemannian metric $g$ (on $L$). Suppose that the group $G$ acts via isometries of the Riemannian metric $g$, then the projection map $L \to L / G$ can be turned into a Riemannian submersion. This situation can be generalized to the setting of principal fiber bundles with invariant costs, instead of metrics. 
We start with the necessary definitions.

\begin{definition}[Submersion and projection of a cost]\label{def:proj-fiber}
Let $ P_i: M_i \to N_i$ for $i = 1,2$ be two submersions between manifolds with compact fibers and ${c}: M_1 \times M_2 \to \R$ be a continuous cost. 
Define the projected cost $\underline{c}: N_1 \times N_2 \to \R$ by 
\begin{equation}
\underline c(\underline x,\underline y) \coloneqq \min \{ {c}({x},{y})\,;\, {x} \in  P_1^{-1}(\{\underline x\})\text{ and }{y} \in  P_2^{-1} (\{\underline y \}) \}\,.
\end{equation}
\end{definition}
In general, this definition does not lead to interesting costs. For instance, one is often interested in costs which are distances when $M_1=M_2$ and $N_1=N_2$. In this situation, the projected cost is not a distance in general. However, that is the case under a transitive group action on the fibers, leaving the cost invariant. We now describe this structure in the general case of a cost on a product space. 

\begin{definition}[Principal bundle with $c$-invariant cost]\label{ThInvariantCostDefinition}
Consider a principal fiber bundle $(G,L_1 \times L_2)$ where $G$ is a compact Lie group and a continuous cost $c\colon L_1 \times L_2 \to \R$ such that $G$ leaves the cost invariant. Namely, we assume 
    \begin{equation}
        c(g \cdot x,g\cdot y) = c(x,y)\,, \quad \forall (x,y) \in L_1 \times L_2\,.
    \end{equation}
    We consider the projection $P : L_1 \times L_2  \to L_1/G \times L_2 /G$.
The corresponding projected cost is denoted $\underline c$.
\end{definition}

We can now state our slight generalization of Kim and McCann's result on Riemannian submersions.

\begin{theorem}[Cross-curvature for invariant costs on principal fiber bundles]  \label{ThPrincipalFiberResult}
    Let $G$ be a compact Lie group.
    Consider  a principal fiber bundle $(G,L_1 \times L_2)$ with a cost $c: L_1 \times L_2 \to \R$ which is $G$-invariant. Then, under Assumption \eqref{AssumptionLifts}, there is more cross-curvature on $L_1/G \times L_2 /G$ than on $L_1 \times L_2$. More precisely,
    \begin{equation}
    \mathfrak{S}_{\underline{c}}( P_1(x), P_2(y))(\xi,\eta) \geq \mathfrak{S}_{c}(x,y)(\tilde{\xi},\tilde{\eta}) \,,
    \end{equation}
    where $x,y$ are optimal lifts of their projections and $\tilde \xi, \tilde \eta$ are horizontal lifts of $\xi,\eta$.
\end{theorem}
We refer to Appendix \ref{sec:appendix-submersion} for the proof and a complete definition of the objects involved in the theorem. Importantly, this suggests that the result could also be transferred to the Wasserstein space, in which case we need a \emph{robust} notion of cost-preserving submersion, under which our notion of NNCC space is preserved.
This leads us to extending Riemannian submersions to arbitrary cost spaces, in the spirit of \emph{submetries} for metric spaces, see \Cref{rem:submetries}.

\begin{definition}[Cost submersion]\label{DefCostSubmersion}
    Consider two cost spaces defined by $c\colon X\times Y\to [-\infty,+\infty]$ and $\cu\colon \Xu\times \Yu\to [-\infty,+\infty]$, and let $P_1\colon X\to\underline{X}$, $P_2\colon Y\to\underline{Y}$ be two surjective maps. 
    Write $x\sim x'$ whenever $P_1(x)=P_1(x')$ and $y\sim y'$ whenever $P_2(y)=P_2(y')$. We say that $P\colon (x,y)\mapsto (P_1(x),P_2(y))$ is a \emph{cost submersion} if: 
    \begin{enumerate}
    \item for every $(x,y)\in X\times Y,$ setting $(\xu,\yu)=P(x,y)$ we have 
\begin{equation}\label{eq:c-subm-property}
    \cu(\xu,\yu)=\inf_{x'\sim x,y'\sim y}c(x',y')=\inf_{y'\sim y}c(x,y')=\inf_{x'\sim x}c(x',y)\,;
\end{equation}
\item in the previous equation, all the infima are attained when $\cu(\xu,\yu)$ is finite.
\end{enumerate} 
We also say that $(x,y)$ is \emph{optimal} if $c(x,y)$ is finite and $c(x,y)=\inf_{x'\sim x,y'\sim y}c(x',y')$, i.e.\ $c(x,y)=\cu(P_1(x),P_2(y))$.
\end{definition}
Through a cost submersion, we intend to transfer a structure from the ``total space'' $(X\times Y,c)$ (variational c-segments, NNCC space) to the ``base space'' $(\Xu\times\Yu,\cu)$. 
Let us start with variational c-segments. First we show that if its endpoints are optimal, a variational c-segment on $(X\times Y,c)$ is ``horizontal''. Note that in this section we denote the base point by $y_0$ instead of $\yb$ for better readability.

\begin{lemma}\label{lemma:c-subm-optimal}
    Let $P\colon (X\times Y,c)\to(\Xu\times \Yu,\cu)$ be a cost submersion. Let $s\mapsto(\rmx(s),y_0)$ be a variational c-segment on $(X\times Y,c)$ such that $(\rmx(0),y_0)$ and $(\rmx(1),y_0)$ are optimal. Then $(\rmx(s),y_0)$ is optimal for each $0<s<1$. 
\end{lemma}
\begin{proof}
    Fix $0<s<1$. By definition $c(\rmx(s),y_0)$ is finite and for all $y\in Y$,
    \begin{equation}\label{eq:prop:sncc-proj-1}
        c(\rmx(s),y_0)-c(\rmx(s),y)\leq (1-s)[c(\rmx(0),y_0)-c(\rmx(0),y)]+s[c(\rmx(1),y_0)-c(\rmx(1),y)]\,.
    \end{equation}
    For any $y\sim y_0$ the right-hand side of \eqref{eq:prop:sncc-proj-1} lies in $[-\infty,0]$ since the endpoints $(\rmx(0),y_0)$ and $(\rmx(1),y_0)$ realize the minimum of $c$ along their respective fibers. We deduce that $c(\rmx(s),y_0)\leq c(\rmx(s),y)$ for any $y\sim y_0$ which by the cost submersion property \eqref{eq:c-subm-property} implies that $(\rmx(s),y_0)$ is optimal.
\end{proof}

A ``horizontal'' variational c-segment can then be projected to give a variational c-segment on the base.

\begin{lemma}[Projecting variational c-segments]\label{lemma:projected-c-segments}
    Let $P\colon (X\times Y,c)\to(\Xu\times \Yu,\cu)$ be a cost submersion. Let $\underline{x_0},\underline{x_1}\in\underline{X}$ and $\underline{y_0}\in\underline{Y}$ be such that $\cu(\underline{x_0},\underline{y_0})$ and $\cu(\underline{x_1},\underline{y_0})$ are finite. Given $y_0\in P_2^{-1}(\{\underline{y_0}\})$, 
    suppose there exists a variational c-segment $s\mapsto (\rmx(s),y_0)$ on $(X\times Y,c)$ with $P_1(\rmx(0))=\underline{x_0}$ and $P_1(\rmx(1))=\underline{x_1}$ such that $(\rmx(0),y_0)$ and $(\rmx(1),y_0)$ are optimal. Define $\underline\rmx(s)= P_1(\rmx(s))$. Then $s\mapsto(\underline\rmx(s),\underline{y_0})$ is a variational c-segment on $(\Xu\times\Yu,\cu)$ between $(\underline{x_0},\underline{y_0})$ and $(\underline{x_1},\underline{y_0})$. 
\end{lemma}
\begin{proof}
    Fix $\yu\in\Yu$. By the NNCC inequality, for any $y\in P_2^{-1}(\{\yu\})$,
    \begin{equation*}
        c(\rmx(s),y_0)-c(\rmx(s),y)\leq (1-s)[c(\rmx(0),y_0)-c(\rmx(0),y)]+s[c(\rmx(1),y_0)-c(\rmx(1),y)]\,.
    \end{equation*}
    Here we have $c(\rmx(0),y_0)=\cu(\underline{x_0},\underline{y_0})$ and $c(\rmx(1),y_0)=\cu(\underline{x_1},\underline{y_0})$ and we can bound $-c(\rmx(0),y)\leq -\cu(\underline{x_0},\yu)$ and $-c(\rmx(1),y)\leq -\cu(\underline{x_1},\yu)$. By \Cref{lemma:c-subm-optimal} we also have $c(\rmx(s),y_0)=\cu(\underline{\rmx}(s),\underline{y_0})$ even though we could do without this information here (only retaining the latter quantity is finite) and simply bound $c(\rmx(s),y_0)\geq \cu(\underline{\rmx}(s),\underline{y_0})$ (but see \Cref{rem:c-sub-NNCC-conv}). Therefore
    \[
        \cu(\underline{\rmx}(s),\underline{y_0})-c(\rmx(s),y)\leq (1-s)[\cu(\underline{x_0},\underline{y_0})-\cu(\underline{x_0},\yu)]+s[\cu(\underline{x_1},\underline{y_0})-\cu(\underline{x_1},\yu)]\,.
    \]
    Maximizing the left-hand side over $y\in P_2^{-1}(\{\yu\})$ combined with the cost-submersion property \eqref{eq:c-subm-property} gives us the desired NNCC inequality,
    \begin{equation}\label{eq:nncc-ineq-proj}
        \cu(\underline{\rmx}(s),\underline{y_0})-\cu(\underline{\rmx}(s),\yu)\leq (1-s)[\cu(\underline{x_0},\underline{y_0})-\cu(\underline{x_0},\yu)]+s[\cu(\underline{x_1},\underline{y_0})-\cu(\underline{x_1},\yu)]\,.
    \end{equation}
\end{proof}

As a direct consequence we obtain:

\begin{proposition}[Cost submersions preserve NNCC]\label{prop:sncc-surj}
    Let $P\colon (X\times Y,c)\to(\Xu\times \Yu,\cu)$ be a cost submersion. If $(X\times Y,c)$ is an NNCC space then so is $(\underline{X}\times \underline{Y},\underline{c})$.
\end{proposition}
\begin{proof}
    Let $\underline{x_0},\underline{x_1}\in\underline{X}$ and $\underline{y_0}\in\underline{Y}$ be such that $\cu(\underline{x_0},\underline{y_0})$ and $\cu(\underline{x_1},\underline{y_0})$ are finite. Take any $y_0\in P_2^{-1}(\{\underline{y_0}\})$. By definition of a cost submersion there exist $x_0\in P_1^{-1}(\{\underline{x_0}\})$, $x_1\in P_1^{-1}(\{\underline{x_1}\})$ such that $(x_0,y_0)$ and $(x_1,y_0)$ are optimal. Since $(X\times Y,c)$ is an NNCC space there exists a variational c-segment $(\rmx,y_0)$ such that $\rmx(0)=x_0$, $\rmx(1)=x_1$. By \Cref{lemma:projected-c-segments},  $(\underline\rmx,\underline{y_0})$ with $\underline\rmx(s)\coloneqq P_1(\rmx(s))$ is a suitable variational c-segment on $(\Xu\times\Yu,\cu)$.
\end{proof}

\begin{remark}\label{rem:c-sub-NNCC-conv}
    It is not hard to check that Lemmas \ref{lemma:c-subm-optimal}, \ref{lemma:projected-c-segments} and \Cref{prop:sncc-surj} hold when the NNCC condition is replaced with condition \snccconv{} or condition \smtw{}. For the \smtw{} condition, the proofs follow exactly the same line of reasoning. For the \snccconv{} condition, one needs to verify inequality \eqref{eq:nncc-ineq-proj} on any subinterval $[a,b]$ of $[0,1]$, 
    \[
    \cu(\underline{\rmx}(s),\underline{y_0})-\cu(\underline \rmx(s),\yu)\leq (1-s)[\cu(\underline\rmx(a),\underline{y_0})-\cu(\underline\rmx(a),\yu)]+s[\cu(\underline\rmx(b),\underline{y_0})-\cu(\underline\rmx(b),\yu)]\,.
    \]
    In order to repeat the proof of Lemma \ref{lemma:projected-c-segments} the points $(\rmx(a),y_0)$ and $(\rmx(b),y_0)$ need to be optimal: this is ensured by \Cref{lemma:c-subm-optimal}. Therefore, the situation is analogous to Kim and McCann's result \cite[Corollary 4.7]{kim2012towards} which shows that in the classical smooth setting both nonnegative cross-curvature and the MTW condition \eqref{eq:MTW-condition} are preserved by Riemannian submersions. 
\end{remark}

\begin{remark}[Submetries]\label{rem:submetries}
    Given $x\in X$ and $\alpha\in\R$, let $X_\alpha(x)\coloneqq \{y\in Y : c(x,y)\leq \alpha\}$, and define similarly the sets $\Xu_\alpha(\xu)$, $Y_\alpha(y)$, $\Yu_\alpha(\yu)$. Then one can check that $P_1,P_2$ realize a cost submersion if and only if for every $x\in X$, the image of $X_\alpha(x)$ under $P_2$ is $\Xu_\alpha(P_1(x))$ and for every $y\in Y,$ the image of $Y_\alpha(y)$ under $P_1$ is $\Yu_\alpha(P_2(y))$. 
    In the metric case $X=Y,$ $\Xu=\Yu$, and $c(x,y)=d(x,y)^p$ where $d$ is a distance and $p>0$, the projected cost is such that $\underline{d}(\xu,\yu)\coloneqq \cu(\xu,\yu)^{1/p}$ is automatically a distance too. In that setting cost submersions are \emph{submetries}, which were introduced by Berestovskii \cite{Berestovskii1987} as a metric version of Riemannian submersions. See also \cite[Section 4.6]{BuragoGromovPerelman1992} and \cite{KapovitchLytchak2022}.
\end{remark}

As an interesting consequence of \Cref{prop:sncc-surj}, translation-invariant NNCC costs on the Euclidean space are stable by taking infimal convolutions:

\begin{corollary}
[Infimal convolution of translation invariant costs]\label{cor:infimal-convolution}
    Let $X = Y = \R^n$ and let $c_1,c_2\colon \R^n \to [0,+\infty)$. Define the infimal convolution of the costs $c_1(x - y)$ and $c_2(x - y)$ by
    \begin{equation}
        (c_1 \boxplus c_2)(x,y) \coloneqq \inf \{ c_1(x' - y') + c_2(x'' - y'') \,: \, x' + x'' = x \text{ and }y' + y'' = y\}\,.
    \end{equation}
     Then, if $(\Rn\times\Rn,c_1)$ and $(\Rn\times\Rn,c_2)$ are both NNCC spaces with lower semi-continuous costs and either $c_1$ or $c_2$ is coercive,  then $(\Rn\times\Rn,c_1 \boxplus c_2)$ is also an NNCC space.
\end{corollary}

\begin{proof}
Note that the infimal convolution is a projection of the costs $c_1 + c_2$ on the product space by the projections $P_1(x',x'')= x'+x''$ and $P_2(y',y'') =y' + y'' $. To obtain the required properties in Proposition \ref{prop:sncc-surj}, we reformulate it as an example of Definition \ref{ThInvariantCostDefinition}. 
Define the linear action of the group $\R^n$ on $\R^n \times \R^n$ by $\lambda \in \R^n$,  $(x',x'') \mapsto (x' + \lambda,x'' -\lambda)$.
The associated diagonal action leaves the cost $c_1 \boxplus c_2$ invariant. 
 As a consequence, the two following infima coincide:
\begin{multline}\label{EqTICosts}
    \inf \{ c_1(x' - y') + c_2(x - y - (x' - y')) \,: \, x',y' \in \R^n \}
     = \inf_{y'}\{ c_1(x'-y') + c_2(x-y - (x'-y'))\}\,.
\end{multline}
By the lower semi-continuity of the costs and coercivity of one of the two costs, the infimum on the right-hand side in Eq.$\,$\eqref{EqTICosts} is attained.
By symmetry, the same argument applies when fixing $(y',y'')$ and minimizing on $(x',x'')$. The assumption in Definition \ref{DefCostSubmersion} is thus satisfied.
\end{proof}

\begin{example}[Soft-Threshold value functions]
As a direct application of \Cref{cor:infimal-convolution}, the infimal convolution between the Euclidean norm and its square (both satisfying NNCC) satisfies the NNCC condition: for $\varepsilon > 0$ on $X = Y = \R^n$,
    \begin{equation}\label{EqSoftThresholding}
         \|\cdot\| \boxplus \frac{1}{2\varepsilon}\| \cdot \|^2(x-y) = \begin{cases}
            \frac{1}{2\varepsilon}\|x-y\|^2 \text{ if } \|x-y\|\leq {\varepsilon}  \\
             \|x-y\| - \frac{\varepsilon}{2}\text{ otherwise}.
        \end{cases}
    \end{equation}
    The $L^1$ norm is often used as a regularization promoting sparsity in inverse problems or statistics, in particular in the Lasso method, see \cite{tibshirani96regression}.
    These methods use the argmin of the infimal convolution between the $L^1$ norm and the squared Euclidean norm. The corresponding cost is the sum of the previous cost in Formula \eqref{EqSoftThresholding} over all coordinates. By stability to products, this cost also satisfies the NNCC property.
\end{example}

\begin{remark}[Approximation of lower semi-continuous translation invariant costs]\label{RemarkTICosts}
A direct consequence of this property is that any lower semi-continuous cost of the form $c(x- y)$ with $c: \R^n \to [0,+\infty]$ that is NNCC can be approximated by continuous costs that are NNCC via the usual approximation $c \boxplus \frac 1{2\varepsilon} \| \cdot- \cdot\|^2$ for $\varepsilon \to 0$. 
\end{remark}

\subsection{NNCC metric spaces are positively curved} \label{sec:NNCC-PC}

When $M$ is a smooth Riemannian manifold and $d$ denotes the Riemannian distance, Loeper \cite{loeper2009regularity} showed that if the cost $c(x,y)=d^2(x,y)$ satisfies the MTW condition \eqref{eq:MTW-condition}, then $M$ necessarily has nonnegative sectional curvature. Therefore, in a smooth Riemannian setting nonnegative cross-curvature, stronger than the MTW condition, always implies nonnegative sectional curvature. In this section we prove an analogue of this result in metric geometry. Let us start with some definitions.
Given a complete metric space $(X,d)$, a \emph{curve} is a continuous map $\gamma\colon [0,1]\rightarrow X$. Its length is defined as
\[
\operatorname{Length}(\gamma)\coloneqq \sup \sum_{k=1}^N d(\gamma(t_k),\gamma(t_{k-1}))\,,
\]
where the supremum is taken over all $N\geq 1$ and all partitions $0=t_0 <t_1 <\dots<t_N =1$. Let $x_0=\gamma(0)$ and $x_1=\gamma(1)$ (we say that $\gamma$ connects $x_0$ to $x_1$). By the triangle inequality it holds that $d(x_0,x_1)\le \operatorname{Length}(\gamma)$.
The metric space $(X,d)$ is then called a \emph{geodesic space} if for any two points $x_0,x_1\in X$,
\[
d(x_0,x_1)=\min_\gamma \operatorname{Length}(\gamma)\,,
\]
where the minimization is taken over all curves connecting $x_0$ to $x_1$. The fact that the minimum is actually attained is important and the minimizer is called a (length-minimizing) geodesic. By convention, geodesics are always considered to have constant speed parametrization, namely
\[
d(\gamma(s),\gamma(r))=\lvert s-r\rvert\,d(x_0,x_1)\,,
\]
for any $s,r\in[0,1]$.

Alexandrov introduced a synthetic notion of curvature bounds for general (non-smooth) metric spaces \cite{aleksandrov1951theorem}. This notion is a generalization of lower bounds for the sectional curvature on Riemannian manifolds. It is based on comparing (appropriately defined) triangles in $(X,d)$ with reference triangles in a 2-dimensional Riemannian manifold with constant curvature.
Defining these curvature bounds can be done in several equivalent ways. To give a notion of nonnegative curvature, the reference Riemannian manifold is taken to be the flat $\R^2$ space. In that case a geodesic space $(X,d)$ is called a positively curved (PC) space (in the sense of Alexandrov) if, for any point $y\in X$ and any geodesic $\gamma$,
\begin{equation}\label{eq:PC-inequality}
    \forall s\in[0,1],\quad d^2(\gamma(s),y)\geq (1-s)d^2(\gamma(0),y)+s\,d^2(\gamma(1),y)-s(1-s)d^2(\gamma(0),\gamma(1))\,.
\end{equation}
Note that if $(X,d)$ is a Hilbert space equality holds in \eqref{eq:PC-inequality}, in which case this relation is an alternative representation of the parallelogram law. Analogously, a geodesic space $(X,d)$ is called a nonpositively curved (NPC) space (in the sense of Alexandrov) if for any geodesic $\gamma$ and any $y\in X$, \eqref{eq:PC-inequality} holds with a reverse inequality. We refer to \cite{burago2001course} for more details. A non-trivial example of a PC space is the $2$-Wasserstein space $(\Pc_2(\R^n),W_2)$ \cite{ambrosio2005gradient} (see also Section \ref{sec:sncc_proba}).
The main result of this section is the following.

\begin{proposition}[NNCC$\implies$PC]\label{thm:NNCC-implies-PC}
    Consider a geodesic space $(X,d)$ such that $(X\times X,d^2)$ is an NNCC space. Then $(X,d)$ is a positively curved space in the sense of Alexandrov.
\end{proposition}

The proof of this proposition relies on the following lemma which is also of independent interest since it describes the behavior of variational c-segments in a metric space.

\begin{lemma}\label{lemma:distance-1-convex-c-segments}
    Let $(X,d)$ be any metric space, $\yb\in X$ and consider a variational c-segment $(s\in[0,1])\mapsto(\rmx(s),\yb)$ in $(X\times X,d^2)$. Then for all $s\in[0,1]$,
    \begin{equation}\label{eq:distance-1-convex-c-segments}
        d^2(\rmx(s),\yb)\leq (1-s)d^2(\rmx(0),\yb)+sd^2(\rmx(1),\yb)-s(1-s)d^2(\rmx(0),\rmx(1))\,.
    \end{equation}
\end{lemma}
\begin{proof}
    The NNCC inequality says that     
    for all $y\in X$, 
    \begin{equation}\label{eq:proof-lemma:distance-1-convex-c-segments}
        d^2(\rmx(s),\yb)-d^2(\rmx(s),y)\leq (1-s)[d^2(x_0,\yb)-d^2(x_0,y)]+s[d^2(x_1,\yb)-d^2(x_1,y)]\,,
    \end{equation}
    where we denote $x_0=\rmx(0)$ and $x_1=\rmx(1)$. 
    By the triangle and Young's inequalities, we may write for any $y\in X$,
    \begin{align*}
        s(1-s)d^2(x_0,x_1) &\leq s(1-s) \big(d(x_0,y) + d(y,x_1))^2\\
        &\leq s(1-s) \big((1+\varepsilon)d^2(x_0,y)+ (1+\varepsilon^{-1}) d^2(x_1,y)\big)\\
        &= (1-s) d^2(x_0,y) + s \,d^2(x_1,y)\,,
    \end{align*}
    with $\varepsilon=1/s-1$. Then taking $y=\rmx(s)$ in \eqref{eq:proof-lemma:distance-1-convex-c-segments} gives the desired inequality. 
\end{proof}

\begin{proof}[Proof of \Cref{thm:NNCC-implies-PC}]
    Let us prove that every geodesic satisfies the PC inequality \eqref{eq:PC-inequality}. Fix $x_0,x_1\in X$, a geodesic $\gamma\colon[0,1]\to X$ joining $x_0$ to $x_1$, and fix $t\in[0,1]$.
    Since $(X\times X,d^2)$ is NNCC there exists a variational c-segment $s\mapsto (\rmx(s),\gamma(t))$ such that $\rmx(0)=x_0$ and $\rmx(1)=x_1$. 
    
    Let us first show that at $s=t$ we have $\rmx(t)=\gamma(t)$. By \Cref{lemma:distance-1-convex-c-segments} we know that at $s=t$,
    \begin{equation}\label{eq:proof-prop-pc-2}
    d^2(\rmx(t),\gamma(t))\leq (1-t)d^2(x_0,\gamma(t)) +t\,d^2(x_1,\gamma(t))-t(1-t)d^2(x_0,x_1)\,.
    \end{equation}
    Since $\gamma$ is a geodesic, we have 
    \begin{equation}\label{eq:proof-prop-pc-1}
       \begin{aligned}
        d^2(x_0,\gamma(t))&=t^2d^2(x_0,x_1),\\
        d^2(x_1,\gamma(t))&=(1-t)^2d^2(x_0,x_1)\,. 
       \end{aligned}
    \end{equation}
    This gives us $d^2(\rmx(t),\gamma(t))\leq 0$, i.e.\ $\rmx(t)=\gamma(t)$. We then use the NNCC inequality: for all $s\in[0,1]$,
    \begin{equation}
        \forall y\in Y,\quad d^2(\rmx(s),\gamma(t))-d^2(\rmx(s),y)\leq (1-s)[d^2(x_0,\gamma(t)) -d^2(x_0,y)]+s\,[d^2(x_1,\gamma(t))-d^2(x_1,y)]\,.
    \end{equation}
    Taking $s=t$, using \eqref{eq:proof-prop-pc-1} and $\rmx(t)=\gamma(t)$ we obtain the desired inequality,
    \begin{equation}\label{eq:proof-prop-pc-3}
    \forall y\in X,\quad d^2(\gamma(t),y)\geq (1-t)d^2(x_0,y) + t\, d^2(x_1,y) - t(1-t)d^2(x_0,x_1)\,.
    \end{equation}  
\end{proof}

\Cref{thm:NNCC-implies-PC} says that the existence of variational c-segments for \emph{any} base point $\yb$ implies positive curvature. Let us now prove a result in the converse direction: in a PC space, geodesic curves passing through a given base point are variational c-segments.

\begin{lemma}[PC inequality and NNCC inequality]\label{lemma:geodesics-are-vcs}
    Let $(X,d)$ be a geodesic space. Let $\gamma\colon[0,1]\to X$ be a geodesic and fix any $t\in[0,1]$. Then $\gamma$ satisfies the PC inequality \eqref{eq:PC-inequality} if and only if $s\mapsto(\gamma(s),\gamma(t))$ is a variational c-segment on $(X\times X,d^2)$, i.e.\ it satisfies the NNCC inequality \eqref{eq:NNCC-inequality}. 
\end{lemma}
\begin{proof}
    The path $s\mapsto(\gamma(s),\gamma(t))$ is a variational c-segment on $(X\times X,d^2)$ if and only if for all $s\in[0,1]$ and all $y\in X$,
    \begin{equation}\label{eq:prop:geodesics-are-vcs}
    d^2(\gamma(s),\gamma(t))-d^2(\gamma(s),y)\leq (1-s)[d^2(\gamma(0),\gamma(t))-d^2(\gamma(0),y)]+s\,[d^2(\gamma(1),\gamma(t))-d^2(\gamma(1),y)].
    \end{equation}
    Since $\gamma$ is a geodesic we have $d^2(\gamma(s),\gamma(t))=(s-t)^2d^2(\gamma(0),\gamma(1))$. Then \eqref{eq:prop:geodesics-are-vcs} simplifies into
    \[
    d^2(\gamma(s),y)\geq (1-s)d^2(\gamma(0),y)+s\,d^2(\gamma(1),y)-s(1-s)d^2(\gamma(0),\gamma(1)).
    \]
    This is precisely the PC inequality \eqref{eq:PC-inequality}.
\end{proof}

As a direct consequence we have:

\begin{proposition}[Geodesics are variational c-segments]\label{prop:geodesics-are-vcs}
    Let $(X,d)$ be a PC metric space. Let $\gamma\colon[0,1]\to X$ be a geodesic and fix any $t\in[0,1]$. Then $s\mapsto(\gamma(s),\gamma(t))$ is a variational c-segment on $(X\times X,d^2)$. 
\end{proposition}

\begin{remark}
    Let $(X,d)$ be a non-positively curved space (NPC space). Then, reversing inequalities in the proof of \Cref{lemma:geodesics-are-vcs} directly implies the following result: if $\gamma\colon[0,1]\to X$ is a geodesic and $t\in [0,1]$, then $s\mapsto (\gamma(s),\gamma(t))$ is a variational c-segment on $(X\times X,c=-d^2)$.
\end{remark}

\subsection{Stability under Gromov--Hausdorff convergence}\label{SecGH}

A recurrent benefit of synthetic formulations is stability under weak notions of convergence. One may think about pointwise convergence of convex functions, or closer to the present subject, stability under Gromov--Hausdorff convergence of nonnegative sectional curvature in the sense of Alexandrov \cite{burago2001course}, Ricci curvature lower bounds \cite{LottVillani2009,10.1007/s11511-006-0002-8,sturm_2006_ii}, or the MTW condition for the squared distance on Riemannian manifolds \cite{Villani_4th_order}. 
In this section, we show that when the cost is a function of a distance, our notion of an NNCC space is preserved under Gromov--Hausdorff convergence. While there are only a few known NNCC spaces, this stability result may be useful for building new examples. 
Note that the standard, differential definition of nonnegative cross-curvature (requiring that the cost is $C^4$) is not stable under such a notion of convergence.
We also note that the proof of \Cref{thm:GH} is short and elementary, in contrast to some results in the same spirit such as stability of Ricci curvature lower bounds under Gromov--Hausdorff convergence.

We define the Gromov--Hausdorff distance following Burago--Burago--Ivanov \cite[Section 7.3.3.]{burago2001course}. Given two sets $X$ and $Y$ we say that $R\subset X\times Y$ is a correspondence between $X$ and $Y$ if for each $x\in X$ there exists $y\in Y$ such that $(x,y)\in R$ and for each $y\in Y$ there exists $x\in X$ such that $(x,y)\in R$. Let $\calR(X,Y)$ denote the set of all correspondences between $X$ and $Y$. Let $\bX=(X,d_X)$ and $\bY=(Y,d_Y)$ denote two compact metric spaces. The Gromov--Hausdorff distance between them is defined by
\[
    \GH(\bX,\bY)=\inf_{R\in\calR(X,Y)}\sup\{\abs{d_X(x,x')-d_Y(y,y')} : x,x'\in X,y,y'\in Y,(x,y)\in R,(x',y')\in R\}\,.
\]
This quantity vanishes if and only if there exists an isometry between $X$ and $Y$, that is an invertible map $T\colon X \to Y$ such that $d_Y(T(x),T(x')) = d_X(x,x')$ for all $x,x' \in X$. The Gromov--Hausdorff distance is a true metric on the quotient by isometries of the space of compact metric spaces.
\begin{theorem} \label{thm:GH}
    Let $\bX_n=(X_n,d_n)$ be a sequence of compact metric spaces, let $f\colon [0,+\infty)\to [0,+\infty)$ be a locally Lipschitz function and define the costs $c_n(x,y)=f(d_n(x,y))$ on $X_n\times X_n$. Suppose that for each $n\geq 0$, $(X_n\times X_n,c_n)$ is an NNCC space. Suppose that $\bX_n$ converges to a compact metric space $\bZ=(Z,d_Z)$ in the Gromov--Hausdorff topology. Then $(Z\times Z,c_Z\coloneqq f(d_Z))$ is an NNCC space.
\end{theorem}
\begin{proof}
    Since $\GH(\bX_n,\bZ)$ converges to $0$ as $n\to\infty$ there exists a sequence $\eps_n>0$ converging to $0$ and for each $n\geq 0$ a correspondence $R_n\in\calR(X_n,Z)$ such that 
    \begin{equation}\label{YO}
        \forall (x,z)\in R_n,\forall(x',z')\in R_n,\quad \abs{d_n(x,x')-d_Z(z,z')} \leq \eps_n\,.
    \end{equation}
    Let us fix $\zb,z,\wb\in Z$, which will play the respective roles of the starting point, ending point and base point of a variational c-segment. For each $n\geq 0$, since $R_n$ is a correspondence between $X_n$ and $Z$, there exist $\xb_n,x_n,\yb_n\in X_n$ such that 
    \[
    (\xb_n,\zb)\in R_n,\quad (x_n,z)\in R_n,\quad (\yb_n,\wb)\in R_n\,.
    \]
    Fix $0<s<1$. For each $n\geq 0$, since $(X_n\times X_n,c_n)$ is an NNCC space there exists a point $\rmx_n(s)\in X_n$ such that 
    \begin{equation}\label{YOO}
        \forall y\in X_n,\quad c_n(\rmx_n(s),\yb_n)-c_n(\rmx_n(s),y)\leq (1-s)[c_n(\xb_n,\yb_n)-c_n(\xb_n,y)] + s[c_n(x_n,\yb_n)-c_n(x_n,y)]\,.
    \end{equation}
    For each $n$ we can then find a point $\rmz_n(s)\in Z$ such that $(\rmx_n(s),\rmz_n(s))\in R_n$. Since $(Z,d_Z)$ is compact there exists $\rmz(s)\in Z$ such that a subsequence of $(\rmz_n(s))_n$ converges to $\rmz(s)$ as $n\to\infty$.     
    Note that throughout all these operations $s$ is kept fixed. 
    
    To conclude we want to pass to the limit in \eqref{YOO}. Fix $w\in Z$ and let $y_n\in X_n$ such that $(y_n,w)\in R_n$. By \eqref{YO} we have 
    \begin{align*}
        \abs{d_n(\rmx_n(s),\yb_n)-d_Z(\rmz_n(s),\wb)}\leq \eps_n\,,\\
        \abs{d_n(\rmx_n(s),y_n)-d_Z(\rmz_n(s),w)}\leq \eps_n\,,\\
        \abs{d_n(\xb_n,\yb_n)-d_Z(\zb,\wb)}\leq \eps_n\,,\\
        \abs{d_n(\xb_n,y_n)-d_Z(\zb,w)}\leq \eps_n\,,\\
        \abs{d_n(x_n,\yb_n)-d_Z(z,\wb)}\leq \eps_n\,,\\
        \abs{d_n(x_n,y_n)-d_Z(z,w)}\leq \eps_n\,,
    \end{align*}
    since all the above points are in respective correspondence for $R_n$. This implies 
    \[
    \abs{c_n(\rmx_n(s),\yb_n)-c_Z(\rmz_n(s),\wb)}=\abs{f(d_n(\rmx_n(s),\yb_n))-f(d_Z(\rmz_n(s),\wb))}\leq C\eps_n\,,
    \]
    where $C$ is a constant that depends on the diameters of $(X_n,d_n)$ and $(Z,d_Z)$ and the local Lipschitz constant of $f$. Combined with \eqref{YOO} we find 
    \begin{equation*}
        c_Z(\rmz_n(s),\wb)-c_Z(\rmz_n(s),w)\leq (1-s)[c_Z(\zb,\wb)-c_Z(\zb,w)] + s[c_Z(z,\wb)-c_Z(z,w)]+C\eps_n\,.
    \end{equation*}
    As $n\to\infty$ we obtain 
    \begin{equation*}
        c_Z(\rmz(s),\wb)-c_Z(\rmz(s),w)\leq (1-s)[c_Z(\zb,\wb)-c_Z(\zb,w)] + s[c_Z(z,\wb)-c_Z(z,w)]\,.
    \end{equation*}
\end{proof}

\subsection{The MTW condition and the Loeper maximum principle} \label{sec:sncc-discussion}

The goal of this section is to provide points of comparison between our synthetic notion of nonnegative cross-curvature and the synthetic version of the MTW condition known as Loeper's maximum principle. Compare: \Cref{thm:loeper-formulation}\ref*{enum:LMP-c-segment} and \Cref{thm:KM_characterization}, \Cref{lem:auto-csegment-mtw} and \Cref{lemma:auto-csegment}, \eqref{eq:loeper-formulation-2} and the NNCC inequality, and anticipating the next section, \Cref{lemma:subdiff-connected} and \Cref{prop:NNCC-contact-sets}. Because we use cost functions that may take $\pm\infty$ values, in certain cases some care has to be taken to properly define c-transforms, c-subdifferentials, etc. Appendix \ref{sec:app:infinite} contains the necessary background material to handle infinite values, and the main definitions will be recalled in the main text.

Given two $n$-dimensional smooth manifolds $X$ and $Y$ and a nondegenerate cost function $c\in C^4(X\times Y)$ (see the regularity assumptions \eqref{eq:ass-C4}--\eqref{eq:ass-non-deg}), the MTW condition is given by: for all $(x,y)\in X\times Y$ and all $\xi\in T_xX$, $\eta\in T_yY$,
\begin{equation}\label{eq:MTW-condition}
    \Sc {(x,y)}(\xi,\eta)\geq 0\quad\text{whenever}\quad\nabla^2_{xy}c(x,y)(\xi,\eta)=0\,.
\end{equation}
Here $\Sc$ is the MTW tensor, defined in \Cref{def:MTW-tensor}.
In other words, the MTW condition requires nonnegativity of $\Sc$ on pairs of vectors $(\xi,\eta)$ that are orthogonal for the Kim--McCann metric $\nabla^2_{xy}c$. In contrast nonnegative cross-curvature demands $\Sc\geq 0$ for all tangent vectors $(\xi,\eta)$. 

While \eqref{eq:MTW-condition} is a differential condition based on $\Sc$, thus needing four derivatives on the cost function, Loeper proved a series of equivalent characterizations that turn out to have lower regularity requirements \cite[Theorems 3.1 and 3.2, Proposition 2.11]{loeper2009regularity}. We state here some of them, slightly reformulated and for simplicity under the full set of assumptions \eqref{eq:ass-C4}--\eqref{eq:ass-biconvex} although not every assumption is always needed. See also \cite[Section 2.5]{TrudingerWang_rational_mech_2009}, \cite[Theorem 3.1]{kim2010continuity}.

\begin{theorem} \label{thm:loeper-formulation}
    Let $X,Y$ be two $n$-dimensional manifolds and suppose that $(X,Y,c)$ is a triple satisfying \eqref{eq:ass-C4}--\eqref{eq:ass-biconvex}. Then the following statements are equivalent. 
    \begin{enumerate}[(i)]
    
        \item\label{enum:LMP-curve} For any $(x_0,x_1,\yb)\in X\times X\times Y,$ there exists a continuous curve $\rmx$ joining $x_0$ to $x_1$ such that for all $s\in[0,1]$,
        \begin{equation}\label{eq:loeper-formulation-2}
        \forall y\in Y, \quad c(\rmx(s),\yb)-c(\rmx(s),y)\leq \max\{c(x_0,\yb)-c(x_0,y), c(x_1,\yb)-c(x_1,y)\}.
        \end{equation}
        
        \item\label{enum:LMP-c-segment} For any $(x_0,x_1,\yb)\in X\times X\times Y,$ denoting by $\rmx$ the c-segment joining $x_0$ to $x_1$ with base $\yb$ we have for all $s\in[0,1]$,
        \begin{equation}\label{eq:loeper-formulation}
        \forall y\in Y, \quad c(\rmx(s),\yb)-c(\rmx(s),y)\leq \max\{c(x_0,\yb)-c(x_0,y), c(x_1,\yb)-c(x_1,y)\}.
        \end{equation}
            
        \item\label{enum:LMP-MTW-condition} $c$ satisfies the MTW condition \eqref{eq:MTW-condition}.
    \end{enumerate}
\end{theorem}

Inequality~\eqref{eq:loeper-formulation}, when applied to all c-segments contained in a given c-segment, implies that the function $s\mapsto c(\rmx(s),\yb)-c(\rmx(s),y)$ is \emph{quasi-convex}, i.e.\ its lower levelsets are convex. This can be compared with the Kim--McCann condition in \Cref{thm:KM_characterization} that characterizes nonnegative cross-curvature and involves the \emph{convexity} of the same function $s\mapsto c(\rmx(s),\yb)-c(\rmx(s),y)$. 

Condition \ref*{enum:LMP-curve} is the synthetic formulation \smtw{} introduced in Section \ref{sec:sncc-def}. Condition \ref*{enum:LMP-c-segment} is often the one called Loeper's maximum principle. 
The equivalence of \ref*{enum:LMP-curve} and \ref*{enum:LMP-c-segment} in \Cref{thm:loeper-formulation} says that the continuous curve in \eqref{eq:loeper-formulation-2} can always be taken to be a c-segment. In fact Loeper's proof of \cite[Proposition 2.11]{loeper2009regularity} shows that \eqref{eq:loeper-formulation-2} encodes c-segments, up to time reparametrization.  We state a version of this result here and give it a proof in \Cref{AppendixOtherProofs} for the reader's convenience. Note the similarity with \Cref{lemma:auto-csegment}. 

\begin{lemma}[Curves satisfying \smtw{} are automatically reparametrizations of c-segments]\label{lem:auto-csegment-mtw}

    \leavevmode
    
    \noindent
    Let $X$ and $Y$ be two $n$-dimensional smooth manifolds and let $c\in C^1(X\times Y)$. Let $(s\in [0,1]) \mapsto \rmx(s)$ be a smooth curve in $X$ and let $\yb\in Y$ be such that, for all $y\in Y$, 
    \begin{equation}\label{eq:lemma:auto-csegment-mtw}
        c(\rmx(s),\yb)-c(\rmx(s),y)\leq \max\{c(\rmx(0),\yb)-c(\rmx(0),y), c(\rmx(1),\yb)-c(\rmx(1),y)\}.
    \end{equation}
    Then, there exists a continuous function $f\colon[0,1]\to[0,1]$, with $f(0) = 0$ and $f(1) = 1$, such that $\nabla_yc(\rmx(s),\yb) = (1- f(s))\nabla_y c(\rmx(0),\yb) + f(s)\nabla_yc(\rmx(1),\yb)$ for all $s\in[0,1]$.
    \end{lemma}

Because the MTW condition originates in the problem of the regularity of optimal transport maps, conditions \ref*{enum:LMP-curve} and \ref*{enum:LMP-c-segment} in \Cref{thm:loeper-formulation} are sometimes stated in a different but equivalent form, connectedness (for \ref*{enum:LMP-curve}) and c-convexity (for \ref*{enum:LMP-c-segment}) of the \emph{c-subdifferential}. We make this link explicit here under great generality, and allowing $\pm\infty$ values. Given two arbitrary sets $X$ and $Y$ and an arbitrary cost $c\colon X\times Y\to [-\infty,+\infty]$, the \emph{c-transform} of the function $\phi\colon Y\to [-\infty,+\infty]$ is defined by 
\begin{equation}\label{eq:def-c-transform-main}
\phi^c(x)=\sup\{\beta\in\R : \phi(\cdot)\leq c(x,\cdot)-\beta\} = \inf_{y\in Y} c(x,y)-\phi(y)\,,
\end{equation}
with the rule $(+\infty)-(+\infty)=(-\infty)-(-\infty)=+\infty$ inside the infimum, see \eqref{eq:def-c-transform} in Appendix \ref{sec:app:infinite}. We similarly define the c-transform of a function $\psi\colon X\to [-\infty,+\infty]$ which we still denote by $\psi^c$. 
A function $\phi\colon Y\to [-\infty,+\infty]$ is said to be \emph{c-concave} if there exists $\psi\colon X\to[-\infty,+\infty]$ such that $\phi=\psi^c$. 
We denote by $\Phi^c$ the set of c-concave functions on $Y$. 
We also recall that the c-subdifferential $\partial^c\phi(\yb)$ of $\phi\colon Y\to [-\infty,+\infty]$ at a point $\yb\in Y$ where $\phi(\yb)$ is finite is the subset 
\begin{equation}\label{eq:def-c-subdiff-main}
\partial^c\phi(\yb)=\{x\in X : \text{$c(x,\yb)$ is finite and } \phi^c(x)+\phi(\yb)= c(x,\yb)\},
\end{equation}
and if $\phi(\yb)=\pm\infty$ then $\partial^c\phi(\yb)=\emptyset$, see \eqref{eq:def-c-subdifferential-y}.

\begin{lemma}\label{lemma:subdiff-connected}
    Let $X$ and $Y$ be two arbitrary set and let $c\colon X\times Y\to [-\infty,+\infty]$ be an arbitrary function. 
    Given $x_0,x_1,\tilde x\in X$ and $\yb\in Y$ such that $c(x_0,\yb)$ and $c(x_1,\yb)$ are finite, the following statements are equivalent.

    \begin{enumerate}[(i)]
        \item\label{enum:LMP-csubdiff-i} 
        $c(\tilde x,\yb)$ is finite and 
        \[
        \forall y\in Y,\,\, c(\tilde x,\yb)-c(\tilde x,y)\leq \max\{c(x_0,\yb)-c(x_0,y), c(x_1,\yb)-c(x_1,y)\}.
        \]
        
        \item\label{enum:LMP-csubdiff-ii} For all $\phi\in\Phi^c$, $x_0\in\partial^c\phi(\yb)$ and $x_1\in \partial^c\phi(\yb)$ necessarily implies $\tilde x\in \partial^c\phi(\yb)$.
        
        \item\label{enum:LMP-csubdiff-iii} For all $\phi\colon Y\to [-\infty,+\infty]$, $x_0\in\partial^c\phi(\yb)$ and $x_1\in \partial^c\phi(\yb)$ necessarily implies $\tilde x\in \partial^c\phi(\yb)$.

    \end{enumerate}
    As a direct consequence suppose that $X$ is a topological space. Then the following statements are equivalent. 
    \begin{enumerate}[(i),resume]
        \item\label{enum:LMP-csubdiff-iv} $(X\times Y,c)$ satisfies condition \smtw{}. 
        
        \item\label{enum:LMP-csubdiff-v} For every c-concave function $\phi\colon Y\to [-\infty,+\infty]$ and every $\yb\in Y,$ $\partial^c\phi(\yb)$ is pathwise connected or empty. 
        
        \item\label{enum:LMP-csubdiff-vi} For every function $\phi\colon Y\to [-\infty,+\infty]$ and every $\yb\in Y,$ $\partial^c\phi(\yb)$ is pathwise connected or empty. 
    \end{enumerate}
\end{lemma}

Since the proof of \Cref{lemma:subdiff-connected} is similar to arguments found in the existing literature, we postpone it to \Cref{AppendixOtherProofs}. The equivalence of \ref*{enum:LMP-csubdiff-iv} and \ref*{enum:LMP-csubdiff-v} is well-known to specialists (in smooth settings), see \cite[Proposition 2.11]{loeper2009regularity} and its proof, \cite[Section 2.5]{TrudingerWang_rational_mech_2009}, \cite[Proposition 12.15]{villani2009optimal}. It is stated here under more general assumptions.

\subsection{The Figalli--Kim--McCann characterization} \label{sec:Figalli-Kim-McCann}

In~\cite{Figalli_Kim_McCann_screening2011}, Figalli, Kim and McCann characterized nonnegative cross-curvature in a smooth and compact setting in terms of the convexity of the set $\Phi^c$ of c-concave functions. In this section we derive a related characterization of the NNCC inequality and NNCC spaces. Here is their result, stated under a variant of their assumptions.

\begin{theorem}[{\cite[Theorem 3.2]{Figalli_Kim_McCann_screening2011}}] \label{thm:FigalliKimMcCann}
    Let $X$ and $Y$ be two open bounded subsets of $\Rn$, suppose that $(X,Y,c)$ satisfies \eqref{eq:ass-C4}--\eqref{eq:ass-biconvex} and that $c$ is continuous on the closure $\overline X\times\overline Y$. Then $c$ has nonnegative cross-curvature if and only if the set $\Phi^c$ is convex, in the sense that for any two $\phi_0,\phi_1\in\Phi^c$ and any $s\in [0,1]$, the function $\phi_s\coloneqq (1-s)\phi_0+s\phi_1$ is c-concave.
\end{theorem}

One key feature of the assumptions in \Cref{thm:FigalliKimMcCann} is that compactness of $X$ and $Y$ combined with continuity of the cost function guarantees c-concave functions $\phi$ to always have nonempty c-subdifferentials (see \eqref{eq:def-c-subdiff-main}), in the sense that
\begin{equation} \label{eq:condition-nonempty-subdiff}
    \forall y\in Y,\quad \partial^c\phi(y)\neq\emptyset\,.
\end{equation}
Moreover, condition \eqref{eq:condition-nonempty-subdiff} is by itself always stronger than c-concavity. Indeed if $x\in\partial^c\phi(y)$ then 
\[
c(x,y)-\phi^c(x)\leq \phi(y)\,,
\]
with every quantity finite, which in turn gives $\phi^{cc}(y)\leq \phi(y)$, thus in fact $\phi^{cc}(y)= \phi(y)$. This means that \eqref{eq:condition-nonempty-subdiff} implies $\phi=\phi^{cc}$. When $X$ and $Y$ are more general sets or when considering a more general cost function, condition \eqref{eq:condition-nonempty-subdiff} may not be automatically satisfied by a c-concave function $\phi$. 
It turns out that more than c-concavity, it is condition \eqref{eq:condition-nonempty-subdiff} that we will need here. Since condition \eqref{eq:condition-nonempty-subdiff} forces $\phi$ to be finite-valued, we relax it to allow infinite values and only require c-subdifferentials to be nonempty at points where $\phi$ is finite.
Given an arbitrary set $Y$, we therefore define
\begin{equation}\label{eq:def-Phic0}
    \widetilde\Phi^c=\{\phi\colon Y\to[-\infty,+\infty]: \,\,\partial^c\phi(y)\neq\emptyset\text{ at every point $y\in Y$ where $\phi$ is finite}\}.
\end{equation}
Note that contrary to \eqref{eq:condition-nonempty-subdiff}, in general the relaxed condition $\phi\in\widetilde\Phi^c$ doesn't imply (and is not implied by) c-concavity of $\phi$.
The next result then gives a characterization of NNCC spaces in terms of c-subdifferentials; it is an analogue of \Cref{lemma:subdiff-connected} for the MTW condition. Given $s\in (0,1)$, in order to properly define the convex combination $\phi_s$ of two functions $\phi_0,\phi_1\colon Y\to [-\infty,+\infty]$ that take infinite values, we use the hypograph,
\begin{equation}\label{eq:def-phi-s-hypo}
\hypo\phi_s \coloneqq \{(y,\beta)\in Y\times\R : \beta=(1-s)\beta_0+s\,\beta_1, (y,\beta_0)\in\hypo\phi_0,\,(y,\beta_1)\in\hypo\phi_1\},
\end{equation}
see also the use of hypographs in Appendix \ref{sec:app:infinite}. This says that  
\begin{equation}\label{eq:def-phi-s}
    \phi_s(y)\coloneqq (1-s)\phi_0(y)+s\,\phi_1(y)\,,
\end{equation}
with the rule $(+\infty)+(-\infty)=(-\infty)+(+\infty)=-\infty$ in the right-hand side.

\begin{proposition}\label{prop:NNCC-contact-sets}
    Let $X$ and $Y$ be two arbitrary sets and $c\colon X\times Y\to [-\infty,+\infty]$ an arbitrary function. Given $x_0,x_1,\tilde x\in X$ and $\yb\in Y$ such that $c(x_0,\yb)$, $c(x_1,\yb)$ are finite and given $s\in (0,1)$, the following statements are equivalent.
    \begin{enumerate}[(i)]
        \item\label{enum:NNCC-csubdiff-i} $c(\tilde x,\yb)$ is finite and 
        \[\forall y\in Y,\quad c(\tilde x,\yb)-c(\tilde x,y)\leq (1-s) [c(x_0,\yb)-c(x_0,y)]+ s [c(x_1,\yb)-c(x_1,y)]\,, 
        \]
        with undefined combinations in the right-hand side assigned the value $+\infty$.
        
        \item\label{enum:NNCC-csubdiff-ii} For all $\phi_0,\phi_1\in\widetilde\Phi^c$, $x_0\in\partial^c\phi_0(\yb)$ and $x_1\in \partial^c\phi_1(\yb)$ necessarily implies $\tilde x\in \partial^c\phi_s(\yb)$.
        
        \item\label{enum:NNCC-csubdiff-iii} For all $\phi_0,\phi_1\colon Y\to[-\infty,+\infty]$, $x_0\in\partial^c\phi_0(\yb)$ and $x_1\in \partial^c\phi_1(\yb)$ necessarily implies $\tilde x\in \partial^c\phi_s(\yb)$.
    \end{enumerate}
    Above and below, $\phi_s$ is always defined by \eqref{eq:def-phi-s}. As a direct consequence we have that the following statements are equivalent. 
    \begin{enumerate}[(i),resume]
        \item\label{enum:NNCC-csubdiff-iv} $(X\times Y,c)$ is an NNCC space. 
        \item\label{enum:NNCC-csubdiff-v} $\widetilde\Phi^c$ is convex in the sense that $\phi_0,\phi_1\in\widetilde\Phi^c$ implies $\phi_s\in\widetilde\Phi^c$ for every $0<s<1$. 
    \end{enumerate}
\end{proposition}

In our setting nonnegative cross-curvature is therefore characterized by the convexity of the set $\widetilde\Phi^c$. 
In the case where the cost takes only finite values, \ref*{enum:NNCC-csubdiff-v} can be replaced with convexity of the set 
\[
\Phi^c_0 = \{\phi\colon Y\to\R : \partial^c\phi(y)\neq\emptyset \text{ for every $y\in Y$}\}.
\]
Recall by the discussion after \eqref{eq:condition-nonempty-subdiff} that in particular $\Phi^c_0\subset\Phi^c$. This characterization of NNCC spaces is closer in spirit to the one of Figalli--Kim--McCann. 

\begin{proof}[Proof of \Cref{prop:NNCC-contact-sets}]
    Since the implication from \ref*{enum:NNCC-csubdiff-iii} to \ref*{enum:NNCC-csubdiff-ii} is obvious, we show that \ref*{enum:NNCC-csubdiff-i} implies \ref*{enum:NNCC-csubdiff-iii} and that \ref*{enum:NNCC-csubdiff-ii} implies \ref*{enum:NNCC-csubdiff-i}.

    \ref*{enum:NNCC-csubdiff-i}$\implies$\ref*{enum:NNCC-csubdiff-iii}. Let $\phi_0,\phi_1\colon Y\to[-\infty,+\infty]$ 
    be such that $\phi_0(\yb)$, $\phi_1(\yb)$ are finite and 
    \begin{equation}\label{eq:proof-prop:NNCC-contact-sets-1}
    \forall y\in Y,\quad\left\{\begin{aligned}
    c(x_0,\yb)-c(x_0,y)&\leq \phi_0(\yb)-\phi_0(y)\\
    c(x_1,\yb)-c(x_1,y)&\leq \phi_1(\yb)-\phi_1(y)\,.
    \end{aligned}\right.
    \end{equation}
    The above inequalities may involve infinities but never any undefined combination. We then want to add the first inequality multiplied by $1-s$ to the second inequality multiplied by $s$. If there are no undefined infinite combination we obtain using \ref*{enum:NNCC-csubdiff-i}
    \begin{align}
    \forall y\in Y,\quad c(\tilde x,\yb)-c(\tilde x,y)&\leq (1-s) [c(x_0,\yb)-c(x_0,y)]+ s [c(x_1,\yb)-c(x_1,y)]\label{eq:proof-prop:NNCC-contact-sets-2}\\
    &\leq (1-s)[\phi_0(\yb)-\phi_0(y)]+s\,[\phi_1(\yb)-\phi_1(y)]\\
    &\leq \phi_s(\yb)-\phi_s(y)\,,\label{eq:proof-prop:NNCC-contact-sets-4}
    \end{align}
    since $\phi_0(\yb)$ and $\phi_1(\yb)$ are finite. If the right-hand side of \eqref{eq:proof-prop:NNCC-contact-sets-2} contains an undefined infinite combination, then one of $c(x_0,y)$ or $c(x_1,y)$ must be $-\infty$. Without loss of generality say it is $c(x_0,y)$. Then \eqref{eq:proof-prop:NNCC-contact-sets-1} forces $\phi_0(y)=-\infty$. But then $\phi_s(y)=-\infty$ by construction of $\phi_s$, so that \eqref{eq:proof-prop:NNCC-contact-sets-4} holds. This proves that $\tilde x\in\partial^c\phi_s(\yb)$. 

    \ref*{enum:NNCC-csubdiff-ii}$\implies$\ref*{enum:NNCC-csubdiff-i}. Define $\phi_0(y)=c(x_0,y)$ and $\phi_1(y)=c(x_1,y)$. 
    These two functions belong to $\widetilde\Phi^c$ since for any $y\in Y$ such that $\phi_i(y)$ is finite we have $x_i\in\partial^c\phi_i(y)$ (for $i=0,1$). One may consider for instance the function $\psi_i(x)=0$ at $x=x_i$ and $-\infty$ elsewhere, which is sub-conjugate to $\phi_i$, satisfies $\phi_i=\psi_i^c$ and $\psi_i(x_i)+\phi_i(\yb)=c(x_i,\yb)$. In particular since $c(x_0,\yb),c(x_1,\yb)$ are finite we have that $x_i\in\partial^c\phi_i(\yb)$, $i=0,1$. By \ref*{enum:NNCC-csubdiff-ii} we have $\tilde x\in\partial\phi_s(\yb)$ which means that 
    $c(\tilde x,\yb)$ is finite and 
    $c(\tilde x,\yb)-c(\tilde x,y)\leq \phi_s(\yb)-\phi_s(y)$ for all $y\in Y$. This is precisely \ref*{enum:NNCC-csubdiff-i}. 
\end{proof}

\section{Cross-curvature of the Wasserstein space}\label{sec:sncc_proba}

In this section, we prove that the space of probability measures inherits the NNCC property of its base space when endowed with the corresponding transport cost.
The converse implication is also true.
This converse implication also holds for condition \smtw{},
however condition \smtw{} is not inherited by the space of probability measures in general: we provide a counterexample.

\subsection{Main results}

Throughout this section, we assume that the sets $X$ and $Y$ are Polish spaces and that the cost $c\colon X\times Y\to\R\cup\{+\infty\}$ is a lower semi-continuous function bounded from below.
We denote by $\calP(Z)$ the set of Borel probability measures on $Z$. Given $\mu\in \mathcal{P}(X)$ and $\nu\in\mathcal{P}(Y)$, the optimal transport problem of sending mass from the source $\mu$ to the target $\nu$ according to cost $c$ is given by 
\begin{equation}\label{eq:ot}
    \Tc_c(\mu,\nu)\coloneqq\inf_{\pi\in\Pi(\mu,\nu)} \int_{X\times Y} c(x,y) \, d \pi \,.
\end{equation}
Here the minimization is taken among all the admissible transport plans
\[
\Pi(\mu,\nu)=\left\{ \pi\in\Pc(X\times Y): (\p_1)_\# \pi=\mu, \,(\p_2)_\# \pi=\nu \right\}.
\]
The maps $\p_1\colon X\times Y\to X$ and $\p_2\colon X\times Y\to Y$ denote projections onto the first and second component respectively, and $(\cdot)_\#\pi $ represents the pushforward operation. 
We will refer to $(\calP(X)\times \calP(Y),\Tc_c)$ as a Wasserstein space. In the considered setting, existence of solutions for Problem \eqref{eq:ot} is guaranteed \cite[Theorem 1.7]{santambrogio2015optimal}. A solution $\pi$ that achieves a finite value $\Tc_c(\mu,\nu)$ will be called an optimal coupling of $(\mu,\nu)$. 
As for uniqueness of solutions, it requires further assumptions, see Section \ref{ssec:converse}.

On the Wasserstein space $(\calP(X)\times \calP(Y),\Tc_c)$, a curve $(\mu(s),\nu)$ is a variational c-segment if $\Tc_c(\mu(s),\nu)<+\infty$ and for any $\sigma \in \Pc(Y)$,
\begin{equation}\label{EqSNCCOnMeasures}
    \Tc_c(\mu(s),\nu) - \Tc_c(\mu(s),\sigma)\leq (1-s)[\Tc_c(\mu(0),\nu) - \Tc_c(\mu(0),\sigma)]+s[\Tc_c(\mu(1),\nu) - \Tc_c(\mu(1),\sigma)]\,.
\end{equation}
In general the transport costs $\Tc_c(\cdot,\sigma)$ may take the value $+\infty$ but never $-\infty$ since $\Tc_c$ is bounded from below under our assumptions. In order to systematically construct such paths $\mu(s)$ we propose a procedure that gives mass to variational c-segments.
It is given in the following definition. Here we have $\p_1(x_0,x_1,y)=x_0$, $\p_2(x_0,x_1,y)=x_1$ and $\p_3(x_0,x_1,y)=y$.

\begin{remark}[Notation]For the remainder of this section, all references to measurable functions or sets should be understood as referring to Borel measurability.
\end{remark}

\begin{definition}[Lifted c-segments]\label{def:lifted_csegment}
    Let $(s\in[0,1])\mapsto \mu(s)$ be a path in $\calP(X)$ and let $\nu\in\Pc(Y)$. We say that $s\mapsto (\mu(s),\nu)$ is a \emph{lifted c-segment} from $(X\times Y,c)$ if there exist a measurable set $\Gamma\subset X\times X\times Y,$ a $3$-plan $\gamma_\nu \in \mathcal{P}(X \times X \times Y)$, and a collection of measurable maps $\Lambda_s\colon \Gamma\to X$ ($s\in [0,1]$) such that
    \begin{enumerate}[(i)]                
        \item $\gamma_\nu$ is concentrated on $\Gamma$ (i.e. $\gamma_\nu(\Gamma)=1$), and $\mu(s)=(\Lambda_s)_\#\gamma_\nu$ and $\nu=(\p_3)_\#\gamma_\nu$ ; \label{enum:lcs-gamma}
        
        \item for $\gamma_\nu$-almost every $(x_0,x_1,y)\in\Gamma,$ we have: $\Lambda_0(x_0,x_1,y)=x_0$, $\Lambda_1(x_0,x_1,y)=x_1$ and $s\mapsto (\Lambda_s(x_0,x_1,y),y)$ is a variational c-segment on $(X\times Y,c)$ ; \label{enum:lcs-Lambda}
        
        \item$\pi_0\coloneqq (\p_1,\p_3)_\# \gamma_\nu$ and $\pi_1\coloneqq (\p_2,\p_3)_\# \gamma_\nu$ are respective optimal couplings of $(\mu(0),\nu)$ and $(\mu(1),\nu)$ for the cost $c$.\label{enum:lcs-pi}
    \end{enumerate}
\end{definition}

In other words we assume that we have at our disposal a class of variational c-segments, each depending measurably on its endpoints collected in the set $\Gamma$. The definition is flexible in the choice of $\Gamma$, in particular it does not necessarily require $(X\times Y,c)$ to be an NNCC space. For instance with $\Gamma=\{(x_0,x_1,x_0)\}$ where the base point is taken to match the initial point we may lift geodesics in a PC metric space, see \Cref{prop:geodesics-are-vcs}. After that the variational c-segments are weighted according to $\gamma_\nu$, through their endpoints, and combined into a lifted c-segment. 

The above definition may be straightforwardly adapted to lift \snccconv{}-variational c-segments (respectively, \smtw{}-variational c-segments). This simply requires asking the map $s\mapsto (\Lambda_s(x_0,x_1,y),y)$ to be an \snccconv{}-variational c-segments on $(X\times Y,c)$ (respectively a \smtw{}-variational c-segment). To avoid confusion, we call this curve \snccconv{}-lifted c-segment (respectively, \smtw{}-lifted c-segment). Since this construction is not the main focus of this work but is needed in Section \ref{ssec:converse}, we develop it in Appendix \ref{sec:snccconv-lift-v2}.

Let us now give some conditions that guarantee the existence of a lifted c-segment between $(\mu_0,\nu)$ and $(\mu_1,\nu)$ such that $\Tc_c(\mu_0,\nu)<+\infty$, $\Tc_c(\mu_1,\nu)<+\infty$. First, optimal couplings $\pi_0$ and $\pi_1$ in point \ref*{enum:lcs-pi} always exist under our assumptions. Second by the definition of variational c-segments we must have $c(x_0,\yb)<+\infty$, $c(x_1,\yb)<+\infty$ for any $(x_0,x_1,\yb)\in\Gamma$. The largest possible set $\Gamma$ is thus
\begin{equation}\label{eq:maximal-Gamma}
    \Gamma=\{(x_0,x_1,\yb)\in X\times X\times Y : c(x_0,\yb)<+\infty,\,c(x_1,\yb)<+\infty\}.
\end{equation}
Since $\pi_0$ and $\pi_1$ share a common marginal $\nu$, it is always possible to further couple these into a $3$-plan $\gamma_\nu$ that is concentrated on $\Gamma$, for example by ``gluing'' along $\nu$ \cite[Lemma 5.3.2]{ambrosio2005gradient}. To sum up, we may always find $\gamma_\nu$ satisfying \ref*{enum:lcs-pi}. Let us now consider point \ref*{enum:lcs-Lambda}. It asks that variational c-segments may be obtained through a \emph{measurable} map $\Lambda_s$. The next result shows this added measurability requirement is not an issue.

\begin{lemma}\label{lemma:continuous-cost-measurable-NNCC}
    Let $X$ and $Y$ be Polish spaces and $c\colon X\times Y\to\R\cup\{+\infty\}$ a lower semi-continuous function bounded from below. Define $\Gamma$ by \eqref{eq:maximal-Gamma} and consider a probability measure $\gamma \in \Pc(X\times X \times Y)$ concentrated on $\Gamma$. If $(X\times Y,c)$ is an NNCC space then there exists a collection of \emph{measurable} maps $\Lambda_s\colon \Gamma\to X$ ($s\in [0,1]$) such that for $\gamma$-almost every $(x_0,x_1,y)\in \Gamma$, the path $s\mapsto (\Lambda_s(x_0,x_1,y),y)$ is a variational c-segment from $(x_0,y)$ to $(x_1,y)$.
\end{lemma}

See Appendix \ref{sec:appendix-proofs} for the proof.
All the considerations above prove the existence of lifted c-segments.

\begin{lemma}[Existence of lifted c-segments] \label{lemma:existence-lifted}
    Let $(X\times Y,c)$ be an NNCC space, where $X,Y$ are two Polish spaces and $c$ is lower semi-continuous and bounded from below.
    Then for any $\mu_0,\mu_1\in\Pc(X)$ and $\nu\in\Pc(Y)$ such that $\Tc_c(\mu_0,\nu)<+\infty$ and $\Tc_c(\mu_1,\nu)<+\infty$, there exists a lifted c-segment from $(X\times Y,c)$ between $(\mu_0,\nu)$ and $(\mu_1,\nu)$. 
\end{lemma}

Let us give a few words on the lack of uniqueness of lifted c-segments. In general, there are three sources of non-uniqueness:  the first one is in the choice of the optimal transport plans $\pi_i$ in \Cref{def:lifted_csegment}\ref*{enum:lcs-pi}; the second one is in how $\pi_0$ and $\pi_1$ are combined into a $3$-plan $\gamma_\nu$ (see Example \ref{ex:MTW_counterexample}); the last one is in the measurable selection $\Lambda_s$.
If optimal transport plans $\pi_i$ for $(\mu_i,\nu)$ are unique and induced by $\nu$-a.e. unique maps $T_i:Y\rightarrow X$, that is $\pi_i=[y\mapsto (T_i(y),y)]_\#\nu$, then the $3$-plan $\gamma_\nu$ is also unique and is given by $[y\mapsto(T_0(y),T_1(y),y)]_\#\nu$.
Uniqueness of the measurable selection requires instead uniqueness of variational c-segments on the base space.
This condition may be restrictive in general.
In Lemma \ref{lem:Tc-var-c-seg_nuac}, we will consider a setting where this holds true and lifted c-segments (and more generally variational c-segments in the Wasserstein space) are unique.

The main reason for us to introduce lifted c-segments is that they are always variational c-segments. To properly state this result, let us now introduce an assumption that ensures finiteness of certain integrals:

\begin{assumption}\label{hyp:c-bounded-above}
    The function $x\mapsto\inf_{y\in Y}c(x,y)$ is uniformly bounded above on $X$. 
\end{assumption}
\noindent
This assumption is often satisfied in practice, for instance for nonnegative costs on $X \times X$ that vanish along the diagonal. We then have:

\begin{proposition}\label{prop:var_lifted_c-seg}
    Let $X$ and $Y$ be Polish spaces and $c\colon X\times Y\to\R\cup\{+\infty\}$ a lower semi-continuous function bounded from below satisfying \Cref{hyp:c-bounded-above}.
    If $s\mapsto(\mu(s),\nu)$ is a lifted c-segment from $(X\times Y,c)$, then it is a variational c-segment on $(\calP(X)\times \calP(Y),\Tc_c)$.
\end{proposition}

Before proving \Cref{prop:var_lifted_c-seg}, we state the main result of this section which is a direct consequence of \Cref{lemma:existence-lifted} and \Cref{prop:var_lifted_c-seg}.

\begin{theorem}\label{ThMainTheorem}
    Let $X$ and $Y$ be Polish spaces and $c\colon X\times Y\to\R\cup\{+\infty\}$ a lower semi-continuous cost function bounded from below satisfying \Cref{hyp:c-bounded-above}. Suppose that $(X\times Y,c)$ is an NNCC space.
    Then, the Wasserstein space $(\Pc(X)\times\Pc(Y),\Tc_c)$ is an NNCC space.
\end{theorem}

Let us now prove \Cref{prop:var_lifted_c-seg}. We start with a simple but important coupling result, whose proof is deferred to \Cref{AppendixOtherProofs}.

\begin{lemma}[Coupling extension]\label{lemma:coupling-extension}
    Let $X,Y,\Gamma$ be Polish spaces, $\mu\in\calP(X)$, $\sigma\in\calP(Y)$ and $\gamma\in \Pc(\Gamma)$. Let $\pi\in\Pi(\mu,\sigma)$ and  $\Lambda\colon \Gamma\to X$ be a measurable map such that $\Lambda_\#\gamma=\mu$. Then, there exists $\tilde\gamma\in\Pi(\gamma,\sigma)$ such that $[(w,y)\mapsto(\Lambda(w),y)]_\#\tilde\gamma=\pi$.
\end{lemma}

In other words, if $\mu$ is a pushforward measure obtained from another measure $\gamma$, and given another measure $\sigma$, any coupling $\pi$ of $(\mu,\sigma)$ can be seen as a coupling of $(\gamma,\sigma)$.

\begin{proof}[Proof of \Cref{prop:var_lifted_c-seg}]
    Let $s\mapsto(\mu(s),\nu)$ be a lifted c-segment induced by $(\Gamma,\Lambda_s,\gamma_\nu)$, as in \Cref{def:lifted_csegment}. Set $\mu_0=\mu(0)$, $\mu_1=\mu(1)$ and define $\pi(s)=[(x_0,x_1,y)\mapsto (\Lambda_s(x_0,x_1,y),y)]_\#\gamma_\nu$. Note that $\pi(s)$ is a coupling of $(\mu(s),\nu)$ and that $\pi(0)=(\p_1,\p_3)_\#\gamma_\nu$ and $\pi(1)=(\p_2,\p_3)_\#\gamma_\nu$ are optimal couplings of $(\mu_0,\nu)$ and $(\mu_1,\nu)$ respectively: in particular $\int c\,d\pi(0)$ and $\int c\,d\pi(1)$ are finite. 

    Fix $s\in [0,1]$ for the remainder of the proof. 
    We want to prove that $\Tc_c(\mu(s),\nu)<+\infty$ and that for every $\sigma\in\Pc(Y)$,
    \begin{equation*}\label{eq:proof-prop:var_lifted_c-seg-0}
    \Tc_c(\mu(s),\nu)-\Tc_c(\mu(s),\sigma) \leq (1-s)\Big(\Tc_c(\mu_0,\nu)-\Tc_c(\mu_0,\sigma)\Big)+s\Big(\Tc_c(\mu_1,\nu)-\Tc_c(\mu_1,\sigma)\Big).
    \end{equation*}
    Let $\sigma\in\calP(Y)$ such that $\Tc_c(\mu(s),\sigma)<+\infty$. Note that such a $\sigma$ exists by \Cref{hyp:c-bounded-above}: taking a map $T\colon X\to Y$ such that $c(x,T(x))\leq \inf_{y\in Y}c(x,y)+1$ and choosing it measurable (feasible by lower semi-continuity of $c$), then $\sigma\coloneqq T_\#\mu(s)$ satisfies $\Tc_c(\mu(s),\sigma)<+\infty$ by \Cref{hyp:c-bounded-above}. Finally let $\tilde\pi(s)$ be an optimal coupling of $(\mu(s),\sigma)$.
    Since $\mu(s)=(\Lambda_s)_\#\gamma_\nu$, we use \Cref{lemma:coupling-extension} to find $\gamma^4\in\calP(\Gamma\times Y)$ which we view as an element of $\calP(X\times X\times Y\times Y)$ (a $4$-plan) such that $(\p_1,\p_2,\p_3)_\#\gamma^4=\gamma_\nu$, $(\p_4)_\#\gamma^4=\sigma$ and such that 
    \begin{equation}\label{eq:proof-prop:var_lifted_c-seg-1}
        [(x_0,x_1,y,z)\mapsto (\Lambda_s(x_0,x_1,y),z)]_\#\gamma^4=\tilde\pi(s)\,.
    \end{equation}
    In case $\Gamma$ is not itself a Polish space it may always be extended to a Polish subspace of $X\times X\times Y$ and the map $\Lambda_s$ extended measurably since only its action on a set where $\gamma_\nu$ concentrates is relevant. The NNCC inequality gives for every $(x_0,x_1,y)\in\Gamma$ and $z\in Y$,
    \begin{equation}\label{eq:proof-prop:var_lifted_c-seg-11}
    c(\Lambda_s(x_0,x_1,y),y)-c(\Lambda_s(x_0,x_1,y),z)\leq (1-s)[c(x_0,y)-c(x_0,z)]+s[c(x_1,y)-c(x_1,z)]\,.
    \end{equation}    
    Since the cost function is bounded from below and only differences of costs appear in \eqref{eq:proof-prop:var_lifted_c-seg-11} we may assume from now on that $c\geq 0$. 
    We may then integrate \eqref{eq:proof-prop:var_lifted_c-seg-11} against $\gamma^4(dx_0,dx_1,dy,dz)$ (which is concentrated on $\Gamma\times Y$) and obtain 
    \begin{equation}\label{eq:proof-prop:var_lifted_c-seg-12}
    \int_{X\times Y}c\,d\pi(s)-\int_{X\times Y} c\,d\tilde\pi(s)\leq (1-s)\Big(\int_{X\times Y}c\,d\pi(0)-\int_{X\times Y}c\,d\tilde\pi(0)\Big)+s\Big(\int_{X\times Y}c\,d\pi(1)-\int_{X\times Y}c\,d\tilde\pi(1)\Big)\,.
    \end{equation}
    Here $\int_{X\times Y}c\,d\pi(0)$ and $\int_{X\times Y}c\,d\pi(1)$ are finite by definition of a lifted c-segment and $\int_{X\times Y}c\,d\tilde\pi(s)$ is finite by construction. Therefore \eqref{eq:proof-prop:var_lifted_c-seg-12} is well-defined. Moreover since $\pi(0)$, $\pi(1)$ and $\tilde\pi(s)$ are optimal \eqref{eq:proof-prop:var_lifted_c-seg-12} implies 
    \begin{equation}\label{eq:proof-prop:var_lifted_c-seg-2}
    \int_{X\times Y}c\,d\pi(s)-\Tc_c(\mu(s),\sigma)\leq (1-s)\Big(\Tc_c(\mu_0,\nu)-\Tc_c(\mu_0,\sigma)\Big)+s\Big(\Tc_c(\mu_1,\nu)-\Tc_c(\mu_1,\sigma)\Big).
    \end{equation}
    This gives us the desired inequality,
    \begin{equation}\label{eq:proof-prop:var_lifted_c-seg-3}
        \Tc_c(\mu(s),\nu)-\Tc_c(\mu(s),\sigma) \leq (1-s)\Big(\Tc_c(\mu_0,\nu)-\Tc_c(\mu_0,\sigma)\Big)+s\Big(\Tc_c(\mu_1,\nu)-\Tc_c(\mu_1,\sigma)\Big),
    \end{equation}
    together with finiteness of $\Tc_c(\mu(s),\nu)$. In conclusion we have established \eqref{eq:proof-prop:var_lifted_c-seg-3} when $\Tc_c(\mu(s),\sigma)$ is finite. Since \eqref{eq:proof-prop:var_lifted_c-seg-3} automatically holds when $\Tc_c(\mu(s),\sigma)=+\infty$, this finishes the proof. 
\end{proof}

\begin{remark}\label{rmk:pis-optimal}
    The proof of \Cref{prop:var_lifted_c-seg} also shows that the constructed plan $\pi(s)$ is an optimal coupling of the variational c-segment $(\mu(s),\nu)$, for each $s\in[0,1]$. This follows from taking $\sigma=\nu$ and noting that the right-hand side of \eqref{eq:proof-prop:var_lifted_c-seg-2} vanishes, so that $\int_{X\times Y}c\,d\pi(s)\leq \Tc_c(\mu(s),\nu)$. Therefore
    \[
    \Tc_c(\mu(s),\nu)=\int_{X\times Y}c\,d\pi(s)=\int_{X\times X\times Y} c(\Lambda_s(x_0,x_1,y),y)\,d\gamma_\nu(x_0,x_1,y)<+\infty\,.
    \]
\end{remark}

\begin{remark}
    If $X,Y$ are Polish spaces and $c$ is continuous and bounded, then $\Pc(X)$ and $\Pc(Y)$ are Polish spaces and the transport cost $\Tc_c$ is continuous with respect to the narrow topology and bounded. Then, thanks to Lemma \ref{lemma:continuous-cost-measurable-NNCC}, one can iterate the construction of lifted c-segments and lift the NNCC property from the Wasserstein space to the Wasserstein space on the Wasserstein space itself (and so on).
\end{remark}

The proof of Proposition \ref{prop:var_lifted_c-seg} can be easily adapted to show that whenever the base space $(X\times Y,c)$ satisfies the stronger condition \snccconv{}, this is inherited as well by the corresponding Wasserstein space $(\Pc(X)\times\Pc(Y),\Tc_c)$. 
However, this requires a stronger definition of lift which in turn requires stronger hypotheses on the base space.
For better readability, we state the following theorem here and postpone its proof to Appendix \ref{sec:snccconv-lift-v2}.

\begin{theorem}\label{thm:nncc-conv-lift}
    Let $(X\times Y,c)$ satisfy condition \snccconv{} for $X,Y$ two Polish spaces and $c$ a continuous function bounded from below satisfying Assumption \ref{hyp:c-bounded-above}.
    Suppose also that \snccconv{}-variational c-segments are continuous curves on $(0,1)$ with limits at $0$ and $1$.
    Then, the Wasserstein space $(\Pc(X)\times\Pc(Y),\Tc_c)$ satisfies condition \snccconv{} as well.
\end{theorem}

The most important cases of application of \Cref{ThMainTheorem} are squared $2$-Wasserstein distances.
When $(X,d)$ is a complete and separable metric space, the squared $2$-Wasserstein distance between $\mu$ and $\nu\in\calP(X)$ is defined by 
\begin{equation}\label{eq:2Wdistance}
    W_2^2(\mu,\nu) = \min_{\pi\in\Pi(\mu,\nu)}\int_{X\times X} d^2(x,y)\,d\pi\,.
\end{equation}

\begin{theorem}[The $2$-Wasserstein space is NNCC]
Let $(X,d)$ be a complete and separable metric space. If $(X\times X,d^2)$ is an NNCC space then so is $(\calP(X)\times \Pc(X),W_2^2)$. 
\end{theorem}

Because of its prevalence in applications let us look at a few examples of squared $2$-Wasserstein spaces that are NNCC:

\begin{example}[The standard $2$-Wasserstein space and generalized geodesics]\label{ex:euclidean}
    If $X=Y=\R^n$ and $d(x,y)=\lvert x-y\rvert$ is the standard Euclidean distance, the space $(\Pc(X),W_2^2)$ is NNCC. 
    In this setting, lifted c-segments with respect to the cost $W_2^2$ coincide with generalized geodesics, a notion of interpolation between probability measures which has been extensively considered in \cite{ambrosio2005gradient} for the study of gradient flows in $(\Pc(X),W_2)$. The reader can compare our Definition \ref{def:lifted_csegment} with \cite[Definition 9.2.2]{ambrosio2005gradient}.
    We discuss for simplicity the case where transport costs in \eqref{eq:2Wdistance} for $(\mu_0, \nu)$ and $(\mu_1,\nu)$ are induced by transport maps. That is, there exist two maps $T_0,T_1\colon Y\rightarrow X$ such that $\pi_i=(T_i,\Id_Y)_\#\nu$ are optimal couplings of $(\mu_i,\nu)$. A generalized geodesic is then defined as $\mu(s)=(T_s)_\#\nu$, where $T_s(y)=(1-s)T_0(y)+sT_1(y)$.
    In this case, the 3-plan $\gamma_\nu$ in Definition \ref{def:lifted_csegment} is $[y\mapsto(T_0(y),T_1(y),y)]_\#\nu$, and if the maps $T_i$ are unique $\nu$-a.e., $\gamma_\nu$ is unique as well.
    As pointed out in Example \ref{ex:Hilbert-norm}, variational c-segments in $(\R^n\times \R^n, |\cdot-\cdot|^2)$ are unique for any triplet $(x_0,x_1,y)\in (\R^n)^3$, so that the measurable maps $\Lambda_s$ are uniquely determined.
    Therefore, in this case, the lifted c-segment from $(\mu_0, \nu)$ to $(\mu_1,\nu)$ is unique.
    Note that the fact that the function $s\mapsto W_2^2(\mu(s),\nu)-W_2^2(\mu(s),\sigma)$ is convex along generalized geodesics, for any $\sigma\in\Pc(X)$, when the base space is the Euclidean space, had already been pointed out in \cite{matthes2019variational}.
\end{example}

\begin{example}
    Let $X=Y=\mathbb{S}^n$, the $n$-dimensional sphere, and let $d$ be the geodesic distance on the sphere. The space $(\mathbb{S}^n,d^2)$ is NNCC, see Section \ref{ssec:examples_riemannian}. Therefore the space $(\Pc(\mathbb{S}^n),W_2^2)$ is NNCC.
\end{example}

Optimal transport on the base space of positive semi-definite matrices $\mathrm{S}_{+}^n$ has recently gained interest, in particular in machine learning \cite{ju2023deep,bonet2023slicedwasserstein,yair2020domain}. In these applications, different choices of metrics, and consequently geometries, can be made. 
A first possibility is the affine invariant metric on positive definite matrices $\mathrm{S}_{++}^n$, also called the log-Euclidean metric:
\begin{equation}\label{eq:log-euclidean}
S_0,S_1\in \mathrm{S}_{++}^n\longmapsto d_L(S_0,S_1)\coloneqq\|\log(S_0)-\log(S_1)\|.
\end{equation}
Here the matrix logarithm is defined as
$\log(S)\coloneqq U^T[\operatorname{diag}(\log(\lambda_i))] U$ for $U^T[\operatorname{diag}(\lambda_i)]U=S$ the eigendecomposition of $S$, 
$\| S \|=\operatorname{Tr}(S^\top S)$ is the Fr\"obenius norm and $\mathrm{S}_{++}^n$ denotes the set of positive definite matrices.
Since the space $(\mathrm{S}_{++}^n\times \mathrm{S}_{++}^n,\|\cdot-\cdot\|^2)$ is NNCC and the logarithm mapping is surjective in $\mathrm{S}_{++}^n$, the space $(\mathrm{S}_{++}^n\times \mathrm{S}_{++}^n,d_L^2)$ is NNCC as well (see Example \ref{ex:Hilbert-norm}). As a consequence:

\begin{example}
    If $X=Y=\mathrm{S}_{++}^n$ and $d$ is the log-Euclidean metric \eqref{eq:log-euclidean}, the space $(\Pc(\mathrm{S}_{++}^n),W_2^2)$ is NNCC.
\end{example}

The Bures--Wasserstein geometry (\Cref{ssec:BW}) is another natural example in which nonnegative cross-curvature holds. Indeed, this is a simple consequence of Theorems \ref{thm:BW_NNCC} and \ref{ThMainTheorem}.
This setting is of interest as a metric on the space of Gaussian mixtures, induced by the Bures--Wasserstein distance. It was first introduced in \cite{8590715}, see also \cite{delon2020wassersteintypedistancespacegaussian}.

\begin{example}
    Let $X=Y=\R^n \times \mathrm{S}_{+}^n$ and let $d((a,x),(b,y))^2=\| a-b\|^2 + \BW(x,y)^2$, where $\BW$ is the Bures--Wasserstein distance defined in Section \ref{ssec:BW}. Then $(\Pc(\R^n \times \mathrm{S}_{+}^n),W_2^2)$ is an NNCC space.
\end{example}

Our last example consists of a type of infinite product. We show that adding randomness to a measurable NNCC space preserves NNCC. Consider a probability space $(\Omega,\Fc,\mathbb{P})$, two Polish spaces $X$ and $Y$ and a lower semi-continuous cost $c\colon X\times Y\to\R\cup\{+\infty\}$ bounded from below. We say that $T$ is a \emph{random element} on $X$ if it is a measurable map $T\colon\Omega\to X$ ($X$ being equipped with its Borel $\sigma$-algebra). We similarly define random elements on $Y$. Let $\Xc$ and $\Yc$ denote the set of random elements on $X$ and $Y$ respectively. If $T\in\Xc$ and $S\in\Yc$ we define the expected cost 
\[
\hat\Tc_c(T,S)=\mathbb{E}(c(T,S))=\int_\Omega c(T(\omega), S(\omega))\,d\mathbb{P}(\omega)\,.
\]
We then have the following result, which can be seen as an infinite product version of \Cref{prop:products}.

\begin{proposition}[Random variables and infinite products]\label{prop:RE}
    Suppose that $(X\times Y,c)$ is an NNCC space. Then $(\Xc\times\Yc,\hat\Tc_c)$ is an NNCC space.
\end{proposition}
\begin{proof}
    Let $T_0,T_1$ be random elements on $\Xc$ and $\bar{S}$ a random element on $Y$ such that $\hat\Tc_c(T_i,\bar S)<+\infty$, $i=0,1$. Then $c(T_i,\bar S)<+\infty$ $\mathbb{P}$-a.s.
    Define $\Gamma$ by \eqref{eq:maximal-Gamma} and fix $0<s<1$. By \Cref{lemma:continuous-cost-measurable-NNCC}, setting $\gamma=(T_0,T_1,\bar S)_\#\mathbb{P}$, there exists a measurable $\Lambda_s\colon\Gamma\to X$ such that for any random element $S$ on $Y$ the NNCC inequality holds $\mathbb{P}$-a.e.,
    \begin{multline}\label{eq:proof-prop:RE-1}
    c(\Lambda_s(T_0(\omega), T_1(\omega), \bar S(\omega)), \bar S(\omega))-c(\Lambda_s(T_0(\omega), T_1(\omega), \bar S(\omega)), S(\omega)) \leq \\(1-s)[c(T_0(\omega),\bar S(\omega))-c(T_0(\omega),S(\omega))] + s[c(T_1(\omega),\bar S(\omega))-c(T_1(\omega),S(\omega))]\,.
\end{multline}
Define $T_s(\omega)=\Lambda_s(T_0(\omega), T_1(\omega), \bar S(\omega))$, a random element on $X$ since $\Lambda_s$ is measurable. Integrating \eqref{eq:proof-prop:RE-1} against $\mathbb{P}$ we obtain the desired inequality. 
\end{proof}

Let us conclude this section by discussing variational c-segments in the case of a convex cost with an application to the optimal transport cost.
Assume that $Y$ is a Banach space and that $y \mapsto c(x,y)$ is convex for each $x\in X$. Convex analysis can be used to study differential properties of variational c-segments.
Consider $\yb\in Y$ and $x_0,x_1 \in X$ such that there exist $q_0 \in \partial_y c(x_0,\yb)$ and $q_1 \in \partial_y c(x_1,\yb)$, where $\partial_y c(x,y)$ denotes the subdifferential from convex analysis in the $y$ variable only. One has, for all $y\in Y,$
\begin{align}
    c(x_0,\yb) - c(x_0, y) &\leq \langle q_0,\yb - y \rangle\,,\\
    c(x_1,\yb) - c(x_1, y) &\leq \langle q_1,\yb- y \rangle\,.
\end{align}
Let $(\rmx,\yb)$ be a variational c-segment with $\rmx(0)=x_0$ and $\rmx(1)=x_1$, then the two previous inequalities imply 
\begin{equation}
  c(\rmx(s),\yb)  - c(\rmx(s), y) \leq (1-s)  [c(x_0,\yb) - c(x_0, y)] + s[c(x_1,\yb) - c(x_1, y)] \leq (1-s)\langle q_0,\yb-y \rangle + s\langle q_1,\yb-y \rangle\,,
\end{equation}
which means $(1-s)q_0+s \,q_1 \in \partial_y c(\rmx(s),\yb)$. In summary, we obtain a generalization of \Cref{lemma:auto-csegment} in the absence of differentiability: if $(\rmx,\yb)$ is a variational c-segment, then 
\[
(1-s)\partial_y c(\rmx(0),\yb) + s \,\partial_y c(\rmx(1),\yb) \subset \partial_y c(\rmx(s),\yb),
\]
where the addition denotes here the Minkowski sum of sets.

We now apply this remark to the optimal transport cost itself. We consider Polish spaces $X$, $Y$ and a continuous cost $c$ on $X\times Y$. We use the duality between the space of bounded continuous functions $C_b(X)$ and regular Radon measures on $X$.
The dual formulation of optimal transport reads
\begin{equation}\label{eq:otdual}    
    \Tc_c(\mu,\nu)=\sup_{\phi\in C_b(Y)} \int_X \phi^c \, d\mu + \int_Y \phi \,d \nu,
\end{equation}
 where we recall that the c-transform of $\phi$ is defined as
\begin{equation}\label{eq:c-trans}
    \phi^c(x)\coloneqq\inf_{y\in Y} c(x,y)-\phi(y)\,.    
\end{equation}
Let $H(\phi) = \int_X-\phi^c\,d\mu=\sup_{T: X \to Y }\int_X [\phi(T(x))-c(x,T(x))]\,d\mu(x)$, where the supremum is taken over measurable maps $T\colon X \to Y$. Then $H$ is convex and lower semi-continuous and
\begin{equation}
    \Tc_c(\mu,\nu)= \sup_{\phi\in C_b(Y)}  \int_Y \phi \,d \nu - H(\phi)\,.
\end{equation}
Applying \cite[Proposition 5.1 and Corollary 5.2]{ekeland1999convex} to this duality setting, we obtain that the subdifferential of the optimal transport cost in its second slot equals the set of optimal potentials that are continuous and bounded,
\[
\partial_\nu \Tc_c(\mu,\nu)\coloneqq \left\{ \phi \in C_b(Y)~:~ \int_X \phi^c d\mu+\int_Y \phi \,d\nu = \Tc_c(\mu,\nu) \right\}\,.
\]
Note that, in general, the supremum in the dual formulation is not necessarily attained in $C_b(Y)$. When $X \times Y$ is compact and $c$ is continuous, optimal potentials exist \cite[Theorem 1.39]{santambrogio2015optimal}, thus $\partial_\nu \Tc_c(\mu,\nu)$ is nonempty.
In this setting, the remark above on convex costs that are NNCC can be applied to the optimal transport cost, since it is convex in its input measures. 

\begin{proposition}\label{prop:convexityTcI2}
    Let $X$ and $Y$ be compact Polish spaces and $c\colon X\times Y\to\R$ a continuous cost. Consider a variational c-segment $(\mu(s),\nu)$ in $(\Pc(X)\times \Pc(Y),\Tc_c)$ and let $\phi_0,\phi_1\in C_b(Y)$ be two optimal potentials for $(\mu(0),\nu)$ and $(\mu(1),\nu)$ respectively. Then for any $0< s< 1$, $\phi_s = (1-s)\phi_0+s\,\phi_1$ is an optimal potential for $(\mu(s),\nu)$. 
\end{proposition}

\subsection{On the converse implication}\label{ssec:converse}

It is relatively straightforward to show a converse to \Cref{ThMainTheorem}:

\begin{proposition}[NNCC of Wasserstein $\implies$ NNCC of the base space]\label{ThConverseImplication}
    Let $X$ and $Y$ be two Polish spaces and $c\colon X\times Y\to\R\cup\{+\infty\}$ a lower semi-continuous cost function bounded from below. Suppose that $(\Pc(X)\times\Pc(Y),\Tc_c)$ is an NNCC space. Then so is $(X\times Y,c)$.
\end{proposition}

\begin{proof}
    Let us show the contrapositive. Assume that $(X\times Y,c)$ is not an NNCC space. Then there exist $\yb\in Y$, $x_0,x_1\in X'\coloneqq\{x\in X : c(x,\yb)<+\infty\}$ and $s\in (0,1)$ such that for every $x\in X'$ we can find $y\in Y$ such that the NNCC inequality does not hold. This means there exists a map $T\colon X'\to Y$ such that for all $x\in X'$,
    \begin{equation}\label{eq:proof-1}
    c(x,\yb)-c(x,T(x))>(1-s)[c(x_0,\yb)-c(x_0,T(x))]+s\,[c(x_1,\yb)-c(x_1,T(x))]\,.
    \end{equation}
    Moreover, by defining the set-valued map 
    \[
    \Psi(x) = \{ y\in Y \,:\, c(x,\yb)-c(x,T(x)) >  (1-s)[c(x_0,\yb)-c(x_0,T(x))]+s\,[c(x_1,\yb)-c(x_1,T(x))]\}\,,
    \]
    $T$ can be chosen measurable by means of a measurable selection of $\Psi$, similarly to the proof of \Cref{lemma:continuous-cost-measurable-NNCC}.
    
    To show that $(\Pc(X)\times\Pc(Y),\Tc_c)$ is not an NNCC space, let us define the Dirac masses $\mu_0=\delta_{x_0}$, $\mu_1=\delta_{x_1}$ and $\nub=\delta_{\yb}$. Given $\alpha\in\R$ whose value will be chosen later consider the set $X_\alpha = \{x\in X' : c(x,T(x))\leq\alpha\}$ and define the map $\widetilde T\colon X'\to Y$ by
    \begin{equation*}
        \widetilde T\colon\left\{\begin{aligned}
        x\in X_\alpha&\longmapsto T(x)\\
        x\in X'\setminus X_\alpha &\longmapsto \yb\,.
        \end{aligned}\right.
    \end{equation*}
    Then $\widetilde T$ satisfies for all $x\in X'$
    \begin{equation}\label{eq:proof-2}
        c(x,\yb)-c(x,\widetilde T(x))\geq (1-s)[c(x_0,\yb)-c(x_0,\widetilde T(x))]+s\,[c(x_1,\yb)-c(x_1,\widetilde T(x))]\,,
    \end{equation}
    with a strict inequality on $X_\alpha$ and with equality on $X'\setminus X_\alpha$. Take now $\mu\in\Pc(X)$ satisfying $\Tc_c(\mu,\nub)<+\infty$. Then $\mu$ is concentrated on $X'$ and we may define $\nu=\widetilde T_\#\mu$. Integrating \eqref{eq:proof-2} against $\mu$ gives us 
    \begin{equation}\label{eq:proof-3}
    \Tc_c(\mu,\nub)-\int_{X'} c(x,\widetilde T(x))\,d\mu(x) \geq (1-s)[\Tc_c(\mu_0,\nub)-\Tc_c(\mu_0,\nu)] + s\,[\Tc_c(\mu_1,\nub)-\Tc_c(\mu_1,\nu)]\,.
    \end{equation}
    The left-hand side is finite since by splitting the integral on $X_\alpha$ and $X'\setminus X_\alpha$ we find that  
    \[
    \int_{X'} c(x,\widetilde T(x))\,d\mu(x)\leq \alpha + \Tc_c(\mu,\nub)\,.
    \]
    Therefore the inequality in \eqref{eq:proof-3} is strict as soon as $\mu(X_\alpha)>0$. We can ensure that happens by taking $\alpha$ large enough since \eqref{eq:proof-1} forces $c(x,T(x))<+\infty$ on $X'$, thus $X_\alpha$ forms an increasing family such that $\bigcup_{\alpha\in\R}X_\alpha$ is all of $X'$, which has full $\mu$-measure. 
    To finish, we bound $\int_{X'} c(x,\widetilde T(x))\,d\mu(x)\geq \Tc_c(\mu,\nu)$, which proves that for every $\mu\in\Pc(X)$ such that $\Tc_c(\mu,\nub)<+\infty$, there exists $\nu\in\Pc(Y)$ such that 
    \[
    \Tc_c(\mu,\nub)-\Tc_c(\mu,\nu)>(1-s)[\Tc_c(\mu_0,\nub)-\Tc_c(\mu_0,\nu)] + s\,[\Tc_c(\mu_1,\nub)-\Tc_c(\mu_1,\nu)]\,.
    \]
    This forbids $(\Pc(X)\times\Pc(Y),\Tc_c)$ to be an NNCC space.
\end{proof}

One can then ask for a converse result for the two other properties, namely \snccconv{} and \smtw{}.
We prove it under more stringent assumptions on the cost than for NNCC. While these assumptions might be weakened, our results are sufficient to prove with an example that the \smtw{} property does not lift to the Wasserstein space (see Section \ref{SecCounterExampleMTW}).
Let us recall that a cost function $c$ is said to be twisted (or to satisfy the twist condition) if 
\begin{equation}\label{eq:twist_condition}
    \text{$\nabla_y c(\cdot,y)$ is injective for every $y\in Y$.}
\end{equation}

\begin{proposition}\label{ThMainTheorem_converse}
    Let $X,Y\subset\Rd$ be compact domains, with $Y$ coinciding with the closure of its interior, $Y=\overline{\operatorname{int}Y}$. Let $c\in C^1(X\times Y)$ be a cost satisfying the twist condition \eqref{eq:twist_condition}. Then, if the Wasserstein space $(\Pc(X)\times \Pc(Y),\Tc_c)$ satisfies the \snccconv{} condition (resp. the \smtw{} condition), so does the space $(X \times Y,c)$.
\end{proposition}

The proof is postponed to Appendix \ref{ssec:appendix-converse}.

\subsection{Condition \smtw{} does not lift to the Wasserstein space}\label{SecCounterExampleMTW}

In this section, we show that condition \smtw{} does not lift to the Wasserstein space, contrary to the NNCC property. To this end, we provide a counterexample. The reasoning is as follows: If the Wasserstein space satisfies condition \smtw{}, then lifts of c-segments from the base space (as defined in Lemma \ref{lem:Tc-var-c-seg_nuac}), whenever these are well-defined, are always \smtw{}-variational c-segments. Therefore, it is sufficient to find an example such that all lifts of c-segments do not satisfy condition \smtw{}.
	
\begin{example}[No lifts satisfy condition \smtw{}]\label{ex:MTW_counterexample}
    We consider the space $\R^2\times\R^2$ endowed with the cost $c(x,y)=-\log(|x-y|)$. On the subspace $N=\R^2\times\R^2\setminus \{x=y\}$, the cost $c$ is smooth and twisted, and c-segments are well-defined for any points $(x_0,\ov y),(x_1,\ov y)\in N$. Furthermore, the space $(N,c)$ satisfies the \smtw{} condition \cite[Example 12.41]{villani2009optimal}. Then, we consider the space $\Pc(\R^2)\times\Pc(\R^2)$, endowed with the transport cost $\Tc_c$, and the following four measures (with disjoint supports):
    \[
    \mu_0=\frac12 \delta_{x_0^1}+\frac12 \delta_{x_0^2}\,, \quad \mu_1=\frac12 \delta_{x_1^1}+\frac12 \delta_{x_1^2}\,, \quad \nu=\delta_{\ov y}\,,\quad \sigma=\delta_{y}\,,
    \]
    where $x_0^1=(-1,0)$, $x_0^2=(1,0)$, $x_1^1=(0,-1)$, $x_1^2=(0,1)$, $\ov y=(-\frac12,0)$ and $y=(\frac12,0)$. There exist four possible different c-segments $(\rmx,\ov y)$ on the base space, corresponding respectively to the four different triplets $( x_0^1, x_1^2, \ov y )$, $( x_0^2, x_1^2, \ov y )$, $( x_0^2, x_1^1, \ov y )$ and $( x_0^1, x_1^1, \ov y )$. These are explicitly given by (see Figure \ref{fig:MTW_counterexample_a})
    \begin{alignat*}{2}
        &\rmx_1(s)= \Big(-\frac{4s^2 - 9s + 5}{8s^2 - 12s + 5}, \frac{s}{8s^2 - 12s + 5}\Big), \quad &&
        \rmx_2(s)= \Big( -\frac{4s^2 + s - 5}{8s^2 - 4s + 5}, \frac{9s}{8s^2 - 4s + 5}\Big), \\
        &\rmx_3(s)= \Big(-\frac{4s^2 + s - 5}{8s^2 - 4s + 5}, -\frac{9s}{8s^2 - 4s + 5}\Big),
        &&\rmx_4(s)=\Big(-\frac{4s^2 - 9s + 5}{8s^2 - 12s + 5}, -\frac{s}{8s^2 - 12s + 5}\Big).
    \end{alignat*}
    Since $\nu$ is concentrated on a single Dirac mass, the optimal transport plans for $(\mu_i,\nu)$,  are $\pi_i=\frac12 \delta_{(x_i^1,\ov y)}+\frac12 \delta_{(x_i^2,\ov y)}$, for $i=0,1$.
    In this case, any transport plan $\gamma_\nu\in \Pc(\R^2)\times\Pc(\R^2)\times \Pc(\R^2)$ as in point \ref{enum:lcs-pi} of Definition \ref{def:lifted_csegment} is given by convex combinations of $\gamma^1_\nu=\frac12 \delta_{(x_0^1,x_1^2,\ov y)}+\frac12 \delta_{(x_0^2,x_1^1,\ov y)}$ and $\gamma^2_\nu=\frac12 \delta_{(x_0^1,x_1^1,\ov y)}+\frac12 \delta_{(x_0^2,x_1^2,\ov y)}$.
    Let us focus on $\gamma^1_\nu$ and the corresponding lift, $\mu^1(s)=\frac12 \delta_{\rmx_1(s)}+\frac12 \delta_{\rmx_3(s)}$. 
    The optimal transport plans for $(\mu^1(s),\nu)$ and $(\mu^1(s),\sigma)$ are again explicit and given, respectively, by $\pi(s)=\frac12 \delta_{(\rmx^1(s), \ov y)}+\frac12 \delta_{(\rmx^3(s), \ov y)}$ and $\tilde\pi(s)=\frac12 \delta_{(\rmx^1(s), y)}+\frac12 \delta_{(\rmx^3(s), y)}$. Then, the difference of transport costs along $\mu^1(s)$ is
    \[
    \Tc_c(\mu^1(s),\nu)-\Tc_c(\mu^1(s),\sigma)=\frac12 \big(c(\rmx_1(s),\ov y)-c(\rmx_1(s),y)\big)+\frac12 \big(c(\rmx_3(s),\ov y)-c(\rmx_3(s),y)\big)
    \]
    and it does not satisfy the \smtw{} inequality (see Figure \ref{fig:MTW_counterexample_b}). By symmetry, neither does $\Tc_c(\mu^2(s),\nu)-\Tc_c(\mu^2(s),\sigma)$, where $\mu^2(s)=\frac12 \delta_{\rmx_2(s)}+\frac12 \delta_{\rmx_4(s)}$. Since any lift of c-segments can be written as the convex combinations $\mu^t=(1-t)\mu^1+t\mu^2$, for $t\in[0,1]$, and the corresponding difference of transport costs is given by
    \[
    \Tc_c(\mu^t(s),\nu)-\Tc_c(\mu^t(s),\sigma)=(1-t)\left[\Tc_c(\mu^1(s),\nu)-\Tc_c(\mu^1(s),\sigma)\right]+t \left[\Tc_c(\mu^2(s),\nu)-\Tc_c(\mu^2(s),\sigma)\right],
    \] 
    there exists no lift of c-segments which satisfies the \smtw{} inequality.
\end{example}

\begin{figure}[t]
    \centering
    \subfloat[]{\includegraphics[trim={1.8cm 8cm 2cm 8cm},clip,width=0.4\textwidth]{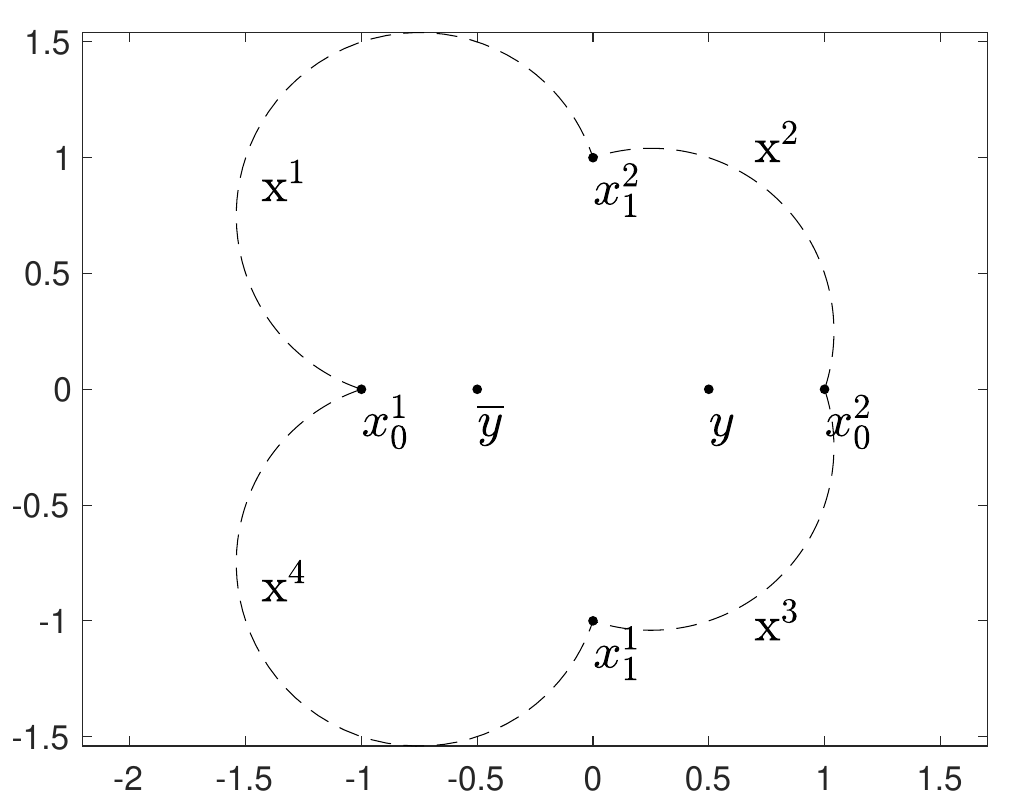} \label{fig:MTW_counterexample_a}}
    \hspace{3em}
    \subfloat[]{\includegraphics[trim={1.8cm 8cm 2cm 8cm},clip,width=0.4\textwidth]{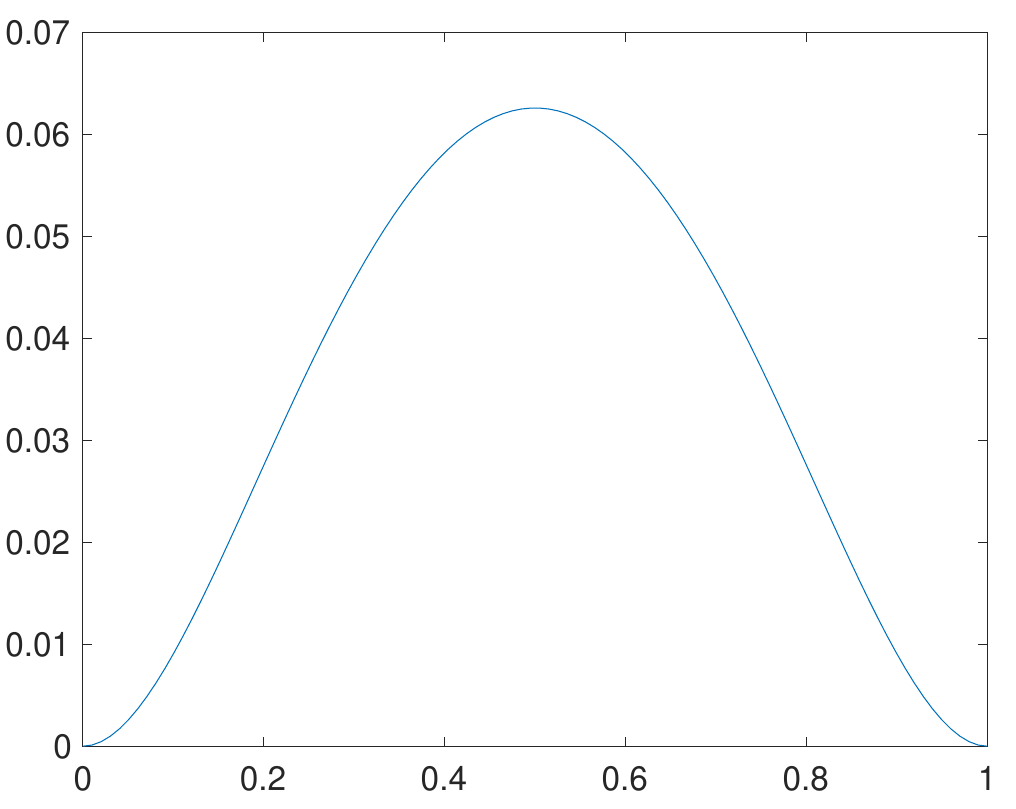}\label{fig:MTW_counterexample_b}}
    \caption{On the left, setting of Example \ref{ex:MTW_counterexample} with the four possible c-segments $\rmx^i$, for $1\le i\le4$, in $(\R^2\times\R^2,c)$ with $c=-\log(|x-y|)$. On the right, plot of the difference of transport costs $f(s)=\Tc_c(\mu^1(s),\nu)-\Tc_c(\mu^1(s),\sigma)$, for the lift of c-segments $\mu^1(s)=\frac12 \delta_{\rmx_1(s)}+\frac12 \delta_{\rmx_3(s)}$, which does not satisfy condition \smtw{}, since $f(s)>(1-s)f(0)+sf(1)=0$ for every $s\in(0,1)$.}
\end{figure}

\section{Other examples}\label{SecExamples}

\subsection{NNCC Riemannian manifolds}\label{ssec:examples_riemannian}

Let $M$ be a compact connected Riemannian manifold, whose geodesic distance is denoted by $d$. We consider the cost $c(x,y)=\frac12d^2(x,y)$ on $M\times M$. Given $x\in M$ and $v\in T_xM$ define:
\begin{itemize}
    \item the \emph{cut time} of $(x,v)$ by 
    \[t_C(x,v)=\max\{t\geq 0 : (\exp_x(sv))_{0\leq s\leq t} \text{ is a minimizing geodesic}\},
    \]
    \item the \emph{injectivity domain} of the exponential map at $x$ by 
    \[\operatorname{I}(x)=\{tv : 0\leq t<t_C(x,v), v\in T_xM\},
    \]
    \item the \emph{cut locus} of $x$ by
    \[
    \operatorname{cut}(x)=\{\exp_x(t_C(x,v)v) : v\in T_xM\setminus\{0\}\},
    \]
    and the cut locus of $M$ by $\operatorname{cut}(M)=\bigcup_{x\in M}(\{x\}\times\operatorname{cut}(x))$.
\end{itemize}

Let $\yb\in M$. If $\operatorname{I}(\yb)$ contains a segment $v(s)=(1-s)v_0+sv_1$, $s\in [0,1]$, then $s\mapsto (\rmx(s),\yb)$ with $\rmx(s)\coloneqq \exp_\yb v(s)$ is a c-segment since $v(s)=\nabla_yc(\rmx(s),\yb)$. Indeed the squared Riemannian distance is continuously differentiable on $(M\times M)\setminus\operatorname{cut}(M)$ (in fact as smooth as $M$ itself) and here $\rmx(s)\in M\setminus \operatorname{cut}(\yb)$. This is a minor extension of the notion of c-segment given in \Cref{def:c-segment} since $c$ fails to be smooth on all of $M\times M$.

\begin{proposition}\label{prop:Riemannian-NNCC}
Suppose that $M$ satisfies 
\begin{itemize}
    \item for each $x\in M$, $\operatorname{I}(x)$ is a convex subset of $T_xM$;
    \item $M$ has nonnegative cross-curvature in the sense that the MTW tensor satisfies $\mathfrak{S}_c(x,y)(\xi,\eta)\geq 0$ for every $(x,y)\in (M\times M)\setminus\operatorname{cut}(M)$ and every $\xi\in T_xM$, $\eta\in T_yM$.
\end{itemize}
Then $(M\times M, d^2)$ satisfies condition \snccconv{}, and in particular is an NNCC space. 
\end{proposition}
\begin{proof}
    Let $x_0,x_1,\yb\in M$. By connectedness and completeness of $M$ there exist minimizing geodesics joining $\yb$ to respectively $x_0$ and $x_1$. The tangent vectors at $\yb$ encoding the geodesics are denoted by $v_0,v_1$, i.e.\ $x_i=\exp_\yb v_i$ for $i=0,1$. Then for any $\varepsilon>0$ smaller than $1$, $(1-\varepsilon)v_0$ and $(1-\varepsilon)v_1$ belong to $\operatorname{I}(\yb)$. 
    
    Since $\operatorname{I}(\yb)$ is convex it contains $(1-\varepsilon) v(s)$ for every $s\in [0,1]$, where $v(s)\coloneqq (1-s)v_0 + s v_1$. Let $\rmx^\eps(s)=\exp_\yb((1-\varepsilon)v(s))$. Let $y\in M$ and fix $0<s_0<1$. Again we may find a minimizing geodesic connecting $x\coloneqq \rmx^\varepsilon(s_0)$ to $y$ whose initial velocity we denote $w\in T_{x}M$. Then  $(1-\varepsilon)w\in \operatorname{I}(x)$, and since $\yb$ belongs to $\exp_x(\operatorname{I}(x))$ there exists a c-segment $(t\in[0,1])\mapsto (x,\rmy(t))$, with $\rmy(0)=\yb$ and $\rmy(1)=\exp_x((1-\varepsilon) w)\eqqcolon y^\varepsilon$. From the ``non-tensorial expression of $\Sc$'' \cite[Lemma 4.5]{kim2010continuity}, we obtain (see e.g.\ the proof of \cite[Theorem 2.10]{kim2012towards})
    \[
    \frac{d^2}{ds^2}\Big|_{s=s_0}\big[c(\rmx^\varepsilon(s),\yb)-c(\rmx^\varepsilon(s),y^\varepsilon)\big] = \int_0^1(1-t)\,\Sc(x,\rmy(t))(\dot\rmx^\varepsilon(s_0),\dot\rmy(t))\,dt\,,
    \]
    which implies by nonnegative cross-curvature
    \begin{equation}\label{prop:Riemannian-NNCC-2}
    \frac{d^2}{ds^2}\Big|_{s=s_0}c(\rmx^\varepsilon(s),\yb)\geq \frac{d^2}{ds^2}\Big|_{s=s_0}c(\rmx^\varepsilon(s),y^\varepsilon)\,.
    \end{equation}
    To conclude we want to bound the right-hand side of \eqref{prop:Riemannian-NNCC-2} by a corresponding quantity at $y$ instead of $y^\varepsilon$. Since $x$ may belong to the cut locus of $y$ the cost $c(\cdot,y)$ may not be twice differentiable at $x$. To get around this, we can use a type of monotonicity of the Hessian of the squared distance along geodesics. We may for instance adapt \cite[Lemma 2.3]{CEMS01}. By \cite[Claim 2.4]{CEMS01} we have the metric inequality for $r\geq 0$ small enough
    \[
    (1-\varepsilon)c(\rmx^\varepsilon(s_0\pm r),y)\leq c(\rmx^\varepsilon(s_0\pm r),y^\varepsilon)-\varepsilon (1-\varepsilon)c(x,y)\,.
    \]
    This implies 
    \begin{equation}\label{prop:Riemannian-NNCC-3}
    (1-\varepsilon)\limsup_{r\to 0}\frac{c(\rmx^\varepsilon(s_0+ r),y)+c(\rmx^\varepsilon(s_0- r),y)-2c(x,y)}{r^2}\leq \frac{d^2}{ds^2}\Big|_{s=s_0}c(\rmx^\varepsilon(s),y^\varepsilon)\,.
    \end{equation}
    Therefore while $s\mapsto c(\rmx^\eps(s),y)$ may not have a Hessian at $s=s_0$, the finite differences remain bounded above. 
    Introduce $f^\varepsilon(s)=c(\rmx^\varepsilon(s),\yb)-(1-\varepsilon)c(\rmx^\varepsilon(s),y)$. Combining \eqref{prop:Riemannian-NNCC-2} with \eqref{prop:Riemannian-NNCC-3} we have shown that 
    \begin{equation}\label{prop:Riemannian-NNCC-4}
    \liminf_{r\to 0}\frac{f^\varepsilon(s_0+r)+f^\varepsilon(s_0-r)-2f(s_0)}{r^2}\geq 0\,.
    \end{equation}
    Since $s_0$ is arbitrary in $(0,1)$ and $f^\varepsilon$ is continuous, this implies convexity of $f^\varepsilon$ (see e.g.\ the proof of \cite[Lemma 3.11]{CEMS01}). Since $f^\varepsilon$ converges pointwise to $c(\rmx(s),\yb)-c(\rmx(s),y)$ as $\varepsilon\to 0$, where $\rmx(s)$ is defined by $\rmx(s)=\exp_\yb(v(s))$, we obtain convexity of $s\mapsto c(\rmx(s),\yb)-c(\rmx(s),y)$. 
    \end{proof}

Note that variational c-segments on a Riemannian manifold are not unique in general, as it is already the case for geodesics.
There are not many examples of Riemannian manifolds that are nonnegatively cross-curved in the differentiable sense. 
Kim and McCann \cite[Remark 3.4 and Section 6]{kim2012towards} show the sphere to have nonnegative cross-curvature outside of its cut locus. Then \Cref{prop:Riemannian-NNCC} implies: 

\begin{corollary}\label{cor:sphereNNCC}
    Spheres and products of spheres satisfy \snccconv{}, and are thus NNCC spaces.
\end{corollary}

In the next section, \Cref{ThFRNNCC} directly extends this result to (a part of) an infinite-dimensional sphere.

\subsection{Hellinger distance and Fisher--Rao metric on probability measures} \label{sec:hellinger}
Let $X$ denote a measurable space. 
The Hellinger distance $\mathcal{H}$ on $\mathcal{P}(X)$, the set of probability measures on $X$, is defined by
\begin{equation}
    \mathcal{H}^2(\mu,\nu) \coloneqq \int_X \left( \sqrt{\frac{d\mu}{d\lambda}(x)} - \sqrt{\frac{d\nu}{d\lambda}(x)}\right)^2 d\lambda(x)\,,
\end{equation}
where $\lambda$ is any dominating measure for both $\mu$ and $\nu$,
that is, $\mu$ and $\nu$ are absolutely continuous with respect to $\lambda$, and $\frac{d\mu}{d\lambda}$ denotes the Radon--Nikodym derivative.
This distance is well-known in statistics, popularized by Le Cam for its use for hypothesis testing \cite{LeCam1986}. 
The Hellinger distance also appears as a limit case in the unbalanced optimal transport distance (see Section \ref{SecUOT}) but on the cone of nonnegative Radon measures.
There it has the geometry of a Hilbert space via the square root transform. 
In contrast, the geometry of the Hellinger distance on the space of probability measures is rather an infinite dimensional sphere endowed with an ambient $L^2$ metric. In the following theorem, we define a family of sets on the sphere that satisfy the NNCC property for the extrinsic squared Hilbert distance.

\begin{theorem}\label{ThSphereAmbientMetricSNCC}
    Let $\mathbb S$ denote the unit sphere in a Hilbert space $(H,\langle\cdot,\cdot\rangle)$. Consider a convex cone $K_0 \subset H$ satisfying $\bracket{x,y}\geq 0$ for all $x,y \in K_0$, and let $S_0 = K_0 \cap \mathbb{S}$.     
    Let $D(x,y)=\lVert x-y\rVert^2$ denote the squared Hilbert distance. Then $(S_0,D)$ satisfies condition \snccconv{}, and is thus an NNCC space.
\end{theorem}

\begin{proof}
    Let $x_0,x_1,\yb\in S_0$. We search for a variational c-segment joining $(x_0,\yb)$ to $(x_1,\yb)$ in the form 
    \begin{equation*}
        \rmx(s)=z(s)+h(s)\yb,
    \end{equation*}
    where $z(s)\coloneqq (1-s)x_0+s x_1$, and $h(s)$ is a scalar chosen so that $\rmx(s)\in \mathbb{S}$. Since $z(s)$ belongs to the unit ball, the line $(t\in\R)\mapsto z(s)+t\yb$ intersects the sphere at one or two points, which correspond to a positive and a negative value of $t$. The choice $h(s)\geq 0$, which can be expressed as 
    \begin{equation*}
        h(s)=\sup\{t \in\R : \norm{z(s)+t\yb} \leq 1\},
    \end{equation*}
    gives $\rmx(s)$ as a conical combination of $x_0$, $x_1$ and $\yb$, thereby guaranteeing that $\rmx(s)$ belongs to $K_0$, and therefore to $S_0$.   
    Note that when $s=0$, a value $t>0$ implies $\norm{x_0+t\yb}^2=1+t^2+2t\bracket{x_0,\yb}>1$ (using that $\bracket{x_0,\yb}\geq 0$); then necessarily $h(0)=0$, which shows $\rmx(0)=x_0$. A similar reasoning gives $\rmx(1)=x_1$.

    Convexity of the unit ball directly implies convexity of the set $C\coloneqq \{(s,t)\in [0,1]\times\R : \norm{z(s)+t\yb}\leq 1\}$; in turn  this implies that $h$ is \emph{concave}. Indeed, given $\lambda\in [0,1]$ and $(s_0,t_0), (s_1,t_1)\in C$, the linear interpolation $(s_\lambda,t_\lambda)$ lies in the convex set $C$; by definition of $h$, we deduce $h(s_\lambda)\geq t_\lambda=(1-\lambda)t_0+\lambda t_1$. Maximizing over $t_0,t_1$ then gives $h(s_\lambda)\geq (1-\lambda)h(s_0)+\lambda h(s_1)$. We compute now for any $y\in S_0$,
    \begin{align*}
        D(\rmx(s),\yb)-D(\rmx(s),y)&=2\bracket{\rmx(s),y-\yb}\\
        &=2\bracket{(1-s)x_0+sx_1,y-\yb}+2h(s)\bracket{\yb,y-\yb}.
    \end{align*}
    Note that $\bracket{\yb,y-\yb}=\bracket{\yb,y}-1\leq 0$ by the Cauchy--Schwarz inequality. Concavity of $h$ then suffices to show the convexity of $s \mapsto D(\rmx(s),\yb)-D(\rmx(s),y)$. 

\end{proof}

The next result says that up to taking a square root transform, the space of probability measures endowed with the squared Hellinger distance is a part of a sphere satisfying the condition in \Cref{ThSphereAmbientMetricSNCC}, showing it to be a NNCC space.

\begin{corollary}\label{ThHellingerSNCC}
  The space $(\mathcal{P}(X),\mathcal{H}^2)$ satisfies condition \snccconv{}, and is thus an NNCC space.
\end{corollary}
\begin{proof}
    Let $\mu_0,\mu_1,\nub \in \Pc(X)$, and define their square root densities $f_0, f_1, \bar g$ with respect to a common dominating measure $\lambda$. Consider the Hilbert space $H=L^2(\lambda)$ and the convex cone $K_0=\{f\in L^2(\lambda) : f \geq 0 \,\, \lambda\text{-a.e.}\}$. Then $f_0,f_1, \bar g$ belong to $S_0\coloneqq \{f\in K_0 : \norm{f}_{L^2(\lambda)}=1\}$. By \Cref{ThSphereAmbientMetricSNCC}, there exists a variational c-segment $(f(s),\bar g)$ connecting $(f_0,\bar g)$ to $(f_1,\bar g)$ which satisfies the \snccconv{} property. Then $(\mu(s) \coloneqq f(s)^2 d \lambda,\nub)$ is a variational c-segment on $(\mathcal{P}(X),\mathcal{H}^2)$ satisfying the \snccconv{} property. We also point out that an explicit computation of $f(s)$ reveals $\mu(s)$ to be independent of $\lambda$.

\end{proof}

Endowed with the Hellinger distance, the metric space $\mathcal{P}(X)$ is not a length space. The induced (length space) distance is called the Fisher--Rao distance, and can be written as
\begin{equation}\label{EqFisherRao}
    \operatorname{FR}(\mu,\nu) = \arccos\left(1 - \frac{1}{2} \mathcal{H}(\mu,\nu)\right).
\end{equation}
The corresponding geometry on the space of probability measures is an infinite-dimensional sphere equipped with its intrinsic distance.

Since a finite-dimensional sphere is NNCC, one can expect this result to hold in infinite dimensions. 
We first start with a lemma that similarly to Theorem \ref{ThSphereAmbientMetricSNCC} describes a family of subsets of the sphere that are NNCC.

\begin{lemma}\label{ThLemmaSphereIntrinsic}
Let $\mathbb S$ be the unit sphere of a Hilbert space $H$ and let $S_0$ denote a subset of $\mathbb S$ satisfying the same properties as in \Cref{ThSphereAmbientMetricSNCC}.
Let $d$ be the geodesic distance on $\mathbb S$. Then $(S_0,d^2)$ satisfies condition \snccconv{} and is thus an NNCC space.
\end{lemma}

\begin{proof}
    Let $x_0,x_1,\yb\in\mathbb{S}$. There exists a 2-dimensional unit sphere $\mathbb{S}_2\subset H$ centered at $0$ which contains these three points. By the proof of \Cref{prop:Riemannian-NNCC} and by \Cref{cor:sphereNNCC}, there exists a curve $\rmx(s)$ on $\mathbb{S}_2$ connecting $x_0$ to $x_1$ such that $(\rmx,\yb)$ is a variational c-segment satisfying the convex NNCC inequality; this curve is given by 
    \begin{equation}\label{eq:proof-ThLemmaSphereIntrinsic-1}
        \rmx(s)=\exp_\yb(v(s)),\quad v(s)=(1-s)v(0)+sv(1).
    \end{equation}
    Here, $v(0)$ and $v(1)$ are tangent vectors to $\yb$ within $\mathbb{S}_2$ which encode the initial velocity of \emph{minimizing geodesics} in $\mathbb{S}_2$ joining $\yb$ to $x_0$ and $x_1$ respectively.

    Consider now another point $y\in\mathbb{S}$ and any $n$-dimensional unit sphere $\mathbb{S}_n\supset\mathbb{S}_2$ centered at $0$ containing $x_0,x_1,\yb$ and $y$. Since geodesics of centered spheres are given by arcs of great circles, $\mathbb{S}_2$ is a totally geodesic submanifold of $\mathbb{S}_n$. As a consequence, $v(0)$ and $v(1)$ still encode minimizing geodesics within $\mathbb{S}_n$. By the proof of \Cref{prop:Riemannian-NNCC} and by \Cref{cor:sphereNNCC}, $(\rmx,\yb)$ is then a variational c-segment on $(\mathbb{S}_n,d^2)$, such that $s\mapsto d^2(\rmx(s),\yb)-d^2(\rmx(s),y)$ is convex. Here we slightly abuse notation by denoting by $d$ the intrinsic distance in $\mathbb{S}_n$, since it coincides with the restriction of the distance on the infinite-dimensional sphere. Since $y$ is an arbitrary point on $\mathbb{S}$, we deduce that $(\rmx(s),\yb)$ is a variational c-segment satisfying the convex NNCC inequality on all of $(\mathbb{S},d^2)$. Therefore, $(\mathbb{S},d^2)$ satisfies condition \snccconv{}.

    To conclude, we need to show that if $x_0,x_1$ and $\yb$ belong to $S_0$, then the path $\rmx(s)$ stays inside $S_0$. Recall that $S_0$ is the intersection of $\mathbb{S}$ and a convex cone $K_0$ such that 
    \begin{equation}\label{eq:proof-ThLemmaSphereIntrinsic-2}
        \bracket{x,y}\geq 0 \quad\text{for all $x,y\in K_0$}. 
    \end{equation}    
    Using the explicit expression of the exponential map on a centered unit sphere, \eqref{eq:proof-ThLemmaSphereIntrinsic-1} can be written as 
    \begin{equation*}\label{eq:proof-ThLemmaSphereIntrinsic-3}
        \rmx(s)=\cos(\alpha(s))\yb+\sin(\alpha(s))\frac{v(s)}{\alpha(s)}, \quad \alpha(s)=\norm{v(s)},\quad v(s)=(1-s)v(0)+sv(1).
    \end{equation*}
    Here and for the remainder of the proof, expressions of the form $\sin(\alpha)/\alpha$ should be replaced with the limiting value $1$ when $\sin(\alpha)=0$. Let $f(\alpha)=\alpha \cos(\alpha)/\sin(\alpha)$. Substituting $v(s)$ with its values at $s=0$ and $s=1$, we obtain the expression 
    \begin{equation*}
        \rmx(s)=\frac{\sin(\alpha(s))}{\alpha(s)}\bigg[\Big(f(\alpha(s))-(1-s)f(\alpha(0))-s f(\alpha(1))\Big)\yb+(1-s)\frac{\alpha(0)}{\sin(\alpha(0))}x_0+s\frac{\alpha(1)}{\sin(\alpha(1))} x_1\bigg].
    \end{equation*}
    Property \eqref{eq:proof-ThLemmaSphereIntrinsic-2} applied to $(x_0,\yb)$ and $(x_1,\yb)$ implies that the angles $\alpha(0)$ and $\alpha(1)$ lie in $[0,\pi/2]$. Convexity of the norm gives $\alpha(s)=\norm{(1-s)v(0)+sv(1)}\leq (1-s)\alpha(0)+s\alpha(1)$. Now, since $\sin(\alpha)/\alpha$ is nonnegative on $[0,\pi/2]$, and $f$ is both decreasing and concave on $[0,\pi/2]$, we obtain that $\rmx(s)$ is a conical combination of $\yb$, $x_0$ and $x_1$. We deduce that $\rmx(s)$ belongs to $K_0$, and therefore to $S_0$.

\end{proof}

\begin{proposition}\label{ThFRNNCC}
    The space $(\mathcal{P}(X),\operatorname{FR}^2)$ satisfies \snccconv{} and is thus NNCC.
\end{proposition}

\begin{proof}
Consider a triplet $(\mu_0,\mu_1,\overline{\nu}) \in \mathcal{P}(X)^3$ and a measure $\lambda$ dominating these three measures. 
The statement of the proposition follows easily by considering the measure $\mu(s)$ corresponding to the square of the variational c-segment on the $L^2(\lambda)$ sphere and using \Cref{ThLemmaSphereIntrinsic}.
\end{proof}

\subsection{Relative entropy, a.k.a. the Kullback--Leibler divergence} \label{sec:KL}

Let $(X,\Ac)$ be a measurable space. Given two probability measures $\mu$ and $\nu$ on $(X,\Ac)$, the \emph{relative entropy} or \emph{Kullback--Leibler divergence} between $\mu$ and $\nu$ is defined by 
\begin{equation}\label{eq:def-KL}
\operatorname{KL}(\mu,\nu) = \begin{cases} 
    \displaystyle\int_X F\Big(\frac{d\mu}{d\nu}\Big)\,d\nu & \text{if $\mu\ll\nu$},\\
+\infty & \text{otherwise},
\end{cases}
\end{equation}
with $F(t)=t\log(t)-t+1\geq 0$. In \eqref{eq:def-KL} $\mu\ll\nu$ means that $\mu$ is absolutely continuous with respect to $\nu$ and $\frac{d\mu}{d\nu}$ denotes the Radon--Nikodym derivative. Since the relative entropy can formally be seen as a Bregman divergence over the space of probability measures $\Pc(X)$, we can expect it to have nonnegative cross-curvature. That is indeed the case. 

\begin{proposition}\label{prop:KL-NNCC}
    $\big(\Pc(X)\times \Pc(X),\operatorname{KL}\big)$ is an NNCC space.
\end{proposition}
The proof of \Cref{prop:KL-NNCC} relies on two auxiliary lemmas.

\begin{lemma}\label{lemma:KL-NNCC-identity}
    Take $\mu_0,\mu_1,\nu\in\Pc(X)$ such that $\operatorname{KL}(\mu_0,\nu)<+\infty$ and $\operatorname{KL}(\mu_1,\nu) < +\infty$. Define $\mu(s)=(1-s)\mu_0+s\,\mu_1$ for some $0<s<1$. Then 
    \begin{equation}\label{eq:lemme-KL-NNCC-1}
        (1-s)\operatorname{KL}(\mu_0,\nu)+s \,\operatorname{KL}(\mu_1,\nu)-\operatorname{KL}(\mu(s),\nu)=(1-s)\operatorname{KL}(\mu_0,\mu(s))+s\operatorname{KL}(\mu_1,\mu(s))\,.
    \end{equation}
\end{lemma}
\begin{proof}
    Finiteness of $\operatorname{KL}(\mu_0,\nu),\operatorname{KL}(\mu_1,\nu)$ forces $\mu_0$, $\mu_1$ and therefore $\mu(s)$ to be absolutely continuous with respect to $\nu$. By the properties of the Radon--Nikodym derivative \cite[Chapter 3]{bogachev2007measure}, we have $\nu$-a.e.
    \begin{equation}\label{eq:proof-prop-KL-NNCC-2}
      \frac{d\mu(s)}{d\nu}=(1-s)\frac{d\mu_0}{d\nu}+s\frac{d\mu_1}{d\nu}\,,
    \end{equation}    
    \begin{equation}\label{eq:proof-prop-KL-NNCC-3}
        \frac{d\mu_i}{d\nu}=\frac{d\mu_i}{d\mu(s)}\frac{d\mu(s)}{d\nu} \quad\text{for } i=0,1\,.
    \end{equation}
    Using successively \eqref{eq:proof-prop-KL-NNCC-2} and \eqref{eq:proof-prop-KL-NNCC-3} twice, we obtain $\nu$-a.e. that
    \begin{align*}
        (1-s)F\Big(\frac{d\mu_0}{d\nu}\Big)+sF\Big(\frac{d\mu_1}{d\nu}\Big)-F\Big(\frac{d\mu(s)}{d\nu}\Big)&=
        \begin{aligned}[t]
        &(1-s)\frac{d\mu_0}{d\nu}\Big[\log\Big(\frac{d\mu_0}{d\nu}\Big)-\log\Big(\frac{d\mu(s)}{d\nu}\Big)\Big]\\
        &+s\frac{d\mu_1}{d\nu}\Big[\log\Big(\frac{d\mu_1}{d\nu}\Big)-\log\Big(\frac{d\mu(s)}{d\nu}\Big)\Big]
        \end{aligned}\\
        &=(1-s)\frac{d\mu_0}{d\nu}\log\Big(\frac{d\mu_0}{d\mu(s)}\Big)+s\frac{d\mu_1}{d\nu}\log\Big(\frac{d\mu_1}{d\mu(s)}\Big)\\
        &=(1-s)F\Big(\frac{d\mu_0}{d\mu(s)}\Big)\frac{d\mu(s)}{d\nu}+sF\Big(\frac{d\mu_1}{d\mu(s)}\Big)\frac{d\mu(s)}{d\nu}\,.
    \end{align*}
    Integrating against $\nu$ yields \eqref{eq:lemme-KL-NNCC-1}. 
\end{proof}

\begin{lemma}\label{lemma:KL-NNCC-bound-below}
    Let $\mu_0,\mu_1,\sigma\in\Pc(X)$, let $0<s<1$ and $\mu(s)=(1-s)\mu_0+s\,\mu_1$. Then 
    \begin{equation}\label{eq:proof-prop-KL-NNCC-5}
        \KL(\mu(s),\sigma)\geq (1-s)\KL(\mu_0,\sigma)+s\,\KL(\mu_1,\sigma)+(1-s)\log(1-s)+s\log(s)\,.
    \end{equation}
\end{lemma}
\begin{proof}
    Assume that $\KL(\mu(s),\sigma)$ is finite. By monotonicity of the logarithm we have
    \begin{align*}
    \KL(\mu(s),\sigma)=\int_X\log\Big(\frac{d\mu(s)}{d\sigma}\Big)\,d\mu(s)=(1-s)\int_X\log\Big(\frac{d\mu(s)}{d\sigma}\Big)\,d\mu_0+s\int_X\log\Big(\frac{d\mu(s)}{d\sigma}\Big)\,d\mu_1\\
    \geq (1-s)\int_X\log\Big((1-s)\frac{d\mu_0}{d\sigma}\Big)\,d\mu_0+s\int_X\log\Big(s\frac{d\mu_1}{d\sigma}\Big)\,d\mu_1,
    \end{align*}
    which implies \eqref{eq:proof-prop-KL-NNCC-5}
\end{proof}
\begin{proof}[Proof of \Cref{prop:KL-NNCC}]
    Fix $\mu_0,\mu_1,\nu\in\Pc(X)$ such that $\operatorname{KL}(\mu_0,\nu),\operatorname{KL}(\mu_1,\nu) < +\infty$. Fix $0<s<1$ and define $\mu(s)=(1-s)\mu_0+s\,\mu_1$. By \Cref{lemma:KL-NNCC-identity} we have 
    \begin{equation}\label{eq:proof-prop-KL-NNCC-1}
        (1-s)\operatorname{KL}(\mu_0,\nu)+s \,\operatorname{KL}(\mu_1,\nu)-\operatorname{KL}(\mu(s),\nu)=(1-s)\operatorname{KL}(\mu_0,\mu(s))+s\operatorname{KL}(\mu_1,\mu(s)).
    \end{equation}
    Note that the right-hand side of \eqref{eq:proof-prop-KL-NNCC-1} does not depend on $\nu$ and note that $\KL(\mu(s),\nu)$ must be finite.
    Let $\sigma$ be another probability measure such that $\operatorname{KL}(\mu_i,\sigma)$ for $i = 0,1$ is finite. Then, \eqref{eq:proof-prop-KL-NNCC-1} holds replacing $\nu$ with $\sigma$. Subtracting these two equalities implies that the NNCC inequality holds as an equality:
    \begin{equation}\label{eq:proof-prop-KL-NNCC-4}
    \KL(\mu(s),\nu)-\KL(\mu(s),\sigma)=(1-s)[\KL(\mu_0,\nu)-\KL(\mu_0,\sigma)]+s[\KL(\mu_1,\nu)-\KL(\mu_1,\sigma)]\,.
    \end{equation}
     To conclude the proof, observe that convexity of $\KL$ and \Cref{lemma:KL-NNCC-bound-below} imply
    \[\KL(\mu(s),\sigma)=+\infty\iff \KL(\mu_0,\sigma)=+\infty\text{ or } \KL(\mu_1,\sigma)=+\infty.
    \]    
\end{proof}

\subsection{Bures--Wasserstein}\label{ssec:BW}

In this section, we show that the set of Positive Semi-Definite (PSD) real valued matrices $\mathrm{S}^n_{+}$ endowed with the Bures--Wasserstein squared metric
\begin{equation}\label{eq:BW}
\operatorname{BW}^2(S_1,S_2) = \operatorname{tr}(S_1) + \operatorname{tr}(S_2) - 2 \operatorname{tr}\left(\sqrt{S_1^{1/2}S_2S_1^{1/2}}\right),
\end{equation}
where $S^\frac12$ denotes the unique square root of a PSD matrix $S$, is NNCC.
Interestingly, this is an example of Riemannian manifold with boundary, the latter being given by the set of singular PSD matrices $\mathrm{S}_{+}^n\setminus \mathrm{S}_{++}^n$.
In \cite{bhatia2017bureswasserstein}, it is shown that    
\begin{equation}\label{eq:BW_variational}
\operatorname{BW}^2(S_1,S_2) = \min_{\substack{M \,:\, MM^\top=S_1\\N \,:\, NN^\top=S_2}} \|M-N\|^2=\min_{U,U' \in \mathrm{O}(n)} \| S_1^{\frac12} U - S_2^{\frac12}U' \|^2= \min_{U \in \mathrm{O}(n)} \| S_1^{\frac12} - S_2^{\frac12} U \|^2\,,
\end{equation}
for any $S_1,S_2\in \mathrm{S}^n_{+}$, where $\| S \|=\operatorname{Tr}(S^\top S)$ is the Fr\"obenius norm and $\mathrm{O}(n)$ is the orthogonal group of matrices in dimension $d \geq 2$. In fact, the map $\tilde P: M \in \mathrm{GL}(n) \mapsto M M^\top$ is a Riemannian submersion between the linear group of matrices $\mathrm{GL}(n)$ endowed with the Fr\"obenius norm and the space of positive definite matrices $\mathrm{S}_{++}^n=\mathrm{GL}(n)/\mathrm{O}(n)$ endowed with the Bures--Wasserstein metric. 
Applying the result by Kim and McCann on Riemannian submersions \cite[Theorem 4.5]{kim2012towards}, one can therefore deduce that $(\mathrm{S}_{++}^n,\BW^2)$ has nonnegative cross-curvature. However, since $\mathrm{S}^n_{+}$ is a Riemannian manifold with boundary, the aforementioned theorem cannot be applied on the whole space. We will resort instead to Proposition \ref{prop:sncc-surj}.

\begin{theorem}\label{thm:BW_NNCC}
    The space $(\mathrm{S}_{+}^n, \BW^2)$ satisfies \snccconv{} and is thus an NNCC space.
\end{theorem}

\begin{remark}
    The space of non-degenerate Gaussian measures forms a geodesically complete subspace in the $2$-Wasserstein space $(\Pc(\Rd),W_2)$. Since they are parametrized by their mean and covariance matrix, it gives a metric on the space $\R^n\times \mathrm{S}_{++}^n$ which is precisely given by
    \begin{equation}
    d^2((m_1,S_1),(m_2,S_2))=|m_1-m_2|^2+\operatorname{BW}^2(S_1,S_2)\,,
    \end{equation}
    where $|\cdot|$ is the Euclidean norm.
    Thanks to Proposition \ref{prop:products}, the space $(\R_+\times \mathrm{S}_{++}^n,d^2)$ is NNCC. More generally, the space $(\R_+\times \mathrm{S}_+^n,d^2)$ is NNCC.
\end{remark}

\begin{proof}
    To apply Proposition \ref{prop:sncc-surj}, consider the following surjective map
    \[
    P: M\in \mathrm{M}(n)\longmapsto MM^\top\in \mathrm{S}^n_{+}
    \]
    where $\mathrm{M}(n)$ is the set of $n\times n$ matrices. The space $(\mathrm{M}(n),\|\cdot\|)$ is a Hilbert space and therefore $(\mathrm{M}(n),\|\cdot-\cdot\|^2)$ satisfies \snccconv{} and is an NNCC space (see Example \ref{ex:Hilbert-norm}).
    The projected cost on $\mathrm{S}_{+}^n$ is the Bures--Wasserstein distance squared from \eqref{eq:BW_variational}.
    Since any $M'\sim M$ (where we recall that $M'\sim M$ if $P(M')=P(M)$) can be written as $M'=MU$ for a unitary matrix $U$, this implies
    \[
    \min_{\substack{M'\sim M\\N'\sim N}} \|M'-N'\|^2 = \min_{\substack{U \text{unitary}\\N'\sim N}} \|(M-N' U^\top)U\|^2=\min_{N'\sim N} \|M-N'\|^2\,.
    \]
    Assumption \ref{eq:c-subm-property} is therefore satisfied, $P$ is a cost submersion and $(\mathrm{S}_{+}^n, \BW^2)$ satisfies \snccconv{} (see Remark \ref{rem:c-sub-NNCC-conv}).
\end{proof}

The previous proof does not provide the form of variational c-segments. This can be useful for applications.

\begin{proposition}\label{prop:BW-segments}
    For $S_0,S_1,S_2\in \mathrm{S}^n_+$, consider the corresponding unique square roots $S_0^{\frac12},S_1^{\frac12}$ and $S_2^{\frac12}$ and the following unitary matrices
    \[
    U_i=(S_2^{\frac12}S_i^{\frac12})(S_i^{\frac12}S_2^{\frac12}S_i^{\frac12})^{-\frac12}\,,\quad \text{for $i=0,1$,}
    \]
    where $(S_i^{\frac12}S_2^{\frac12}S_i^{\frac12})^{-\frac12}$ is any right inverse of $(S_i^{\frac12}S_2^{\frac12}S_i^{\frac12})^{\frac12}$.
    Then, curves of the form $(S(s),S_2)$, with $S(s)=M(s)M(s)^\top$ and $M(s)=(1-s)M_0+sM_1=(1-s)S_0^{\frac12}U_0^\top+sS_1^{\frac12}U_1^\top$,
    are variational c-segments. The curve $S(s)$ is unique if the matrices $S_2^{\frac12}S_i^{\frac12}$, for $i=0,1$, are invertible.
\end{proposition}
\begin{proof}
    To construct variational c-segments from $(S_0,S_2)$ to $(S_1,S_2)$, we need first to find $M_0,M_1,M_2\in \mathrm{M}(n)$ such that $P(M_i)=S_i$ and $(M_0,M_2)$ and $(M_1,M_2)$ are optimal couples.
    Then, given the variational c-segment $M(s)=(1-s)M_0+sM_1$ in the space $(\mathrm{M}(n),\| \cdot-\cdot\|^2)$, from Lemma \ref{lemma:projected-c-segments} the curve $S(s)=P(M(s))=M(s)M(s)^\top$ is a variational c-segment in $(\mathrm{S}_+^n,\BW^2)$.  
    
    In \eqref{eq:BW_variational}, an optimal matrix $U_i$ for $(S_i,S_2)$ is given by any unitary matrix in the polar decomposition of $S_2^\frac12S_i^\frac12$, that is any unitary matrix $U_i$ such that
    \[
    S_2^\frac12S_i^\frac12 =U_i(S_i^\frac12S_2^\frac12S_i^\frac12)^\frac12
    \]
    (see \cite[Theorem 1]{bhatia2017bureswasserstein} for details).
    Then, choosing $M_2=S_2^\frac12$ and for $i=0,1$, $M_i=S_i^\frac12 U_i^\top$ where 
    \[
    U_i=(S_2^{\frac12}S_i^{\frac12})(S_i^{\frac12}S_2^{\frac12}S_i^{\frac12})^{-\frac12}\,,
    \]
    with $(S_i^{\frac12}S_2^{\frac12}S_i^{\frac12})^{-\frac12}$ any right inverse of $(S_i^{\frac12}S_2^{\frac12}S_i^{\frac12})^{\frac12}$, we have
    \[
    |M_i-M_2|^2=\min_{\substack{M_iM_i^\top=S_i\\ M_2M_2^\top=S_2}} |M_i-M_2|^2=\operatorname{BW}^2(S_i,S_2)\,.
    \]
    Hence, the couples $(M_0,M_2)$ and $(M_1,M_2)$ chosen in this way are optimal.
    The polar decompositions of $S_2^{\frac12}S_i^{\frac12}$ are unique if these are invertible, in which case the matrices $M_i$ are uniquely determined.
\end{proof}

We highlight that for $S_0,S_1,S_2\in \mathrm{S}^n_{++}$, variational c-segments with respect to the cost $\operatorname{BW}^2$ correspond to generalized geodesics in the $2$-Wasserstein space $(\Pc(\R^n),W_2)$.
This can be intuitively understood since the space of zero-mean non-degenerate Gaussian measures endowed with the quadratic Wasserstein distance $W_2$ is totally geodesic in the quadratic Wasserstein space on $\R^n$. 
    
\begin{remark}
    The family of variational c-segments constructed above is not necessarily exhaustive, as there may exist variational c-segments on the projected space which are not projections of variational c-segments from the above space.
\end{remark}

\subsection{Unbalanced optimal transport}\label{SecUOT}

Let $X$ and $Y$ be two Polish spaces and let $\mathcal{M}_+(X)$, $\mathcal{M}_+(Y)$ denote the sets of nonnegative and finite Radon measures on $X$ and $Y$ respectively. We consider in this section the case of $\mathcal{M}_+(X) \times \mathcal{M}_+(Y)$ endowed with the unbalanced optimal transport cost. We refer the reader to \cite{SEJOURNE2023407} for an overview and to \cite{GeneralizedOT2,gallouet:hal-03498098,liero2018optimal} for more details on unbalanced optimal transport.  

We start with the definitions of entropy functions and functionals. 
\begin{definition}
An {\it entropy function} $F: \R \to [0,+\infty]$  is a 
 convex, lower semi-continuous, nonnegative function such that $F(1) = 0$ and $F(r) = +\infty$ if $r<0$. Its {\it recession constant} is defined as $F^{'}_{\infty} = \lim_{r\to+\infty} \frac{F(r)}{r}$. 
\end{definition}

\begin{definition}\label{ThCsiszar}
Let $F$ be an entropy function and $\mu,\nu \in \mathcal{M}_+(X)$.  The {\it Csiszár divergence} associated with $F$ is
\begin{equation*}
D_F(\mu,\nu) = \int_X F\left(\frac{d \mu(x)}{d \nu(x) }\right) d \nu(x) + F^{'}_{\infty} \int_X d \mu^{\perp}\,,
\end{equation*}
where $\mu^{\perp}$ is the orthogonal part of the Lebesgue decomposition of $\mu$ with respect to $\nu$. 
\end{definition}
\noindent
Unbalanced optimal transport is an optimal transport problem where soft penalizations (via Csiszár divergences) replace the hard marginal constraints on transport plans.

\begin{definition}[Entropy-Transport problem]\label{ThUOT} 
Let $(\mu,\nu)\in\mathcal{M}_+(X)\times\mathcal{M}_+(Y)$ and $F_0,F_1$ be entropy functions. The {\it unbalanced optimal transport problem} is defined as
\begin{equation}\label{EqPrimalUOT}
    \Uc\Tc(\mu,\nu) = \inf_{\gamma \in \mathcal{M}_+(X\times Y)} D_{F_{0}}(\gamma_0,\mu) + D_{F_1}(\gamma_1,\nu) + \int_{X\times Y} c(x,y) \,d \gamma(x,y)\,,
\end{equation}
where $\gamma_0,\gamma_1$ are respectively the first and second marginal of $\gamma$, and $c:X \times Y \to \R\cup \{+\infty\}$ is a cost function which is bounded below and continuous at each point where it has finite value.
\end{definition}
\noindent
An important example is the Wasserstein--Fisher--Rao (WFR) metric \cite{GeneralizedOT2}, also called Hellinger--Kantorovich \cite{liero2018optimal}, which is obtained for the cost $c(x,y) = -\log(\cos^2(\min(d(x,y),\pi)))$ (where $d$ is a cost on the base space $X \times Y$) and for the Kullback--Leibler entropy functionals \eqref{eq:def-KL}.
We underline the fact that the cost may be unbounded, as it is the case for WFR.

This problem has an equivalent definition on the cone over the base space $X$ denoted by $\mathcal C(X) = (X \times \R_+)/\{(x,0) \, ;\, x \in X\}$. The product of cones is endowed with the following cost $c_\mathcal{C}: \mathcal{C}(X) \times \mathcal{C}(Y) \to \R \cup \{+\infty\}$,
\begin{equation}\label{EqConeFormulation}
c_{\mathcal{C}}((x,r),(y,s)) = \inf_{z >0} \left(rF_0(z/r) + sF_1(z/s) + c(x,y)z\right),
\end{equation}
with the convention that the value of the functions are respectively $[F_0]_\infty'$ and $[F_1]_\infty'$ if $r=0$ or $s = 0$.
If $c(x,y) = +\infty$ then $c_{\mathcal{C}}((x,r),(y,s)) = F_0(0)r + F_2(0)s$.
This cost is obtained by taking an infimum on the perspective function of the function $z \mapsto F_0(z) + F_1(z) + c(x,y)z$. As a consequence of the perspective function, it is one-homogeneous and convex in $(r,s)$. One of the main results in \cite{liero2018optimal} is the equivalence between Definition \ref{ThUOT} and an optimal transport problem with moment constraints instead of marginal constraints:
\begin{theorem}[Conic formulation, \cite{liero2018optimal}]\label{ThConicUOT}
Let $(\mu,\nu)\in\mathcal{M}_+(X)\times\mathcal{M}_+(Y)$. Define
\begin{equation}\label{EqConicFormulation}
\Cc\Tc(\mu,\nu) = \inf_{\gamma \in \tilde\Gamma} \int c_{\mathcal{C}}((x,r),(y,s)) \,d\gamma((x,r),(y,s))\,,
\end{equation}
where $\tilde\Gamma$ denotes the set of positive Radon measures $\gamma$ on $\mathcal{C}(X) \times \mathcal{C}(Y)$  such that 
\begin{equation}\label{EqMoments}
    \begin{cases}
         \mu(x) = \int_{\R_{+}} r [(p_1)_\#\gamma](x,dr)\,,\\
          \nu(y) = \int_{\R_{+}} s [(p_2)_\#\gamma](y,ds)\,,
    \end{cases}
\end{equation}
where $p_1,p_2$ are the two respective projections on $\mathcal{C}(X)$ and $\mathcal{C}(Y)$. It holds that
$$\Cc\Tc = \Uc\Tc\,.$$
In addition, the optimization set can be reduced to probability measures on $\mathcal{C}(X) \times \mathcal{C}(Y)$.
\end{theorem} 
Here are a few important instances of conic and static formulations.
\begin{example}\label{ExampleWFR}
    Choose $F_0 = F_1$ the Kullback--Leibler divergence \eqref{eq:def-KL}, also called relative entropy, and $d$ being a cost on $X\times Y$. Then, there are two important cases:
    \begin{itemize}
        \item The WFR case with $c(x,y) = -\log(\cos(\min(d(x,y), \pi))^2)$. The cost on the cone is 
\begin{equation}\label{EqConeCostLinear}
            c_{\mathcal{C}}((x,r),(y,s)) = r + s - 2\sqrt{rs}\cos(\min(d(x,y),\pi)\,.
        \end{equation}
        \item The Gaussian-Hellinger case with $c(x,y) = d(x,y)^2$, $c_{\mathcal{C}}((x,r),(y,s)) = r + s - 2\sqrt{rs}e^{-d(x,y)^2/2}$.
    \end{itemize}
    A last example is the partial optimal transport case for which $F(r) = |r -1|$ for $r\geq 0$ and $c(x,y)=d(x,y)^q$ with $q \geq 1$. One has $c_{\mathcal{C}}((x,r),(y,s))^q = r + s - (\min(r,s))\min(0,2 - d(x,y)^q)$.
\end{example}

We now prove a simple lemma, reminiscent of the lemma proven in \cite[Lemma 7.11]{liero2018optimal}.  This lemma uses the homogeneity of the cost and the fact that the constraints on the marginals are moment constraints, which are invariant to rescaling: namely, $m\delta_{(x,1)}$ and  $\delta_{(x,m)}$ have the same moment with respect to $m$ as defined in Eq. \eqref{EqMoments}. 
\begin{lemma}\label{ThLemmaReparametrization}
  Let $\mu \in \mathcal{M}_+(X)$,  $\nu \in \mathcal{M}_+(Y)$ non null, and $\tilde \nu \in \mathcal{P}(\mathcal{C}(Y))$ such that $\int_{\R_+}m \tilde \nu(y,dm) = \nu(y)$ and $\mathcal{UT}(\mu,\nu)$ is finite.
  Then, there exists $\tilde \mu \in \mathcal{P}(\mathcal{C}(X))$ such that $\int_{\R_+}m  \tilde \mu(x,dm) = \mu(x)$ and 
  \begin{equation}\label{EqEqualityOnLifts}
      \mathcal{UT}(\mu,\nu) = \mathcal{T}_{c_{\mathcal{C}}}(\tilde \mu,\tilde \nu)\,.
  \end{equation}
\end{lemma}
The proof of this lemma is deferred to Appendix \ref{AppendixOtherProofs}.
The previous lemma was needed to prove that unbalanced optimal transport is a cost submersion of optimal transport for the corresponding cost on the product of cones defined in Formula \eqref{EqConeFormulation}. Thus, our main result can be seen as an application of Proposition \ref{prop:sncc-surj}:
\begin{theorem}\label{ThMainTheoremUOT}
    The space $\mathcal{M}_+(X) \times \mathcal{M}_+(Y)$ endowed with the  $\mathcal{UT}$ cost is NNCC if the cost $c_\mathcal{C}$ on $\mathcal{C}(X) \times \mathcal{C}(Y)$ defined in Formula \eqref{EqConeFormulation} is NNCC.
\end{theorem}

\begin{proof}
There is a cost submersion between the space of probability measures endowed with the transport cost $\mathcal{T}_{c_{\mathcal{C}}}$ and the space of positive and finite Radon measures endowed with the $\mathcal{U}\mathcal{T}$ cost. To prove it, the needed properties to apply Proposition \ref{prop:sncc-surj} are given by Lemma \ref{ThLemmaReparametrization}.
\end{proof}
Let us underline that a direct proof can also be given using a similar strategy than for standard optimal transport.
To get more insight on the condition that the cone needs to be NNCC, the case of the Wasserstein--Fisher--Rao metric is instructive. When the base cost is the squared Riemannian distance on a Riemannian manifold, the cone cost gives a sectional curvature to the cone that is, in general, strictly less in some directions than the sectional curvature of the base manifold. Consequently, the base manifold needs to be positively curved to obtain the cone to be NNCC.
Although this condition is restrictive, we provide hereafter a couple of examples for the Wasserstein--Fisher--Rao metric. An open question is to find other examples in the case of the Hellinger--Kantorovich cost.

\begin{example}
In the WFR case (see Example \ref{ExampleWFR}),
    the simplest example is the $n$-dimensional sphere $X = Y =  \mathbb S_n$ endowed with its natural Riemannian distance.  Then, the cone $\mathcal{C}(\mathbb S_n)$ equipped with the cone cost is isometric to the Euclidean space $\R^{n+1}$ endowed with the Euclidean squared distance which is NNCC.
    Therefore, the space of nonnegative measures on the sphere endowed with the Wasserstein--Fisher--Rao metric is NNCC.
\end{example}

\begin{example}[Optimal transport on the cone of PSD matrices]
    Another important example is related to the Bures--Wasserstein metric.
    Indeed, the space of PSD matrices is itself a cone: Let us parametrize a (non zero) PSD matrix $S'$ by $(\operatorname{tr}(S'), S'/ \operatorname{tr}(S')) = (m,S)$ with $S \in \mathrm{S}^{n}_{+,1}$ the space of unit trace PSD matrices. Associated with the Bures metric, the Bures angle or quantum angle is defined between two unit trace matrices $S_0,S_1$:
    \begin{equation}
        D(S_1,S_2) \coloneqq \arccos\left( \operatorname{tr}\sqrt{S_0^{1/2}S_1 S_0^{1/2}}\right).
    \end{equation}
    Using Formula \eqref{EqConeCostLinear} and by definition, the cone over $\mathrm{S}^n_{+,1}$ endowed with the Bures length gives the Bures metric.
    In particular, the cone over the space of unit trace PSD matrices endowed with the usual Bures metric gives the Bures--Wasserstein metric when taking the WFR case. 
    Since $(\mathrm{S}^n_{+},\operatorname{BW}^2)$ is NNCC (see \Cref{thm:BW_NNCC}), Theorem \ref{ThMainTheoremUOT} applies: the unbalanced optimal transport metric on $\mathcal{M}_+(\mathrm{S}^{n}_{+,1})$ has the NNCC property.
\end{example}

A direct consequence of the cost submersion fact is that lifted c-segments in UOT are projections of lifted c-segments on the probability measures on the cone. 
\begin{remark}[On the necessity of the cone to be NNCC when WFR is NNCC]
In the Wasserstein--Fisher--Rao case, we conjecture an equivalence between the two properties when the diameter of the underlying manifold is less than $\pi/2$. Indeed, in such a case, the cone isometrically embeds into the space of positive Radon measures with the WFR metric. Above this cut-off distance, the isometric embedding is no longer true.
\end{remark}

\subsection{Gromov--Wasserstein} \label{sec:GW}

In this section we show that the set of all metric measure spaces, or more generally gauged spaces, equipped with the squared Gromov--Wasserstein distance is an NNCC space. 
This represents a significant advance over Sturm's discovery \cite{sturm2012space} that the ``space of spaces'' is positively curved, see \Cref{rem:GW-PC-NNCC}.  

A metric measure space is a triple $(X,f,\mu)$,
where $(X,f)$ is a complete separable metric space and $\mu$ is a Borel probability measure on $X$ with full support.
The Gromov--Wasserstein distance, as introduced by Mémoli \cite{memoli2011gromov} and also studied by Sturm \cite{sturm2012space}, is a way to compare two metric measure spaces in terms of their intrinsic qualities, \emph{up to isometries}. In particular one considers  equivalence classes $\bX=[X,f,\mu]$ defined as follows: two metric measure spaces $(X,f,\mu)$ and $(Y,g,\nu)$ are equivalent if there exists a map $T\colon X \to Y$ which is an isometry on the support of $\mu$ and such that $T_\# \mu = \nu$. The set of all equivalence classes of metric measure spaces will be denoted by $\Gauged$. 

The Gromov--Wasserstein distance between $\bX=[X,f,\mu]$ and $\bY=[Y,g,\nu]$ is defined by
\begin{equation}\label{eq:def-GW}
    \GW^2(\bX,\bY)=\inf_{\pi\in\Pi(\mu,\nu)}\int\abs{f(x,x')-g(y,y')}^2\,d\pi(x,y)\,d\pi(x',y')\,.
\end{equation}
This is a metric on $\Gauged$ valued in $[0,+\infty]$. 
From an optimization viewpoint, problem \eqref{eq:def-GW} is harder than optimal transport since it
is a nonconvex quadratic optimization problem. However, as shown in \cite[Theorem 3.1]{sturm2012space}, the structure of geodesics is rather simple. Once an optimal plan $\pi$ is found between $\bX$ and $\bY$, the geodesic consists of $\bX(s) = [X\times Y , (1-s)f + s g,\pi]$, which suggests there may also be a particularly simple form of variational c-segments. This is indeed the case:

\begin{proposition}[Variational c-segments for Gromov--Wasserstein]\label{prop:GW-vcs}
    Let $\bX_0=[X_0,f_0,\mu_0]$, $\bX_1=[X_1,f_1,\mu_1]$, $\bY=[Y,g,\nu]\in\Gauged$ be such that $\GW(\bX_i,\bY)<+\infty$ and let $\pi_i\in\Pi(\mu_i,\nu)$ be an optimal plan for $\GW(\bX_i,\bY)$ ($i=0,1$). Consider any $3$-plan $\gamma\in\calP(X_0\times X_1\times Y)$ such that $(\p_1,\p_3)_\#\gamma=\pi_0$ and $(\p_2,\p_3)_\#\gamma=\pi_1$, and define for every $s\in (0,1)$
    \begin{equation}\label{eq:c-segment-GW}
        \bX(s)=[X_0\times X_1, f(s)\coloneqq (1-s)f_0+s\,f_1,m\coloneqq (\p_1,\p_2)_\#\gamma]\,,
    \end{equation}
    and $\bX(0)=\bX_0$, $\bX(1)=\bX_1$.
    In the expression of $f(s)$, $f_0$ acts on the $X_0$ components and $f_1$ on the $X_1$ components of $(X_0\times X_1)^2$. Then $s\mapsto (\bX(s),\bY)$ is a variational c-segment on $(\Gauged\times\Gauged,\GW^2)$. 
\end{proposition}

Since \eqref{eq:def-GW} always admits a solution \cite[Lemma 1.7]{sturm2012space}, we find as an immediate corollary:

\begin{theorem}
   $(\Gauged\times\Gauged,\GW^2)$ is an NNCC space.  
\end{theorem}

\begin{proof}[Proof of \Cref{prop:GW-vcs}]
    Let $\bX_0=[X_0,f_0,\mu_0]$, $\bX_1=[X_1,f_1,\mu_1]$, $\bY=[Y,g,\nu]$, $\bZ=[Z,h,\sigma]$ denote four (equivalence classes of) metric measure spaces. For $i=0,1$, let $\pi_i\in\Pi(\mu_i,\nu)$ be an optimal Gromov--Wasserstein coupling  of $(\bX_i,\bY)$. Then
    \begin{equation}\label{eq:proof-GW-1}
        \GW^2(\bX_i,\bY) = \int\abs{f_i(x_i,x_i')-g(y,y')}^2\,d\pi_i(x_i,y)\,d\pi_i(x_i',y')\,.
    \end{equation}
    Define $\bX(s)$ by \eqref{eq:c-segment-GW} and fix $s\in [0,1]$.
    Since $\gamma$ is a coupling of $(m,\nu)$ we have the inequality
    \begin{equation}\label{eq:proof-GW-2}
        \GW^2(\bX(s),\bY)\leq \int\abs{f(s)((x_0,x_1),(x_0',x_1'))-g(y,y')}^2\,d\gamma(x_0,x_1,y)\,d\gamma(x_0',x_1',y')\,.
    \end{equation}
    Let $\tilde\pi\in\Pi(m,\sigma)$, which we will optimize over in a moment to obtain the term $\GW^2(\bX(s),\bZ)$. Assume without loss of generality that 
    \[
    \int\abs{f(s)((x_0,x_1),(x_0',x_1'))-h(z,z')}^2\,d\tilde\pi(x_0,x_1,z)\,d\tilde\pi(x_0',x_1',z')<+\infty\,,
    \]
    for if the above value is infinite then the desired conclusion \eqref{eq:proof-GW-4} holds automatically. Glue $\gamma$ and $\tilde\pi$ along $(x_0,x_1)$ into a $4$-plan $\gamma^4\in\calP(X_0\times X_1\times Y\times Z)$. Using \eqref{eq:proof-GW-2} we may write
    \begin{multline}\label{eq:proof-GW-3}
        \GW^2(\bX(s),\bY)-\int\abs{f(s)-h}^2\,d\tilde\pi(x_0,x_1,z)\,d\tilde\pi(x_0',x_1',z') \\
        \leq \int\Big(\abs{f(s)-g}^2-\abs{f(s)-h}^2\Big)\,d\gamma^4(x_0,x_1,y,z)\,d\gamma^4(x_0',x_1',y',z')\,,
    \end{multline}
    where $f(s)$ depends on $((x_0,x_1),(x_0',x_1'))$, $g$ on $(y,y')$ and $h$ on $(z,z')$. By expanding the square, we see that $\abs{f(s)-g}^2-\abs{f(s)-h}^2$ is an affine function of $s$. Therefore the right-hand side of \eqref{eq:proof-GW-3} is affine in $s$ and we can write  
    \begin{multline}\label{eq:proof-GW-4}
        \GW^2(\bX(s),\bY)-\int\abs{f(s)-h}^2\,d\tilde\pi\otimes d\tilde\pi \\
        \leq (1-s)\int\Big(\abs{f(0)-g}^2-\abs{f(0)-h}^2\Big)\,d\gamma^4\otimes d\gamma^4+s\int\Big(\abs{f(1)-g}^2-\abs{f(1)-h}^2\Big)\,d\gamma^4\otimes d\gamma^4 \\
        = (1-s)\Big(\int\abs{f_0-g}^2\,d\pi_0\otimes d\pi_0-\int\abs{f_0-h}^2\,d\tilde\pi\otimes d\tilde\pi\Big)+s\Big(\int\abs{f_1-g}^2\,d\pi_1\otimes d\pi_1-\int\abs{f_1-h}^2\,d\tilde\pi\otimes d\tilde\pi\Big).
    \end{multline}
    We recognize \eqref{eq:proof-GW-1} for $i=0,1$ in the right-hand side. Finally taking a supremum over $\tilde\pi\in\Pi(m,\sigma)$ we obtain the NNCC inequality
    \begin{multline*}
        \GW^2(\bX(s),\bY)- \GW^2(\bX(s),\bZ) \\
        \leq (1-s)\Big(\GW^2(\bX_0,\bY)-\GW^2(\bX_0,\bZ)\Big) + s \Big(\GW^2(\bX_1,\bY) -\GW^2(\bX_1,\bZ)\Big).
    \end{multline*}
    Note in particular that choosing $\bZ=\bX(s)$ implies finiteness of $\GW^2(\bX(s),\bY)$. 
\end{proof}

\begin{remark}
In the proof of \Cref{prop:GW-vcs}, the fact that $f_0,f_1,g$ are distances is not used. Therefore the result holds for more general functions, such as gauge functions \cite[Section 5]{sturm2012space}.
\end{remark}

\begin{remark}[NNCC and PC]\label{rem:GW-PC-NNCC}
    Consider the subspace $\Gauged_0\subset\Gauged$ of metric measure spaces $[X,f,\mu]$ with finite size $\int \abs{f(x,x')}^2 d\mu(x)d\mu(x')<+\infty$. Sturm proved that $(\Gauged_0,\GW)$ is a geodesic space that is positively curved in the sense of Alexandrov \cite[Theorem 4.4]{sturm2012space}. The proof of \Cref{prop:GW-vcs} shows just as well $(\Gauged_0\times\Gauged_0,\GW)$ to be an NNCC space, a stronger property than positive curvature by \Cref{thm:NNCC-implies-PC}.
\end{remark}

\subsection{Procrustes--Wasserstein}

In a similar spirit to Gromov--Wasserstein distances, comparing probability measures in the Euclidean space up to isometries has been addressed in the applied mathematics literature: one candidate is the Procrustes--Wasserstein distance, which is defined by
\begin{equation}
    \operatorname{PW}(\mu,\nu) = \inf_{R,S \in \operatorname{Isom}(\R^n)} W_2(R_\# \mu,S_\# \nu)\,,
\end{equation}
where $R,S$ are isometries of $\R^n$ and $\mu,\nu$ are two probability measures on $\R^n$. Sometimes, it is also defined as the quotient with the group of rotations only after the alignment of the mean of the two measures, see \cite{Grave2018UnsupervisedAO}.

\begin{theorem}
    The space $\mathcal{P}_2(\R^n)/\operatorname{Isom}(\R^n)$ endowed with the $\operatorname{PW}^2$ cost is NNCC.
\end{theorem}

\begin{proof}
    Since $(\mathcal{P}(\R^n),W_2^2)$ is NNCC, Proposition \ref{prop:sncc-surj} on cost submersions shall imply the theorem. We now check that the hypotheses of Proposition \ref{prop:sncc-surj} are fulfilled.
    Note that by a change of variables, one has $W_2(R_\# \mu, S_\# \nu) = W_2( \mu, [R^{-1} S]_\# \nu)$.
    Under the additional fact that the infimum of the minimization problem is attained in $\operatorname{Isom(\R^n)}$, then the formula \eqref{eq:c-subm-property} follows directly. To prove the existence of minimizers, we prove that one can restrict the optimization to a compact subset of isometries. Note that $\operatorname{Isom}(\R^n) = \R^n \ltimes O(n)$ is a semidirect product of groups between translations and rotations. It is sufficient to reduce the set of translations to a compact set. For that, we note that by the convexity of the Euclidean distance squared $W_2(\mu,\nu) \geq \| \operatorname{mean}(\mu) - \operatorname{mean}(\nu)\|$. As an immediate consequence, the norm of the translation part of $R^{-1}S$ is bounded by $W_2(\mu,\nu)$. 
\end{proof}

\begin{remark}
Note that the result is obvious when one considers the quotient only with $O(n)$ instead of $\operatorname{Isom}(\R^n)$ due to the compactness of $O(n)$.
    In the same direction, this result directly extends to Riemannian manifolds $(M,g)$ that satisfy the two conditions: 
    \begin{enumerate}[(i)]
        \item $(M,d^2)$ is NNCC where $d$ is the geodesic distance.
        \item The isometry group of $(M,g)$ is compact.
    \end{enumerate}
    Indeed, in such a case, it satisfies the hypotheses of a cost submersion in Definition \ref{DefCostSubmersion} and one can apply Proposition \ref{prop:sncc-surj}.
\end{remark}

\appendix

\section{Infinite values} \label{sec:app:infinite}

In this appendix, we show how to manipulate infinite values taken by the cost function and by potentials. 
Here are our main references. Moreau \cite[Section 4.c.]{Moreau1970} considers costs taking both $\pm\infty$ values and explains how to deal with undefined expressions of infinites. He has opposite signs in regard to potentials. Villani \cite[Chapter 5]{villani2009optimal} is a main reference, he only considers costs taking the $+\infty$ value. Rockafellar and Wets \cite[Chapter 1]{RockefellarWetsbook} define epigraphs and hypographs of extended-real-valued functions in standard convex analysis.

\subsubsection*{Hypographs}
Since here we want to work with c-concave functions we consider hypographs.
Given an arbitrary set $X$ and a function $\psi\colon X\to[-\infty,+\infty]$, the hypograph of $\psi$ is the subset of $X\times \R$ defined by 
\begin{equation}
    \hypo \psi = \{(x,\alpha)\in X\times\R : \alpha\leq \psi(x)\}.
\end{equation}
Note that $\alpha$ is always finite, thus $\hypo\psi$ is a subset of the product of $X$ with the real numbers, and not with the extended real numbers.
In particular given $x\in X$, if $\psi(x)=-\infty$ then then the intersection of $\hypo f$ and the vertical line $\{x\}\times \R$ is empty, and if $\psi(x)=+\infty$ then $\hypo f$ contains the vertical line $\{x\}\times \R$. 
If $\psi(x)$ is finite then the intersection of $\hypo f$ and $\{x\}\times \R$ is $\{x\}\times (-\infty,\psi(x)]$ and we call $(-\infty,\psi(x)]$ the ``fiber above $x$''.
The domain of $\psi\colon X\to[-\infty,+\infty]$ is the set of $x$'s such that the vertical line $\{x\}\times\R$ intersects the hypograph of $\psi$, i.e.\ $\dom \psi = \{x\in X : \psi(x)>-\infty\}$. 

Let $\Ac$ denote the class of all subsets $A\subset X\times\R$ that satisfy the property:
\begin{enumerate}
    \item[] \label{property-hypographs}
    For every $x\in X$, the fiber above $x$ is either empty or a closed interval of $\R$ unbounded to the left (this includes $\R$ itself). 
\end{enumerate}
Clearly $A$ is the hypograph of a function $\psi\colon X\to [-\infty,+\infty]$ if and only if it belongs to $\Ac$, and  $\hypo(\cdot)$ realizes a one-to-one correspondence between functions $X\to [-\infty,+\infty]$ and  $\Ac$. Given $A\in\Ac$, we may find the corresponding function by 
\begin{equation}\label{eq:psi-from-hypo}
    \psi(x)=\sup\{\alpha\in\R : (x,\alpha)\in A\},
\end{equation}
with $\sup\R=+\infty$, and the convention $\sup \emptyset = -\infty$.

A simple but important property of the class $\Ac$ of hypographs is that it is closed under arbitrary intersections. Indeed we can reason fiber by fiber and then use the fact that a class of intervals of $\R$ that are closed and unbounded to the left is closed under arbitrary intersections.

Let us now consider a cost function $c\colon X\times Y\to[-\infty,+\infty]$, where $Y$ is another arbitrary set. To fix ideas, we will deal with functions $\psi\colon X\to[-\infty,+\infty]$ and $\phi\colon Y\to[-\infty,+\infty]$ that satisfy the sum $\psi(x)+\phi(y)\leq c(x,y)$. 

\subsubsection*{c-transform}
The c-conjugate of an arbitrary subset $B\subset Y\times\R$ is the subset $B^c\subset X\times\R$ defined by 
\begin{equation*}
    B^c := \bigcap_{(y,\beta)\in B} \hypo (c(\cdot,y)-\beta).
\end{equation*}
If $B$ is empty then $B^c=X\times \R$ by convention. 
We see that for any $B\subset Y\times\R$,  $B^c\in\Ac$. More precisely and in view of \eqref{eq:psi-from-hypo}, $B^c$ is the hypograph of the function $\psi\colon X\to[-\infty,+\infty]$ defined by 
\begin{align}
    \psi(x)&=\sup\{\alpha\in\R : \quad \forall (y,\beta)\in B,\,\, \alpha+\beta\leq c(x,y)\}\label{eq:c-transf-alpha-beta}\\
    &=\inf_{(y,\beta)\in B} c(x,y)-\beta.
\end{align}
When $B=\hypo\phi$ for some $\phi\colon Y\to[-\infty,+\infty]$, we define $\phi^c\colon X\to[-\infty,+\infty]$ by $\hypo (\phi^c) = (\hypo \phi)^c$ and call $\phi^c$ the c-transform of $\phi$. We see that 
\begin{equation}\label{eq:def-c-transform}
    \begin{aligned}
        \phi^c(x) &= \inf_{(y,\beta)\in\hypo\phi} c(x,y)-\beta\\
        &=\inf_{y\in Y} c(x,y)-\phi(y),    
    \end{aligned}
\end{equation}
where in the last infimum the following rule should be used: $(+\infty)-(+\infty)=+\infty$ and $(-\infty)-(-\infty)=+\infty$. It may be convenient to remember that the rule for undefined expressions $(+\infty)-(+\infty)$ and $(-\infty)-(-\infty)$ is to give the neutral element for $\inf$, i.e.\ $+\infty$. By exchanging the roles of $x$ and $y$ we similarly define c-conjugates of subsets $A\subset X\times R$ and c-transforms of functions $\psi\colon X\to [-\infty,+\infty]$, still denoted by $A^c$ and by $\psi^c$ respectively.

We say that a function $\psi\colon X\to[-\infty,+\infty]$ is c-concave if there exists $B\subset Y\times \R$ such that $\hypo \psi=B^c$. This is equivalent to the existence of a function $\phi\colon Y\to[-\infty,+\infty]$ such that $\psi=\phi^c$, or to $\psi=\psi^{cc}$. We may similarly define c-concave functions on $Y$.

\subsubsection*{Contact set}
We say that two subsets $A\subset X\times\R$ and $B\subset Y\times \R$ are \emph{sub-conjugate} if $A\subset \hypo(c(\cdot,y)-\beta)$ for every $(y,\beta)\in B$. This can be written as
\begin{equation}\label{eq:def-sub-conjugate}
\forall (x,\alpha)\in A,\,\,\forall (y,\beta)\in B,\quad \alpha+\beta\leq c(x,y),
\end{equation}
which shows sub-conjugation to be a symmetric relation for $A$ and $B$. Sub-conjugation is called \emph{sur-dualité} by Moreau who has opposite signs \cite{Moreau1970}. Given two sub-conjugate subsets $A$ and $B$ and $(x,\alpha)\in A$, their contact set is defined by 
\begin{equation}\label{eq:def-contact-set}
\contact(A,B)=\{(x,\alpha,y,\beta)\in X\times \R\times Y\times \R : c(x,y)\leq \alpha+\beta\},
\end{equation}
and by \eqref{eq:def-sub-conjugate} the inequality is in fact an equality. Therefore the quantity $c(x,y)$ in \eqref{eq:def-contact-set} is finite.

We say that two functions $\psi\colon X\to[-\infty,+\infty]$ and $\phi\colon Y\to[-\infty,+\infty]$ are sub-conjugate if the sets $A=\hypo\psi$ and $B=\hypo\phi$ are sub-conjugate. We see that this can be written as 
\begin{equation}\label{eq:app-sub-conjugate-inequality}
\forall (x,y)\in X\times Y,\quad \psi(x)+\phi(y)\leq c(x,y),
\end{equation}
with the rule $(+\infty)+(-\infty)=(-\infty)+(+\infty)=-\infty$ in the left-hand side; in other words \eqref{eq:app-sub-conjugate-inequality} is required to hold whenever the left-hand side is well-defined. By \eqref{eq:def-c-transform} this is also equivalent to $\psi\leq\phi^c$ and $\phi\leq\psi^c$. In particular $\phi^c$ and $\phi$ are always sub-conjugate. The contact set of $A=\hypo\phi^c$ and $B=\hypo\phi$
is called the c-subdifferential of $\phi$, and can be encoded as only those pairs $(x,y)\in X\times Y$ that satisfy $c(x,y)\leq \phi^c(x)+\phi(y)$, since combining \eqref{eq:def-sub-conjugate} and \eqref{eq:def-contact-set} forces $\alpha=\phi^c(x)$ and $\beta=\phi(y)$ (and in turn these are finite). More generally since a sub-conjugate pair $(\psi,\phi)$ satisfies $\psi\leq\phi^c$ we may ask for $c(x,y)\leq\psi(x)+\phi(y)$. In short, the c-subdifferential of an arbitrary function $\phi\colon Y\to[-\infty,+\infty]$ can be defined as 
\begin{multline}\label{eq:def-c-subdifferential}
    \partial^c\phi=\{(\xb,\yb)\in X\times Y : \text{$\phi(\yb), c(\xb,\yb)$ are finite and} \\
    \text{there exists $\psi\colon X\to[-\infty,+\infty]$ sub-conjugate to $\phi$ such that } c(\xb,\yb)\leq\psi(\xb)+\phi(\yb)\}.
\end{multline}
We can then set for any $\yb\in Y$,
\begin{equation}\label{eq:def-c-subdifferential-y}
\partial^c\phi(\yb)=\{\xb\in X : (\xb,\yb)\in\partial^c\phi\}.
\end{equation}
We could similarly define the c-subdifferential of a function defined on $X$.

\section{Proofs} \label{sec:appendix-proofs}

\subsection{Proof of Theorem \ref{ThPrincipalFiberResult}}\label{sec:appendix-submersion}

In this section we prove our extension of the Riemannian submersion result of Kim and McCann to the case of a principal fiber bundle endowed with a cost diagonally invariant to a group action. Let us recall first Definitions \ref{def:proj-fiber} and \ref{ThInvariantCostDefinition}.

\begin{definition}[Submersion and projection of a cost.]\label{def:proj-fiber-app}
Let $P_i: M_i \to N_i$ for $i = 1,2$ be two submersions between manifolds with compact fibers and ${c}: M_1 \times M_2 \to \R$ be a continuous cost. 
Define the projected cost $\underline{c}: N_1 \times N_2 \to \R$ by 
\begin{equation}
\underline c(\underline x,\underline y) \coloneqq \min \{ {c}({x},{y})\,;\, {x} \in P_1^{-1}(\{\underline x\})\text{ and }{y} \in P_2^{-1} (\{\underline y \}) \}\,.
\end{equation}
\end{definition}
\begin{definition}[Principal bundle with $c$-invariant cost]\label{ThInvariantCostDefinition-app}
Consider a principal fiber bundle $(G,L_1 \times L_2)$ where $G$ is a compact Lie group and a continuous cost $c: L_1 \times L_2 \mapsto \R$ such that $G$ leaves the cost invariant. Namely, we assume 
    \begin{equation}
        c(g \cdot x,g\cdot y) = c(x,y)\,, \quad \forall (x,y) \in L_1 \times L_2\,.
    \end{equation}
    We consider the projection $P : L_1 \times L_2  \to L_1/G \times L_2 /G$.
The corresponding projected cost is denoted $\underline c$.
\end{definition}

Under the assumptions of Definition \ref{ThInvariantCostDefinition-app},
let $y \in L_2$ and $\underline x \in L_1/G$. Then, we can define the set of optimal lifts of $\underline{x}$ at point $y$ by $x(y) = \arg \min_{x \in \pi^{-1}(\{\underline{x}\})} c(x,y)$. Similarly, we can define the set of optimal lifts $y(x)$.
If this minimization leads to a unique point, this enables the definition of lift of curves $(\underline{x}(s),\underline{y})$ by $(x(s),y)$.
The fact that the optimal lifts exist to any given $y \in L_2$ is because the action of $G$ on the product $L_1 \times L_2$ leaves the cost invariant so that one can fix the pre-image of $\underline{x}$: $c(g_1 \cdot x,g_2 \cdot y) = c((g_2^{-1}g_1) \cdot x, y)$.

We make the following assumption:

\begin{assumption}\label{AssumptionLifts}
    The lifting procedure is well-defined for curves, meaning that locally, there is a unique point at which the infimum cost is attained and that it is a smooth operation. 
\end{assumption}
\noindent
In particular, if the curve $\underline{\rmx}$ is $C^k$ for $k\geq 0$ then the curve $\rmx$ is also $C^k$. This is satisfied for instance in the case of Riemannian submersions which have the structure of a principal fiber bundle.
Assuming that the lift of curves is well-defined implies that the lift of tangent vectors is also well-defined.
We call it horizontal lift, similar to the Riemannian submersion case.
The most important fact in this construction is that the lifts of c-segments are c-segments and this fact relies crucially on the invariance of the cost:

\begin{proposition}[Lift of c-segments]
Under Assumption \ref{AssumptionLifts},
    the lift of a c-segment on $ L_1/G \times L_2 /G$ is a c-segment on $L_1 \times L_2$.
\end{proposition}
\begin{proof}
    The definition of a c-segment on $ L_1/G \times L_2 /G$ reads
    \begin{equation}\label{Eqcsegmentdownstairs}
        \nabla_y \underline{c}(\underline{\rmx}(s),\underline{y}) = (1-s)\nabla_{\underline{y}} c(\underline{x}_0,\underline{y})+s\nabla_{\underline{y}} \underline{c}(\underline{x}_1,\nabla{y})\,,\quad \forall s\in[0,1]\,.
    \end{equation}
    Consider the lift of the c-segment denoted by $(\rmx,y)$. Due to the invariance of the cost on $L_1 \times  L_2$, we have $c(g\cdot \rmx(s),y) = c(\rmx(s),g^{-1}\cdot y)$. Now, due to the optimality of the lift, the first-order optimality condition gives
    \begin{equation}\label{EqVanishingMomentum}
         \nabla_y c(\rmx(s),\xi \cdot y) = 0\,,
    \end{equation}
    for all $s \in [0,1]$ and $\xi \in T_{\operatorname{Id}}G$. Note that the notation $\xi \cdot y$ is defined as the tangent vector at $y$ given by the differentiation of a curve $g(t) \cdot y$ starting at identity with initial tangent vector $\xi$. 
    The previous formula means that the differentiation of the cost in the vertical direction (the tangent space to the fiber) vanishes. By local trivialization of the fiber bundle, 
    equations \eqref{Eqcsegmentdownstairs} and \eqref{EqVanishingMomentum} imply that $(\rmx,y)$ is a c-segment on $L_1 \times L_2$.
\end{proof}

This fact allows us to prove \Cref{ThPrincipalFiberResult}.

\begin{proof}[Proof of \Cref{ThPrincipalFiberResult}]
    The proof follows the one of \cite[Theorem 4.5]{kim2012towards} and we include it here for the reader's convenience. Consider $(\rmx,y)$ the lift at point $y$ of the c-segment $(\underline{\rmx},\underline{y})$ in a neighborhood of $0$. Consider also a curve also defined on a neighborhood of $0$, $\underline{y}(t)$ with tangent vector $\eta$ at $t = 0$ and the lift of $y(t)$ at the point $x(0)$. One can now consider the function 
    $F(s,t) = c(\rmx(s),y(t)) - \underline{c}(\underline{\rmx}(s),\underline{y}(t))$ which is nonnegative by definition of $\underline{c}$.
    We now use the definition of the cross-curvature
    \begin{equation}
        \mathfrak{S}_c(x,y)(\xi,\eta)=-\frac{\partial^4}{\partial_s^2\partial_t^2}\Big|_{s=t=0}c(\rmx(s),y(t)) \,,
    \end{equation}
    so that in order to prove the result it is sufficient to prove that $\frac{\partial^4}{\partial_s^2\partial_t^2}\big|_{s=t=0}F \geq 0$. 
    In addition to nonnegativity, $F(s,0) = F(0,t) = 0$ by its definition which implies that $F$ achieves its minimum on these two sets and the Hessian of $F$ is a positive semi-definite matrix. This Hessian must vanish since the diagonal terms are $0$ by differentiation of $F(s,0) = F(0,t) = 0$ and it implies by nonnegativity of the Hessian matrix $\partial_{st}F(0,0) = 0$. It implies that the Taylor expansion of $F$ takes the form
    \begin{equation}
        F(s,t) = f_{21} t^2s + f_{12} ts^2 + f_{31} t^3s + f_{13} ts^3 + f_{22}t^2 s^2\,,
    \end{equation}
    with similar notations to \cite[Proof of Theorem 4.5]{kim2012towards}, where the authors proved that $f_{21} = f_{12} = f_{13} = f_{31} = 0$ by differentiation of $F(t,\pm t) \geq 0$ and $F(t^2,t) \geq 0$. The Taylor expansion of $F(s,t)$ is reduced to $f_{22}t^2s^2$ so that nonnegativity of $F$ directly implies $f_{22} \geq 0$ which is the result.
\end{proof}

\begin{remark}
The result of the theorem is local. As a consequence, Assumption \ref{AssumptionLifts} can be replaced with a local version, i.e. the optimal lifts are well-defined in a neighborhood of $(x,y)$.
\end{remark}

\subsection{Proof of Theorem \ref{thm:nncc-conv-lift}}\label{sec:snccconv-lift-v2}

Condition \snccconv{} is a stronger property with respect to NNCC as it requires the function $s\mapsto c(\rmx(s),\ov y)-c(\rmx(s),y)$ to be convex along variational c-segments $(\rmx,y)$, for any $y\in Y$. That is, the NNCC inequality must be satisfied for any subinterval $[a,b]\subset [0,1]$: The restriction of a variational c-segment $(\rmx(s),\ov y)$ from $[0,1]$ to $[a,b]$ should remain a variational c-segment between $(\rmx(a),\ov y)$ and $(\rmx(b),\ov y)$. 
In order to lift condition \snccconv{} to the Wasserstein space, we need to refine the definition of lift.
We make the following assumption:

\begin{assumption}\label{hyp:continuity-var-c-seg}
    \snccconv{}-variational c-segments on $(X\times Y,c)$ are continuous curves on $(0,1)$ with limits at $0$ and $1$.
\end{assumption}
\noindent
We denote by $\mathcal{S}([0,1], X)$ the space of curves valued in $X$ that are continuous on $(0,1)$ with limit at $0$ and $1$.

\begin{definition}[\snccconv{}-lifted c-segments]\label{def:lifted_csegment-nnccconv}

    Let $(s\in[0,1])\mapsto \mu(s)$ be a path in $\calP(X)$ and let $\nu\in\Pc(Y)$. We say that $s\mapsto (\mu(s),\nu)$ is an \emph{\snccconv{}-lifted c-segment} from $(X\times Y,c)$ if there exist a measurable set $\Gamma\subset X\times X\times Y,$ a collection of measurable maps $\Lambda_s\colon \Gamma\to X$ ($s\in [0,1]$), and a $3$-plan $\gamma_\nu \in \mathcal{P}(X \times X \times Y)$ such that
    \begin{enumerate}[(i)]
        \item $\gamma_\nu$ is concentrated on $\Gamma$, i.e. $\gamma_\nu(\Gamma)=1$, and $\mu(s)=(\Lambda_s)_\#\gamma_\nu$ and $\nu=(\p_3)_\#\gamma_\nu$ ;
        \label{enum:nnccconvlcs-gamma}
        
        \item $\Lambda_0(x_0,x_1,y)=x_0$, $\Lambda_1(x_0,x_1,y)=x_1$, and $s\mapsto (\Lambda_s(x_0,x_1,y),y)$ is an \snccconv{} variational c-segment on $(X\times Y,c)$ for each $(x_0,x_1,y)\in\Gamma$ ; \label{enum:nnccconvlcs-Lambda}
        
        \item$\pi_0\coloneqq (\p_1,\p_3)_\# \gamma_\nu$ and $\pi_1\coloneqq (\p_2,\p_3)_\# \gamma_\nu$ are respective optimal couplings of $(\mu(0),\nu)$ and $(\mu(1),\nu)$ for the cost $c$;\label{enum:nnccconvlcs-pi}

        \item there exists $\Lambda:\Gamma \to \mathcal{S}([0,1], X)$ such that $\Lambda_s \coloneqq \operatorname{ev}_s \circ \Lambda$, where $\operatorname{ev}_s: \mathcal{S}([0,1], X) \to X$ is the evaluation at time $s \in [0,1]$.
        
    \end{enumerate}
\end{definition}

The set $\Gamma$ is defined as in \eqref{eq:maximal-Gamma}.
Note that since the cost function is now continuous, it is in particular real valued and $\Gamma$ coincides with $X\times X\times Y$.
With respect to Definition \ref{def:lifted_csegment}, we require now the additional condition that the collection of measurable maps $\Lambda_s$ have a specific form, composition of the evaluation at time $s$ and a measurable map $\Lambda$ mapping from $\Gamma \subset X\times X\times Y$ to the set of curves $\mathcal{S}([0,1], X)$.
Assumption \ref{hyp:continuity-var-c-seg} ensures that the collection of measurable maps $\Lambda_s$ exists.

\begin{lemma}[Existence of \snccconv{}-lifted c-segments]\label{lemma:existence-lifted-v2}
    Let $(X\times Y,c)$ satisfies condition \snccconv{}, for $X$ and $Y$ Polish spaces and $c\colon X\times Y\to\R$ a continuous cost function bounded from below, satisfying further Assumption \ref{hyp:continuity-var-c-seg}.
    Then for any $\mu_0,\mu_1\in\Pc(X)$ and $\nu\in\Pc(Y)$ such that $\Tc_c(\mu_0,\nu)<+\infty$ and $\Tc_c(\mu_1,\nu)<+\infty$, there exists a \snccconv{}-lifted c-segment from $(X\times Y,c)$ between $(\mu_0,\nu)$ and $(\mu_1,\nu)$.
\end{lemma}
\begin{proof}
    Compared to Lemma \ref{lemma:existence-lifted}, we need now to guarantee the existence of a measurable selection $\Lambda$ as a map from $\Gamma$ to the set of curves valued in $X$ that are continuous on $(0,1)$ with limit at $0$ and $1$.
     
    The \snccconv{} property guarantees the existence of an \snccconv{}-variational c-segment for each triplet $(x_0,x_1,y) \in \Gamma$. We aim to construct a measurable map $\Lambda_s$ mapping each triplet to the evaluation at time $s\in [0,1]$ of an associated \snccconv{}-variational c-segment.
    Consider then the multi-valued map $\tilde \Lambda: \Gamma \to \mathcal{S}([0,1], X)$ which maps a given triplet $(x_0,x_1,y) \in \Gamma$ to the set of \snccconv{}-variational c-segments associated with it.
    By the assumptions, the graph of $\tilde \Lambda$ is the set
    \begin{multline*}
    \big\{
    (x_0,x_1,\ov{y},\rmx) \in \Gamma \times \mathcal{S}([0,1], X) \;: \\
    \text{$(\rmx,y)$ is a \snccconv{}-variational c-segment from $(x_0,\ov{y})$ to $(x_1,\ov{y})$}
    \big\},
    \end{multline*}
    where $\mathcal{S}([0,1], X)$ is endowed with the uniform norm topology. Note that this is a Suslin space (see \cite{bogachev2007measure}) since it is isomorphic to $C([0,1],X)$.
    By \cite[Theorem 6.9.13]{bogachev2007measure}, it is sufficient to show that the graph of $\tilde \Lambda$ is closed to guarantee the existence of a measurable selection $\Lambda$ as in Definition \ref{def:lifted_csegment-nnccconv}. This is true since the convex constraint in \ref{condition-NNCC-conv} can be formulated as an intersection of closed subsets due to the continuity of the cost. Again by continuity of the curves on $(0,1)$, the evaluation $\operatorname{ev}_s$ is continuous for $s \in (0,1)$ and therefore the map $\Lambda_s \coloneqq \operatorname{ev}_s \circ \Lambda$ is measurable.
\end{proof}

We can now present the proof of Theorem \ref{thm:nncc-conv-lift}, which is an adaptation of the proof of Thereom \ref{ThMainTheorem}.

\begin{proof}[Proof of Theorem \ref{thm:nncc-conv-lift}]
    Thanks to the new definition of lift, for any triplet $(x_0,x_1,y)\in \Gamma$, the following inequality holds for any $0\le a\le b\le 1$
    \[
    \begin{aligned}
        c(\Lambda_s(x_0,x_1,y),y)-c(\Lambda_s(x_0,x_1,y),z)\leq& (1-s)[c(\Lambda_a(x_0,x_1,y),y)-c(\Lambda_a(x_0,x_1,y),z)]\\
        &+s[c(\Lambda_b(x_0,x_1,y),y)-c(\Lambda_b(x_0,x_1,y),z)]\,.
    \end{aligned}
    \]
    Moreover, as pointed out in Remark \ref{rmk:pis-optimal}, the transport plan $\pi(s)$ as defined in the proof of Proposition \ref{prop:var_lifted_c-seg} is optimal for any $s\in[0,1]$. These two facts combined ensure that one can repeat the proof of Proposition \ref{prop:var_lifted_c-seg} on any subinterval $[a,b]\subset[0,1]$, implying that the restriction of a lifted c-segment to $[a,b]$ is again a variational c-segment in $(\Pc(X)\times\Pc(Y),\Tc_c)$. \snccconv{}-lifted c-segments satisfy therefore condition \snccconv{}. The claim of Theorem \ref{thm:nncc-conv-lift} is a direct consequence.
\end{proof}

\subsection{Proof of Proposition \ref{ThMainTheorem_converse}}\label{ssec:appendix-converse}

To prove Proposition \ref{ThMainTheorem_converse}, we follow the strategy of \Cref{lemma:auto-csegment} that relies on the differentiability of the cost. In this case, the cost is the transport cost $\Tc_c$ and its differentiability can be obtained when the reference measure $\nu$ has density w.r.t. Lebesgue.
We recall that a c-segment (Definition \ref{def:c-segment}) from $(x_0,y)$ to $(x_1,y)$ is a path $(s\in[0,1])\mapsto (\rmx(s),y)\in X\times Y$ that satisfies
\begin{equation}\label{eq:c-segment-recall}
    \nabla_yc(\rmx(s),y)=(1-s)\nabla_yc(x_0,y)+s\nabla_yc(x_1,y)\,,
\end{equation}
with $\rmx(0)=x_0$ and $\rmx(1)=x_1$.
In the remaining part of this section $c$ will always be assumed to be $C^1$ and twisted, so that c-segments, when they exist, are unique.
We define then the collection of maps $\bar\Lambda_s:X\times X\times Y\to X$, for $s\in[0,1]$, as
\[
(x_0,x_1,y)\mapsto \bar\Lambda_s(x_0,x_1,y)=
\begin{cases}
    \rmx(s) &\text{if the c-segment $(\rmx,y)$ exists for $(x_0,x_1,y)$} \\
    x_0 &\text{otherwise.}
\end{cases}
\]
In practice, $\bar\Lambda$ maps triplets in $(x_0,x_1,y)$ to the evaluation at time $s$ of a corresponding c-segment when this exists, and the map is extended to the whole $X\times X\times Y$ by assigning arbitrarily the first point $x_0$ (but any other measurable extension could be chosen).
By similar arguments to Lemma \ref{lemma:existence-lifted-v2}, one can show that $\bar\Lambda_s$ are measurable maps.

\begin{lemma}\label{lem:Tc-var-c-seg_nuac}
    Let $X,Y$ be two compact subsets of $\Rd$, with $Y=\overline{\operatorname{int}Y}$, and let $c\in C^1(X\times Y)$ be a twisted cost.
    Consider $\mu_0,\mu_1\in \Pc(X)$ and an absolutely continuous measure $\nu\in\Pc(Y)$, and assume furthermore that
    $\supp(\nu)\subset \operatorname{int}Y$.
    Let $\mu$ be an \snccconv{}-variational c-segment (respectively, an \smtw{}-variational c-segment) from $(\mu_0,\nu)$ to $(\mu_1,\nu)$ and $\gamma_\nu$ a 3-plan as in point \ref*{enum:lcs-pi} of Definition \ref{def:lifted_csegment}.
    Then, for $\gamma_\nu$-almost every $(x_0,x_1,y)\in X\times X\times Y$, $s\mapsto (\bar\Lambda_s(x_0,x_1,y),y)$ is a c-segment and $\mu(s)=(\bar\Lambda_s)_\# \gamma_\nu$ (respectively, there exists a continuous function $f\colon[0,1]\to[0,1]$, with $f(0) = 0$ and $f(1) = 1$, such that for $\gamma_\nu$-almost every $(x_0,x_1,y)\in X\times X\times Y$, $s\mapsto (\bar\Lambda_s(x_0,x_1,y),y)$ is a reparametrization via $f$ of a c-segment and $\mu(s)=(\bar\Lambda_{f(s)})_\# \gamma_\nu$).
\end{lemma}
\begin{proof} 
    For $\varepsilon>0$ (small enough) consider a curve $\sigma:(-\varepsilon,\varepsilon)\ni t \mapsto \sigma(t)\in\Pc(Y)$ defined as
    \[
    \sigma(t)=X(t)_\# \nu\,, \quad \text{with $\dot{X}(t)=\xi (X(t))$ and $X(0)=\Id$}\,,
    \]
    where $\xi\in C_c^\infty(\operatorname{int}Y)$ is a smooth vector field. We can write (see \cite[Corollary 10.2.7]{ambrosio2005gradient}, which can be easily adapted to our setting):
    \[
    \Tc_c(\mu(s),\nu)-\Tc_c(\mu(s),\sigma(t))=-t \int_{X\times Y} (\nabla_y c(x_s,y) \cdot \xi( y))\,d \pi(s) +o(t)\,,
    \]
    with $\pi(s)$ optimal transport plan between $\mu(s)$ and $\nu$.
    Since $\mu(s)$ is a \snccconv{}-variational c-segment, hence in particular a variational c-segment, by repeating the argument of \Cref{lemma:auto-csegment} we obtain
    \[
    \int_{X\times Y} (\nabla_y c(x_s,y) \cdot \xi(y))\,d \pi(s)=(1-s) \int_{X\times Y} (\nabla_y c(x_0,y) \cdot \xi(y))\,d \pi_0+s\int_{X\times Y} (\nabla_y c(x_1,y) \cdot \xi(y))\,d \pi_1\,.
    \]
    Absolute continuity of $\nu$ and the twist assumption on the cost $c$ imply that there exist $\nu$-a.e. unique optimal transport maps $T_0$ and $T_1$ from $\nu$ to $\mu_0$ and from $\nu$ to $\mu_1$, and a $\nu$-a.e. unique optimal transport map $T_s$ from $\nu$ to $\mu(s)$ (see \cite[Chapter 2]{villani2003topics}).
    The above equality can be rewritten as
    \[
    \int_{Y} \left[\nabla_y c(T_s(y),y) -(1-s) \nabla_y c(T_0(y),y) -s\nabla_y c(T_1(y),y)\right] \cdot \xi(y)\,d \nu(y)=0\,,
    \]
    for any $\xi\in C_c^\infty(\operatorname{int}Y)$. 
    By compactness of $X$ and $Y$, and since $C_c^\infty(\operatorname{int}Y)$ is dense in $L^2_\nu(Y)$, the equality holds also in $L^2_\nu(Y)$. 
    Then, using that $\xi$ is arbitrary, we deduce that $\nu$-a.e.
    \[
    \nabla_y c(T_s(y),y) =(1-s) \nabla_y c(T_0(y),y) +s\nabla_y c(T_1(y),y)\,, \quad \text{for $s\in[0,1]$.}
    \]
    Therefore, $\nu$-a.e. $T_s(y)$ is a c-segment from $T_0(y)$ to $T_1(y)$. From the regularity of the cost, these are unique and continuous. Finally,
    $\mu(s)=(\bar\Lambda_s)_\# \gamma_\nu$, where $\gamma_\nu=[y\mapsto(T_0(y),T_1(y),y)]_\# \nu$. 

    For the second part of the statement, the reasoning is similar. Assuming now that $\mu(s)$ is a \smtw{}-variational c-segment and considering variations $\sigma(t)$ of the same type as above, we deduce that
    \begin{equation}\label{eq:MTW_c-seg_max}
    \begin{aligned}
    \int_{Y} \nabla_y c(T_s(y),y)\cdot \xi(y)\,d \nu(y) &\le\max\left( \int_{Y} \nabla_y c(T_0(y),y) \cdot \xi(y)\,d \nu(y),\int_{Y} \nabla_y c(T_1(y),y) \cdot \xi(y) \,d \nu(y)\right)\\
    &\le \int_{Y} \max\left(\nabla_y c(T_0(y),y) \cdot \xi(y),\nabla_y c(T_1(y),y) \cdot \xi(y) \right) \,d \nu(y)\,,
    \end{aligned}
    \end{equation}
    for any $\xi\in C_c^\infty(\operatorname{int}Y)$. Again, the inequality actually holds for any $\xi\in L^2_\nu(Y)$ and is therefore true $\nu$-a.e.
    Following the arguments of \Cref{lem:auto-csegment-mtw} and using that $\xi$ is arbitrary, we obtain that $\nu$-a.e.
    \begin{equation}\label{eq:MTW_c-seg_param_y}
    \nabla_y c(T_s(y),y) =(1-f(y,s)) \nabla_y c(T_0(y),y) +f(y,s)\nabla_y c(T_1(y),y)\,, \quad \text{for all $s\in[0,1]$},
    \end{equation}
    where $f\colon Y\times[0,1] \to[0,1]$ is a family of continuous functions, measurable in $y\in\supp(\nu)$, such that $f(\cdot,0) = 0$ and $f(\cdot,1)= 1$. However, one can show that the parametrization of the c-segments has to be independent of $y$.
    The previous inequality can be rewritten
    \begin{multline}
        \int_{Y} f(y,s)(\nabla_y c(T_1(y),y) - \nabla_y c(T_0(y),y)) \cdot \xi(y)\,d \nu(y) \\ \le\max\left( 0 ,\int_{Y} \left(\nabla_y c(T_1(y),y) - \nabla_y c(T_0(y),y)\right) \cdot \xi(y) \,d \nu(y)\right).
    \end{multline}
    Writing $g(y) \coloneqq \nabla_y c(T_1(y),y) - \nabla_y c(T_0(y),y)$, we  assume $\int_Y \| g(y)\|^2 d\nu(y) > 0$, otherwise any global parametrization satisfies the result. Then, we consider $\xi(y) = u(y)\frac{g(y)}{\| g \|^2_{\nu}}$ with a measurable bounded function $u: Y \to \R$. We obtain
    \begin{equation}
        \int_Y u(y) f(y,s) \frac{\lvert g(y) \rvert^2}{\|g\|_\nu^2} d\nu(y) \leq \max\left(0,\int_Y u(y) \frac{\lvert g(y)\rvert^2}{\|g\|_\nu^2} d\nu(y)  \right),
    \end{equation}
    for all $u$. Considering $u$ that has $0$ mean w.r.t. the measure $\tilde{\nu} \coloneqq \frac{\lvert g(y)\rvert^2}{\|g\|_\nu^2} d\nu(y)$, it implies that $f(y,s)$ is $\tilde \nu$ a.e. constant in $y$.
    When $g(y) = 0$, every parametrization is admissible, in particular the one found for all the other points. 
\end{proof}

As an immediate consequence of the previous lemma, in this smooth setting, whenever \snccconv{} or \smtw{} variational c-segments in the Wasserstein space exist for any triplet $(\mu_0,\mu_1,\nu)$, c-segments on the base space are necessarily well-defined.

\begin{corollary}\label{cor:c-segments-existence}
    Let $X,Y$ be compact subsets of $\Rd$, with $Y=\overline{\operatorname{int}Y}$, and let $c\in C^1(X\times Y)$ be a twisted cost. If $(\Pc(X)\times\Pc(Y),\Tc_c)$ satisfies the \snccconv{} condition (or the \smtw{} condition), then c-segments on the base space $(X\times Y,c)$ exist and are unique for any triplets of points $(x_0,x_1,y)\in X\times X\times \operatorname{int}Y$.
\end{corollary}
\begin{proof}
    Take any triplet $(x_0,x_1,y)\in X\times X\times \operatorname{int}Y$. Consider the measures $\mu_0=\delta_{x_0}, \mu_a=\delta_{x_1}$ and an absolutely continuous measure $\nu\in\Pc(Y)$, such that $y\in \supp(\nu)$ and $\supp(\nu)\subset \operatorname{int}Y$. From Lemma \ref{lem:Tc-var-c-seg_nuac} we know that for a.e. $y\in\supp(\nu)$ there exists a unique c-segment from $(x_0,y)$ to $(x_1,y)$. Consider then a sequence $y_k$ converging to $y$ such that, for any $k$, there exists a unique c-segment $\rmx_k$ from $(x_0,y_k)$ to $(x_1,y_k)$. For any $s\in[0,1]$, since $X$ is compact and $c\in C^1(X\times Y)$, we can extract a converging subsequence from $\rmx_k(s)$ converging to a point denoted $\rmx(s)$ and pass to the limit in \eqref{eq:c-segment-recall} to obtain
    \[
    \nabla_y c(\rmx(s),y) =(1-s) \nabla_y c(x_0,y) +s\nabla_y c(x_1,y)\,.
    \]
    Since the cost is twisted, $\rmx(s)$ solution to the previous equation is unique and the whole sequence converges. Hence, $\rmx_k$ converges pointwise to $\rmx$ c-segment from $(x_0,y)$ to $(x_1,y)$. The regularity of the cost implies the continuity of the curve.
\end{proof}

To prove Proposition \ref{ThMainTheorem_converse}, we now extend the previous lemma to atomic measures by a density argument.

\begin{lemma}\label{lem:Tc-var-c-seg}
    In the same setting of Lemma \ref{lem:Tc-var-c-seg_nuac}, consider now a general measure $\nu\in\Pc(Y)$ (not necessarily absolutely continuous) with $\supp(\nu)\subset \operatorname{int}Y$. Assume further that the Wasserstein space $(\Pc(X)\times\Pc(Y),\Tc_c)$ satisfies the \snccconv{} condition (respectively, \smtw{} condition).
    Then, curves of the form $\mu(s)=(\bar\Lambda_s)_\# \gamma_\nu$ always provide \snccconv{}-variational c-segments on $(\Pc(X)\times\Pc(Y),\Tc_c)$ (respectively, \smtw{}-variational c-segments).
\end{lemma}
\begin{proof}
    Let $\nu_k$ be a sequence of absolutely continuous measures, with $\supp(\nu_k)\subset \operatorname{int}Y$, converging weakly with respect to the narrow convergence of measures to $\nu$, as $k\rightarrow +\infty$.
    Since $(\Pc(X)\times\Pc(Y),\Tc_c)$ satisfies the \snccconv{} condition we can find a corresponding sequence of \snccconv{}-variational c-segments $\mu_k$.
    Moreover, from Corollary \ref{cor:c-segments-existence}, the collection of measurable maps $\bar\Lambda_s:X\times X\times Y\to X$, for $s\in[0,1]$, is such that $s\mapsto (\bar\Lambda_s(x_0,x_1,y),y)$ is a c-segment for each $(x_0,x_1,y)\in X\times X\times \operatorname{int}Y$ and can be easily shown to be continuous, from the compactness of $X$ and the regularity of $c$.
    For any $k$, $\mu_k(s)$ is uniquely defined as $\mu_k(s)=(\bar\Lambda_s)_\# \gamma_{\nu_k}$ from Lemma \ref{lem:Tc-var-c-seg_nuac}. Since $\nu_k\rightharpoonup \nu$, by the stability of optimal transport plans \cite[Theorem 5.20]{villani2009optimal}, there exist optimal plans $\pi_0$ and $\pi_1$ between $\mu_0$ and $\nu$ and between $\mu_1$ and $\nu$, respectively, and a measure $\gamma_\nu$ as in Definition \ref{def:lifted_csegment}, such that (up to the extraction of a subsequence), $(\pi_0)_k=(\p_1,\p_3)_\#\gamma_{\nu_k}\rightharpoonup\pi_0$, $(\pi_1)_k=(\p_2,\p_3)_\#\gamma_{\nu_k}\rightharpoonup\pi_1$ and $\gamma_{\nu_k}\rightharpoonup\gamma_\nu$.
    Hence $\mu_k(s)=(\bar\Lambda_s)_\# \gamma_{\nu_k} \rightharpoonup (\bar\Lambda_s)_\# \gamma_\nu=\mu(s)$ (up to the extraction of a subsequence). Since $X$ and $Y$ are compact, $\Tc_c(\cdot,\cdot)$ is continuous with respect to the narrow convergence. Hence, the convex function $\Tc(\mu_k(s),\nu_k)-\Tc(\mu_k(s),\sigma)$ converges pointwise to $\Tc(\mu(s),\nu)-\Tc(\mu(s),\sigma)$, for any $\sigma\in\Pc(Y)$. By pointwise convergence of convex functions, we deduce that $\mu(s)$ is a \snccconv{}-variational c-segment.

    The proof of the second part of the statement is identical. We highlight that for each $k$, the curve $\mu_k$ may have a different parametrization, via a continuous function $f_k$, as explained in Lemma \ref{lem:Tc-var-c-seg_nuac}.  However, condition \smtw{} is independent of such a continuous parametrization, i.e. $\mu_k(s)$ satisfies \smtw{} if and only if so does $\mu_k(f(s))$, for any continuous function $f:[0,1]\rightarrow[0,1]$ with $f(0)=0$ and $f(1)=1$. This implies that one may fix a common parametrization $f$ for any $\mu_k$, as for example the linear one, and pass to the limit.
\end{proof}

Relying on the previous lemma, Proposition \ref{ThMainTheorem_converse} can now be proven straightforwardly by considering atomic measures. 

\begin{proof}[Proof of Proposition \ref{ThMainTheorem_converse}]
    We treat the case of the \snccconv{} condition, the proof being similar for the \smtw{} condition.
    Consider any $\ov y\in\operatorname{int}Y$, $y\in Y$ and $x_0,x_1\in X$. Take the measures $\mu_0=\delta_{x_0}, \mu_1=\delta_{x_1}\in \Pc(X)$ and $\sigma=\delta_y,\nu=\delta_{\ov{y}}\in\Pc(Y)$. Then, there exists at least one \snccconv{}-variational c-segment on $(\Pc(X)\times\Pc(Y),\Tc_c)$ of the form $\mu(s)=\delta_{\rmx(s)}$, where $\rmx$ is the c-segment from $(x_0,\ov y)$ to $(x_1,\ov y)$.
    In particular, the function $s\mapsto \Tc(\mu(s),\nu)-\Tc(\mu(s),\sigma)=c(\rmx(s),\yb)-c(\rmx(s),y)$ is convex for any $y\in Y$ and $(\rmx,\yb)$ is therefore an \snccconv{}-variational c-segment on $(X\times Y,c)$.
    
    The general case $\ov{y}\in Y$ can be tackled by defining $\ov{y}_k\rightarrow \ov{y}$, $(\ov{y}_k)_k\subset \operatorname{int}Y$ and considering as above a sequence of \snccconv{}-variational c-segments $\rmx_k$, from $(x_0,\yb_k)$ to $(x_1,\yb_k)$.
    For any $k$, $s\mapsto c(\rmx_k(s),\ov{y}_k)-c(\rmx_k(s),y)$ is a convex function for any $y\in Y$.
    By compactness of $X$, for any $s\in[0,1]$, we can extract a converging subsequence (without relabeling it) and define the limit curve $\rmx(s)=\lim_k \rmx_k(s)$. Since $c$ is continuous, $c(\rmx_k(s),\ov{y}_k)-c(\rmx_k(s),y)$ converges pointwise to $c(\rmx(s),\ov{y}_k)-c(\rmx(s),y)$, which is therefore convex, for any $y\in Y$, and $\rmx$ is an \snccconv{}-variational c-segment.
    
\end{proof}

\subsection{Other proofs}
\label{AppendixOtherProofs}
\begin{proof}[Proof of \Cref{lem:auto-csegment-mtw}]
    Using a similar argument as in the proof of Lemma \ref{lemma:auto-csegment}, we 
    obtain
    \begin{equation}\label{EqSMTWReducedEquation}
        \langle \nabla_y c(\rmx(s),\overline{y}),\eta\rangle \leq \max(\langle \nabla_y c(\rmx(0),\overline{y}),\eta\rangle,\langle \nabla_y c(\rmx(1),\overline{y}),\eta\rangle)\,,
    \end{equation}
    for all $\eta \in T_{\overline{y}}Y$.
    This inequality implies that $\nabla_y c(\rmx(s),\overline{y}) = \alpha \nabla_y c(\rmx(0),\overline{y}) + \beta \nabla_y c(\rmx(1),\overline{y})$ for some real coefficients $\alpha,\beta$,  by testing this inequality with $\pm \eta $ orthogonal to the span of the two vectors. The inequality reduces to a two-dimensional problem. The first case is when $\nabla_y c(\rmx(0),\overline{y}) = \nabla_y c(\rmx(1),\overline{y})$, then any candidate $f$ satisfies our statement, in particular $f(s) = s$. 
     
    The second case is when the two vectors on the right-hand side of \eqref{EqSMTWReducedEquation} span a two-dimensional space.
    Testing the inequality with a vector such that 
    $\langle \nabla_y c(\rmx(0),\overline{y}),\eta \rangle = \langle \nabla_y c(\rmx(1),\overline{y}), \eta \rangle = 1$, we get $\alpha + \beta \leq 1$ and testing with the opposite, we get $\alpha  + \beta \geq 1$. A similar reasoning with a vector $\eta$ orthogonal $\nabla_y c(x(0),\yb)$ and such that $\langle \nabla_y c(x(0),\yb),\eta\rangle = -1$ gives that $0\leq \beta $. Similarly, $\alpha \geq 0$. 
    Since the coefficient $f(s)$ is unique and the curve $x(s)$ is continuous, $s \mapsto f(s)$ is also continuous, which is the result.
    
    The last case when $\nabla_y c(\rmx(0),\overline{y})$ and $\nabla_y c(\rmx(1),\overline{y})$ are collinear but different is similar.
\end{proof}

\begin{proof}[Proof of \Cref{lemma:subdiff-connected}]
    We show that \ref*{enum:LMP-csubdiff-i} implies \ref*{enum:LMP-csubdiff-iii} and that \ref*{enum:LMP-csubdiff-ii} implies \ref*{enum:LMP-csubdiff-i}.

    \ref*{enum:LMP-csubdiff-i}$\implies$\ref*{enum:LMP-csubdiff-iii}. Let $\phi\colon Y\to [-\infty,+\infty]$ be such that $\phi(\yb)$ is finite and 
    \[
    \forall y\in Y,\quad\left\{\begin{aligned}
    c(x_0,\yb)-c(x_0,y)&\leq \phi(\yb)-\phi(y)\\
    c(x_1,\yb)-c(x_1,y)&\leq \phi(\yb)-\phi(y)\,.
    \end{aligned}\right.\]
    The above inequalities may involve infinities but never any undefined combination. Then $\max\{c(x_0,\yb)-c(x_0,y), c(x_1,\yb)-c(x_1,y)\}\leq \phi(\yb)-\phi(y)$ which combined with \ref*{enum:LMP-csubdiff-i} yields the desired result.

    \ref*{enum:LMP-csubdiff-ii}$\implies$\ref*{enum:LMP-csubdiff-i}. Define the function 
    \[
    \psi(x)=\begin{cases}
        c(x_0,\yb)\quad&\text{if } x=x_0,\\
        c(x_1,\yb)\quad&\text{if } x=x_1,\\
        -\infty \quad&\text{otherwise.}        
    \end{cases}
    \]
    Let $\phi=\psi^c$, a c-concave function. Regardless of the infinities taken by the cost function we have 
    \[
    \phi(y)=\inf_{x\in\{x_0,x_1\}}c(x,y)-\psi(x)=\min\{c(x_0,y)-c(x_0,\yb), c(x_1,y)-c(x_1,\yb)\}.
    \]
    Since $\phi(\yb)=0$ and $\psi(x_0), \psi(x_1)$ are finite we find that $x_0\in\partial^c\phi(\yb)$ and $x_1\in \partial^c\phi(\yb)$. Then by \ref*{enum:LMP-csubdiff-ii} $\tilde x\in \partial^c\phi(\yb)$, which gives finiteness of $c(\tilde x,\yb)$ and
    \[
    \forall y\in Y,\quad c(\tilde x,\yb)-c(\tilde x,y)\leq \phi(\yb)-\phi(y)\,.
    \]
    Substituting with the value of $\phi(y)$ we obtain \ref*{enum:LMP-csubdiff-i}. 
\end{proof}

\begin{proof}[Proof of Lemma \ref{lemma:continuous-cost-measurable-NNCC}]
Let $\Gamma\subset X\times X\times Y$ be defined as in \eqref{eq:maximal-Gamma}.
The NNCC property guarantees the existence of a variational c-segment for each triplet $(x_0,x_1,y) \in \Gamma$. For each $s\in(0,1)$ we aim to construct a measurable map $\Lambda_s$ mapping $\gamma$-a.e. triplet to the evaluation at time $s$ of an associated variational c-segment.

Fix $s\in (0,1)$ and consider the multi-valued map $ \tilde\Lambda_s\colon \Gamma \rightrightarrows X$ which maps a given triplet $(x_0,x_1,y) \in \Gamma$ to the (nonempty by hypothesis) set of points that satisfy the NNCC inequality \eqref{eq:NNCC-inequality}.
By \cite[Theorem 6.9.13]{bogachev2007measure}, 
it is sufficient that the graph of $\tilde \Lambda_s$ is measurable to prove the existence of a measurable selection $\Lambda_s$.
Consider the two transport plans $\gamma^1,\gamma^2\in\Pc(X\times Y)$ defined as $\gamma^i=(\p_i,\p_3)_\#\gamma$, for $i=1,2$.
Since $c$ is lower semi-continuous and bounded from below (and taking finite values on $\Gamma$), it can be approximated by a non-decreasing sequence of continuous costs, denoted $c_n$, $n \in \mathbb N$.
By Egorov's theorem and the fact that $X\times Y$ is Polish, this pointwise convergence can be turned into a uniform convergence on a compact set $Z_k=Z^1_k\cap Z^2_k$, for $Z^1_k$ and $Z^2_k$ two compact sets of measure $\gamma^1(Z^1_k) \geq 1 - 1/k$ and $\gamma^2(Z^2_k) \geq 1 - 1/k$, for every $k \geq 1$ integer (see \cite[Proof of Theorem 5.10]{villani2009optimal}). For every $m \geq 1$, there exists $N(m)$ such that for $n\geq N(m)$, we have
\begin{equation}\label{eq:uniform}
    |c_n(x,y) - c(x,y)| \leq 1/m \quad \text{for $(x,y) \in Z_k$}\,.
\end{equation}

For $m \geq 1$ integer, consider the inequality
\begin{equation}\label{EqRelaxedNNCC}
   c_n(x,y)-c_n(x,z)\leq 4/m +  (1-s)[c_n(x_0,y)-c_n(x_0,z)]+s[c_n(x_1,y)-c_n(x_1,z)]\,.
\end{equation}
Define the closed sets  (since $c_n$ is continuous)
\begin{equation}
    F_{k,n,m} \coloneqq \{ (x_0,x_1,y,x) : \text{for } i= 0,1\, (x_i,y) \in Z_k \text{ , }(x,y) \in Z_k \text{ and }\eqref{EqRelaxedNNCC} \text{ is satisfied }  \forall (x,z) \in Z_k  \}\,.
\end{equation}
By \eqref{eq:uniform} and the triangle inequality, defining
\[
\tilde Z_k=\{(x_0,x_1,y,x)\in X\times X\times Y\times X : (x_0,y)\in Z_k \text{ and } (x_1,y)\in Z_k\},
\]
we have
$\operatorname{Graph}(\tilde \Lambda_s) \cap \tilde Z_k \subset F_{k,n,m}$
and $\cup_{k \geq 1} \operatorname{Graph}(\tilde \Lambda_s) \cap \tilde Z_k=\operatorname{Graph}(\tilde \Lambda_s)$, $\gamma$ a.e.
Finally, by pointwise convergence,
$ \operatorname{Graph}(\tilde \Lambda_s) \cap \tilde Z_k = \cap_{m \geq 1} (\cap_{n \geq N(m)} F_{k,n,m}) $, $\gamma$ a.e.
Then, there exists a measurable set $W$ such that $\gamma(W) = 0$ and we can change $\tilde \Lambda_s$ on $W$ by choosing the value $x_0$, for instance (in fact any measurable extension would work). Denoting
\[
\tilde W=\{ (x_0,x_1,y,x_0)\in X\times X \times Y\times X : (x_0,x_1,y)\in W \},
\]
we have the union $\operatorname{Graph}(\tilde \Lambda_s) = \tilde W \cup \big(\cup_{k\ge1}  \cap_{m \geq 1} (\cap_{n \geq N(m)} F_{k,n,m})\big) $.
Being a countable union and intersection of measurable sets, $\operatorname{Graph}(\tilde \Lambda_s)$ is itself measurable. Then, the selection theorem can be applied and it gives the result.
\end{proof}

\begin{remark}
    The previous proof is inspired by \cite[Proof of Theorem 5.10]{villani2009optimal} in which a similar measurability issue is treated.
\end{remark}

\begin{proof}[Proof of \Cref{lemma:coupling-extension}]
    Define $\tilde\pi=[(w\in \Gamma)\mapsto(\Lambda(w),w)]_\#\gamma$. Since this is a coupling of $(\mu,\gamma)$ there exists (e.g.\ by gluing $\pi$ and $\tilde\pi$ along $X$) a $3$-plan $\gamma^3\in\Pc(X\times Y\times \Gamma)$ such that 
    \begin{align}
        (\p_1,\p_2)_\#\gamma^3=\pi \label{eq:proof-lemma:coupling-extension-0}\\
        (\p_1,\p_3)_\#\gamma^3=\tilde\pi\label{eq:proof-lemma:coupling-extension-1}\,.
    \end{align}
    Note that $\tilde\pi$ is concentrated on the antigraph $\{(\Lambda(w),w) : w\in \Gamma\}$. By \eqref{eq:proof-lemma:coupling-extension-1}, $\gamma^3$ is then concentrated on the set $\Sigma=\{(x,y,w)\in X\times Y\times \Gamma : x=\Lambda(w)\}$. Define $\tilde\gamma=(\p_3,\p_2)_\#\gamma^3\in\Pi(\gamma,\sigma)$. Then using successively the definition of $\tilde\gamma$, the concentration of $\gamma^3$ on $\Sigma$, and \eqref{eq:proof-lemma:coupling-extension-0} we obtain
    \begin{align*}
        [(w,y)\mapsto(\Lambda(w),y)]_\#\tilde\gamma &=[(x,y,w)\mapsto(\Lambda(w),y)]_\#\gamma^3\\
        &=[(x,y,w)\mapsto(x,y)]_\#\gamma^3\\
        &=\pi\,.
    \end{align*}
\end{proof}

\begin{proof}[Proof of \Cref{ThLemmaReparametrization}] 
    We first use \cite[Lemma 7.12]{liero2018optimal} which says that given marginals $\mu, \nu$, there exist lifts $\tilde \mu,\tilde \nu_0$ realizing Equation \eqref{EqEqualityOnLifts} (replacing $\tilde \nu$ by $\tilde \nu_0$).
    However, the measure $\tilde \nu_0$ is not necessarily equal to $\tilde \nu$ which is fixed in our setting. In fact, \cite[Lemma 7.12]{liero2018optimal} shows that it is possible to fix the marginal $\tilde \nu_0$ to $\frac{1}{\nu(Y)}\nu(y)\otimes \delta_{\nu(X)}$, which we use in the rest of the proof. Denote $\tilde \pi_0$ an optimal plan on the cone between the measures $\tilde \nu_0$ and $\tilde \mu$. This optimal plan only charges $(y,\nu(X))$ due to the marginal constraint $\tilde \nu_0$, so that when $y$ is fixed, it defines a measure on $\mathcal{C}(X)$ by disintegration:
    \begin{equation}
      \tilde \pi_0((x,m),(y,n)) =  \tilde \pi_0^y((x,m)) \frac{\nu(y)}{\nu(X)}\otimes \delta_{\nu(X)}\,.
    \end{equation}
    We consider the plan $\tilde \pi$ defined by its disintegration with respect to $\tilde \nu$
    \begin{equation}
        \tilde \pi((x,m),(y,n')) =  r^{\nu(X),n'}_{\#}(\tilde \pi_0^y) \tilde \nu(y,n')\,,
    \end{equation}
    where $r^{n_0,n'}((x,m)) \coloneqq (x,mn'/n_0)$ for $n_0 \neq 0$, which is a map on $\mathcal{C}(X)$. 
    By definition, the marginal of $\tilde \pi$ on $(y,n)$ is $\tilde \nu$. 
    We consider a function $f: \mathcal{C}(X) \times \mathcal{C}(X) \to \R$ that is one-homogeneous with respect to the couple of radial variables $(m,n)$. One has 
    \begin{align}
       \int f((x,m),(y,n))\tilde \pi((x,m),(y,n))  &= \int f((x,mn/\nu(X)),(y,n)) \tilde \pi_0^y((x,m)) \tilde \nu(y,n) \\
        & = \int (1/\nu(X))f((x,m),(y,\nu(X))) \tilde \pi_0^y((x,m)) \int_n n\nu(y,n)\\
        & = \int (1/\nu(Y))f((x,m),(y,\nu(X))) \tilde \pi_0^y((x,m)) \nu(y)\\
        & = \int f((x,m),(y,m)) \tilde \pi_0((x,m),(y,n))\,.
    \end{align}
    It implies that the cost of the plan is preserved since the cost on the product of cones is one-homogeneous. Moreover, the marginal constraint on $(x,m)$ can be tested by functions of the type $mg(x)$ where $g: X \to \R$ 
    which is also homogeneous with respect to $(m,n)$.
\end{proof}

\section*{Acknowledgements}
We wish to thank T. Gallou\"et, Y.-H. Kim, H. Lavenant, C. Léonard and G. Loeper for helpful discussions.
FXV and GT would like to thank the Mathematisches Forschungsinstitut Oberwolfach for hosting in 2024 the workshop 2406, during which this work has advanced significantly. 
The work of FXV and GT is supported by the Bézout Labex (New Monge Problems), funded by ANR, reference ANR-10-LABX-58.

\bibliographystyle{siam}
\bibliography{refs}

\end{document}